\documentclass[a4paper,10pt,reqno]{amsart}

\title{Remarks on modulation spaces and homogeneity}

\usepackage{amssymb}
\usepackage{latexsym}
\usepackage{amsmath}
\usepackage{mathrsfs}
\usepackage{euscript}
\usepackage{amsthm}
\usepackage{upgreek}
\usepackage{cite}
\usepackage{xcolor}
\usepackage{calc}                         

\newcommand{\scal}[2]{\langle #1,#2\rangle}

\newcommand{\re}{\mathbf R}
\newcommand{\rr}[1]{\mathbf R^{#1}}
\newcommand{\zz}[1]{\mathbf Z^{#1}}
\newcommand{\zzp}[1]{\mathbf Z_+^{#1}}
\newcommand{\cc}[1]{\mathbf C^{#1}}
\newcommand{\nn}[1]{\mathbf N^{#1}}

\newcommand{\nm}[2]{\Vert #1\Vert _{#2}}
\newcommand{\nmm}[1]{\Vert #1\Vert }

\newcommand{\op}{\operatorname{Op}}

\newcommand{\sets}[2]{\{ \, #1\, ;\, #2\, \} }
\newcommand{\Sets}[2]{\left \{ \, #1\, ;\, #2\, \right \} }
\newcommand{\ep}{\varepsilon}
\newcommand{\fy}{\varphi}

\newcommand{\cdo}{\, \cdot \, }

\newcommand{\tp}{\operatorname{Tp}}
\newcommand{\vrum}{\vspace{0.1cm}}

\newcommand{\wpr}{{\text{\footnotesize $\#$}}}

\newcommand{\be}{\mathbf{e}}
\newcommand{\beps}{\mathbf{\varepsilon}}

\newcommand{\tK}{\widetilde{K}}

\newcommand{\GL}{\mathbf{M}}

\newcommand{\maclK}{\mathcal K}
\newcommand{\maclP}{\mathcal P}
\newcommand{\maclS}{\mathcal S}

\newcommand{\mascB}{\mathscr B}

\newcommand{\mascF}{\mathscr F}

\newcommand{\mascP}{\mathscr P}
\newcommand{\mascS}{\mathscr S}

\newcommand{\mabfp}{{\boldsymbol p}}
\newcommand{\mabfq}{\boldsymbol q}

\newcommand{\splM}{\EuScript M}

\newcommand{\Srr}{\maclS^\sigma_s(\rr{d})}

\newcommand{\topo}{\tp _\phi (\omega )}
\newcommand{\psdo}{pseudo-differential operator}
\def\inv{^{-1}}

\numberwithin{equation}{section}          
\newtheorem{thm}{Theorem}
\numberwithin{thm}{section}
\newtheorem*{tom}{\rubrik}
\newcommand{\rubrik}{}
\newtheorem{prop}[thm]{Proposition}
\newtheorem{cor}[thm]{Corollary}
\newtheorem{lemma}[thm]{Lemma}

\theoremstyle{definition}

\newtheorem{defn}[thm]{Definition}
\newtheorem{example}[thm]{Example}

\theoremstyle{remark}

\newtheorem{rem}[thm]{Remark}              

\author{Ahmed Abdeljawad}

\address{Dipartimento di Matematica ``G. Peano'', Universit\'a
degli Studi di Torino}

\email{ahmed.abdeljawad@unito.it}

\author{Sandro Coriasco}

\address{Dipartimento di Matematica ``G. Peano'', Universit\'a
degli Studi di Torino}

\email{sandro.coriasco@unito.it}

\author{Joachim Toft}

\address{Department of Mathematics,
Linn{\ae}us University, Sweden}

\email{joachim.toft@lnu.se}

\begin{document}

\par

\title[Liftings for modulation spaces, and
one-parameter groups of $\Psi$DO]{Liftings
for ultra-modulation spaces, and one-parameter
groups of Gevrey type pseudo-differential
operators}



\par

\begin{abstract}
We deduce one-parameter group properties for
pseudo-differential operators $\op (a)$, where
$a$ belongs to the class $\Gamma ^{(\omega _0)}_*$
of certain Gevrey symbols. We use this to show that there
are pseudo-differential operators $\op (a)$ and $\op (b)$
which are inverses to each others, where
$a\in \Gamma ^{(\omega _0)}_*$
and $b\in \Gamma ^{(1/\omega _0)}_*$.

\par

We apply these results to deduce lifting property for
modulation spaces and
construct explicit isomorpisms between them. For each weight
functions $\omega ,\omega _0$ moderated by GRS submultiplicative
weights, we prove that the Toeplitz operator
(or localization operator) $\tp (\omega _0)$ is an isomorphism
from $M^{p,q}_{(\omega )}$ onto $M^{p,q}_{(\omega /\omega _0)}$
for every $p,q \in (0,\infty ]$.
\end{abstract}

\maketitle

\section{Introduction}\label{sec0}

\par

The topological vector spaces $V_1$ and $V_2$ is said to possess
lifting property if there exists a "convenient" homeomorphisms
(a lifting) between them. For example, for any weight $\omega$
on $\rr d$, $p\in (0,\infty]$ and $s\in \mathbf R$ the
convenient mappings
$f\mapsto \omega \cdot f$ and $f\mapsto (1-\Delta )^{s/2}f$
are homeomorphic from the (weighted) Lebesgue space $L^p_{(\omega )}$ and
the Sobolev space $H^p_s$, respectively, into $L^p=H^p_0$,
with inverses $f\mapsto \omega ^{-1}\cdot f$ and
$f\mapsto (1-\Delta )^{-s/2}f$, respectively.
(Cf. \cite{Ho1} and Section \ref{sec1} for notations.)
Hence, these spaces possess lifting properties.

\par

It is often uncomplicated to deduce lifting properties
between (quasi-)Banach spaces of functions and distributions, 
if the definition of their norms only differs by a multiplicative weight
on the involved distributions, or on their Fourier
transforms, which is the case in the previous homeomorphisms. Here
recall that multiplications on the Fourier transform side are linked
to questions on differentiation of the involved elements.
A more complicated situation appear when there are some kind of
interactions between multiplication and differentiation in the
definition of the involved vector spaces.

\par 

An example where such interactions occur concerns the extended
family of Sobolev spaces, introduced by Bony and Chemin in
\cite{BoCh} (see also \cite{Le2}). More precisely, let
$\omega ,\omega _0$ be suitable weight functions and $g$
a suitable Riemannian metric, which are defined on the phase
space $W\simeq T^*\rr d \simeq \rr {2d}$. Then Bony and
Chemin introduced in \cite{BoCh} the
generalised Sobolev space $H(\omega ,g)$ which fits the
H{\"o}rmander-Weyl calculus well in the sense that $H(1,g)=L^2$,
and if $a$ belongs to the H{\"o}rmander class $S(\omega _0,g)$,
then Weyl operator $\op ^w(a)$ with symbol $a$ is continuous from
$H(\omega _0\omega ,g)$ to $H(\omega ,g)$. Moreover, they deduced
group algebras, from which it follows that to each such weight
$\omega _0$, there exist symbols $a$ and $b$ such that
\begin{equation}\label{Eq:IntrInverses1}
\op ^w(a)\circ \op ^w(b) = \op ^w(b)\circ \op ^w(a) = I ,
\quad a\in S(\omega _0,g),\ b\in S(1/\omega _0,g).
\end{equation}
Here $I$ is the identity operator on $\mascS '$. In particular,
by the continuity properties of $\op ^w(a)$ it follows that 
$H(\omega _0\omega ,g)$ and $H(\omega ,g)$ possess lifting properties
with the homeomorphism $\op ^w(a)$, and with $\op ^w(b)$ as its inverse.

\par

The existence of $a$ and $b$ in \eqref{Eq:IntrInverses1} is a consequence
of solution properties of the evolution equation
\begin{equation}\label{Eq:IntrSymeqDiff}
(\partial_t a)(t,\cdo)=(b+\log \vartheta )\wpr a(t, \cdo),
\qquad a(0,\cdo)=a_0\in S(\omega ,g),\ \vartheta \in S(\vartheta ,g),
\end{equation}
which involve the Weyl product $\wpr$ and a fixed element $b\in S(1,g)$.
It is proved that \eqref{Eq:IntrSymeqDiff} has a unique solution
$a(t,\cdo )$ which belongs to $S(\omega\vartheta^t ,g)$
(cf. \cite[Theorem 6.4]{BoCh} or \cite[Theorem 2.6.15]{Le2}). The existence
of $a$ and $b$ in \eqref{Eq:IntrInverses1}
will follow by choosing $\omega =a_0=1$, $t=1$ and $\vartheta =\omega _0$.

\par

If $g$ is the constant euclidean metric on the phase
space $\rr {2d}$, then $S(\omega _0,g)$ equals $S^{(\omega _0)}(\rr {2d})$,
the set of all smooth symbols $a$ which satisfies $|\partial ^\alpha a|
\lesssim \omega _0$.
We notice that also for such simple choices of $g$,
\eqref{Eq:IntrInverses1} above leads to lifting
properties that are not trivial. In fact, let $\omega$ and $\omega _0$ be
polynomially moderate weight on the phase space, and let $\mascB$ be
a suitable translation invariant BF-space. Then it is observed in
\cite{GrTo1} that the continuity results
for pseudo-differential operators on modulation spaces in
\cite{To7,To9} imply that
$\op ^w(a)$ in \eqref{Eq:IntrInverses1}
is continuous and bijective from $M(\omega _0\omega ,\mascB)$
to $M(\omega ,\mascB)$ with continuous inverse $\op ^w(b)$. In particular,
by choosing $\mascB$ to be the mixed norm space $L^{p,q}(\rr {2d})$ of
Lebesgue type, then $M(\omega ,\mascB )$ is equal to the
classical modulation space $M^{p,q}_{(\omega )}$, introduced by
Feichtinger in \cite{Fe4}. Consequently, $\op ^w(a)$ above lifts
$M^{p,q}_{(\omega _0\omega )}$ into $M^{p,q}_{(\omega )}$.

\par

An important class of operators in quantum mechanics and time-frequency
analysis concerns Toeplitz, or localisation operators. The main issue
in \cite{GrTo1,GrTo2} is to show that the Toeplitz operator $\tp (\omega _0)$
lifts $M^{p,q}_{(\omega _0\omega )}$ into $M^{p,q}_{(\omega )}$
for suitable $\omega _0$. The assumptions on $\omega _0$ in \cite{GrTo1}
is that it should be polynomially moderate and satisfies
$\omega _0\in S^{(\omega _0)}$. In \cite{GrTo2}, it is only assumed that
$\omega _0$ is moderated by a GRS weight, but instead it is here
required that $\omega _0$ is \emph{radial in each phase shift}, i.{\,}e.
$\omega _0$ should satisfy
$$
\omega _0(x_1,\dots ,x_d,\xi _1,\dots ,\xi _d) = \vartheta (r_1,\dots ,r_d),
\qquad r_j = |(x_j,\xi _j)|,
$$
for some weight $\vartheta$.

\par

The approaches in \cite{GrTo1,GrTo2} are also different. In \cite{GrTo2},
the lifting properties for $\tp (\omega _0)$ is reached by using the links between
modulation spaces and Bargmann-Foch spaces in combination of suitable
estimates for a sort of generalised gamma-functions. The approach in
\cite{GrTo1} relies on corresponding lifting properties for pseudo-differential
operators, as follows:
\begin{enumerate}
\item $\tp (\omega _0)=\op ^w(c)$ for some $c\in S^{(\omega _0)}$;

\vrum

\item by the definitions, it follows by straightforward computations that
if $\vartheta =\omega _0^{\frac 12}$, then $\tp (\omega _0)$ is a homeomorphism
from $M^{p,q}_{(\vartheta )}$ to $M^{p,q}_{(\vartheta )}$;

\vrum

\item combining \eqref{Eq:IntrInverses1} with Wiener's lemma for
$(S^{(1)},\wpr )$ ensures that the inverse of $\tp (\omega _0)$
in (2) is a pseudo-differential operator $\op ^w(b)$ with the symbol
$b$ in $S^{(1/\omega _0)}$;

\vrum

\item by (1), (3) and duality,
$$
T_1 \equiv \op ^w(b)\circ \tp (\omega _0)
\quad \text{and}\quad
T_2\equiv \tp (\omega _0)\circ \op ^w(b)
$$
are both the identity operator on
$\mascS '(\rr d)$, since $T_1$ is the identity operator on
$M^{p,q}_{(\vartheta )}$, $T_2$ is the identity operator on
$M^{p,q}_{(1/\vartheta )}$, and
$\mascS \subseteq M^{p,q}_{(\vartheta )}\cap M^{p,q}_{(1/\vartheta )}$.

\vrum

\item by (4), $T_1=T_2=\op ^w(1)$ is the identity operator on each
$M^{p,q}_{(\omega )}$. Since
$$
\tp (\omega _0) = \op ^w(c)\, : \, M^{p,q}_{(\omega _0\omega )}
\to M^{p,q}_{(\omega )}
\quad \text{and}\quad
\op ^w(b)\, : \, M^{p,q}_{(\omega )}
\to M^{p,q}_{(\omega _0\omega )}
$$
are continuous (cf. \cite{To7,To9}) and inverses to each other,
it follows that they are homeomorphisms.
\end{enumerate}

\medspace

In the first part of the paper we deduce an analog of
\eqref{Eq:IntrInverses1} for the
Gevrey type symbol classes $\Gamma ^{(\omega _0)}_{s}$ and
$\Gamma ^{(\omega _0)}_{0,s}$ of orders $s\ge 1$, the set of all
$a\in C^\infty$ such that
$$
|\partial ^\alpha a (X)| \lesssim h^{|\alpha |}\alpha !^s\omega (X)
$$
for some $h>0$, respectively for every $h>0$, considered in \cite{CaTo}.
That is, in Section \ref{sec3} we show that there exist symbols $a$ and
$b$ such that
\begin{equation}\label{Eq:IntrInverses2}
\op ^w(a)\circ \op ^w(b) = \op ^w(b)\circ \op ^w(a) = I ,
\quad a\in \Gamma ^{(\omega _0}_{s},\ b\in \Gamma ^{(1/\omega _0}_{s},
\end{equation}
and similarly when $\Gamma ^{(\omega _0)}_{s}$ and $\Gamma ^{(1/\omega _0)}_{s}$
are replaced by 
$\Gamma ^{(\omega _0)}_{0,s}$ and $\Gamma ^{(1/\omega _0)}_{0,s}$, respectively.

\par

For general $\omega _0$ it is clear that $\Gamma ^{(\omega _0)}_{0,s}
\subseteq \Gamma ^{(\omega _0)}_{s}\subseteq S^{(\omega _0)}$. On the other
hand, for the weights $\omega _1$, $\omega _2$ and $\omega _3$
in
$\Gamma ^{(\omega _1)}_{0,s}$, $\Gamma ^{(\omega _2)}_{s}$ and
$S^{(\omega _3)}$ we always assume that they belong to
$\mascP _{E,s}(\rr {2d})$, $\mascP _{E,s}^0(\rr {2d})$ and
$\mascP (\rr {2d})$, respectively. That is, they should satisfy
\begin{gather*}
\omega _1(X+Y)\lesssim \omega _1(X)e^{r_1|Y|^{\frac 1s}},
\quad
\omega _2(X+Y)\lesssim \omega _2(X)e^{r_2|Y|^{\frac 1s}},
\\[1ex]
\text{and}\quad
\omega _3(X+Y)\lesssim \omega _3(X)(1+|Y|)^N,
\end{gather*}
for some $r_1>0$ and $N>0$, and every $r_2>0$. Since it is clear
that $\mascP \subseteq \mascP _{E,s}^0\subseteq \mascP _{E,s}$,
it follows by straightforward computations that there are admissible
$a_1\in \Gamma ^{(\omega _1)}_{0,s}$ which are not contained in any admissible
$\Gamma ^{(\omega _0)}_{s}$ and $S^{(\omega _0)}$, and admissible
$a_2\in \Gamma ^{(\omega _2)}_{s}$ which are not contained in
any $S^{(\omega _0)}$.

\par

As in \cite{BoCh}, \eqref{Eq:IntrInverses2} is obtained by
proving that the evolution equation
\begin{equation}\label{Eq:IntrSymeqDiff2}
(\partial_t a)(t,\cdo)=(b+\log \vartheta )\wpr a(t, \cdo),
\qquad a(0,\cdo)=a_0\in \Gamma ^{(\omega )}_{s},\ \vartheta \in
\Gamma ^{(\vartheta )}_{s},
\end{equation}
analogous to \eqref{Eq:IntrSymeqDiff}, has a unique solution
$a(t,\cdo )$ which belongs to $\Gamma ^{(\omega\vartheta^t )}_{s}$
(and similarly when the $\Gamma ^{(\omega )}_{s}$-spaces are replaced
by corresponding $\Gamma ^{(\omega )}_{0,s}$-spaces), given in Section 
\ref{sec3}.

\par

In Sections \ref{sec4} and \ref{sec5} we use the framework in
\cite{GrTo1} in combination with \eqref{Eq:IntrInverses2} to extend the
lifting properties in \cite{GrTo1} in such ways that the involved weights
are allowed to belong to $\mascP _{E,s}^0$ or in $\mascP _{E,s}$ instead 
of the
smaller set $\mascP$ which is the assumption in \cite{GrTo1}.

\par

Our main result, which is similar to \cite[Theorem 0.1]{GrTo1},
can be stated as follows.

\par

\begin{thm}\label{originalgoal}
  Let $s\ge 1$, $\omega ,\omega _0\in \mascP _{E,s}^0(\rr {2d})$,
  $p,q\in (0,\infty ]$ and let $\phi \in
  \maclS _s(\rr{d})$. Then the Toeplitz operator $\tp _\phi (\omega
  _0)$ is an isomorphism from $M^{p,q}_{(\omega )}(\rr {2d})$ onto
  $M^{p,q}_{(\omega /\omega _0)}(\rr {2d})$.
\end{thm}

\par

We note that, in contrast to \cite{GrTo1,GrTo2}, our lifting properties
also hold for modulation spaces which may fail to be Banach spaces,
since $p$ and $q$ in Theorem \ref{originalgoal} are allowed to be
smaller than $1$.

\par

We will establish several related result. Firstly, the window
function may be chosen in certain modulation spaces that are much
larger than the Gelfand-Shilov
space $\maclS _s$. Secondly, the theorem holds for a more general
family of modulation spaces that includes the classical modulation spaces.
Finally, we also establish isomorphisms given by
pseudo-differential operators rather than Toeplitz operators.

\par

In contrast to \cite{GrTo2}, we do not impose in Theorem 
\ref{originalgoal} and in its
related results in Section \ref{sec5} that $\omega _0$ should
be radial in each phase shift (cf. e.{\,}g.
\cite[Theorem 4.3]{GrTo2}). Summing up, our lifting
results in Section \ref{sec5} extend the lifting results
in \cite{GrTo1,GrTo2}.

\par

The paper is organised as follows. In Section \ref{sec1} we introduce
some notation, and discuss modulation spaces and Gelfand-Shilov spaces of
functions and distributions, and pseudo-differential calculus. In Section
\ref{sec2} we introduce and discuss basic properties for confinements of
symbols in $\Gamma ^{(\omega _0)}_{s}$ and in $\Gamma ^{(\omega _0)}_{0,s}$.
These considerations are related to the discussions in \cite{BoCh,Le2},
but here adapted to symbols that possess Gevrey regularity, e.{\,}g.
when the symbols belong to $\Gamma ^{(\omega _0)}_{s}$ or
$\Gamma ^{(\omega _0)}_{0,s}$. 

\par

In contrast to the classical H{\"o}rmander symbol classes $S^r_{1 ,0}$
and the SG-classes $\operatorname{SG}^{m,\mu}_{1,1}$, techniques on
asymptotic expansions are absent for symbols in the classes
$\Gamma ^{(\omega _0)}_{s}$ or in $\Gamma ^{(\omega _0)}_{0,s}$,
and might be absent for symbols in the general H{\"o}rmander class
$S(m,g)$. The approach with confinements is, roughly speaking,
a sort of stand-in of these absent asymptotic expansion techniques.

\par

In Section \ref{sec3} we show that the \eqref{Eq:IntrSymeqDiff2}
has a unique solution with the requested properties, which leads to
\eqref{Eq:IntrInverses2}. In Sections \ref{sec4} and \ref{sec5}
we use the results
from Section \ref{sec3} to deduce lifting properties for modulation
spaces under pseudo-differential operators and Toeplitz operators
with symbols in $\Gamma ^{(\omega _0)}_{s}$ or
in $\Gamma ^{(\omega _0)}_{0,s}$

\par

\section{Preliminaries}\label{sec1}

\par

In this section we recall some basic facts on
modulation spaces, Gelfand-Shilov spaces of functions
and distributions and pseudo-differential calculus
(cf. \cite{Fe3,Fe4,Fe5,Fe5.5,Fe6,FG1,FG2,Gc2,GS06,Ho1,Ho3,Le2,
PiTe04,Sj1,Te2,To5,To7,To8,To9,To10}).

\par

\subsection{Weight functions}
A \emph{weight} on $\rr d$ is a positive function $\omega \in  L^\infty _{loc}(\rr d)$
such that $1/\omega \in  L^\infty _{loc}(\rr d)$. If $\omega$ and $v$ are weights on
$\rr d$, then $\omega$ is called \emph{moderate} or \emph{$v$-moderate}, if
	\begin{equation}\label{eq:2}
		\omega(x+y)\le C\omega(x)v(y),\quad x,y\in\rr{d},
	\end{equation}
for some constant $C$. The set of all moderate weights on $\rr d$ is denoted
by $\mascP _E(\rr d)$.
The weight $v$ on $\rr d$ is called submultiplicative, if it is even and \eqref{eq:2}
holds for $\omega =v$ and $C=1$. From now on, $v$ always denote a submultiplicative
weight if nothing else is stated. In particular,
if \eqref{eq:2} holds and $v$ is submultiplicative, then it follows
by straightforward computations that
\begin{equation}\label{eq:2Next}
\begin{gathered}
C^{-1}\frac {\omega (x)}{v(y)} \le \omega(x+y) \le C\omega(x)v(y),
\\[1ex]
v(x+y) \le v(x)v(y)
\quad \text{and}\quad v(x)=v(-x)\ge 1,
\quad x,y\in\rr{d}.
\end{gathered}
\end{equation}

\par

If $\omega$ is a moderate weight on $\rr d$, then there is a
submultiplicative weight
$v$ on $\rr d$ such that \eqref{eq:2} and \eqref{eq:2Next}
hold (cf. \cite{To5,To7,To14}). Moreover if $v$ is
submultiplicative on $\rr d$, then
\begin{equation}\label{Eq:CondSubWeights}
1\lesssim v(x) \lesssim e^{r|x|}
\end{equation}
for some constant $r>0$ (cf. \cite{Gc2.5}). Here and in what follows we write 
$A(\theta )\lesssim B(\theta )$, $\theta \in \Omega$,
if there is a constant $c>0$ such that $A(\theta )\le cB(\theta )$
for all $\theta \in \Omega$. In particular, if $\omega$ is moderate, then
\begin{equation}\label{Eq:ModWeightProp}
\omega (x+y)\lesssim \omega (x)e^{r|y|}
\quad \text{and}\quad
e^{-r|x|}\le \omega (x)\lesssim e^{r|x|},\quad
x,y\in \rr d
\end{equation}
for some $r>0$.

\par

Next we introduce suitable subclasses of $\mascP _E$.

\par

\begin{defn}\label{Definition:Weightclasses}
Let $s>0$. The set $\mascP _{E,s}(\rr d)$ ($\mascP _{E,s}^0(\rr d)$) consists of
		all $\omega \in \mascP _E(\rr d)$ such that
		\begin{equation}\label{eq:gsmodw}
			\omega(x+y)\lesssim \omega(x) e^{r|y|^\frac{1}{s}}, \quad x,y\in\rr{d};
		\end{equation}
		holds for some (every) $r>0$.
\end{defn}

\par

By \eqref{Eq:ModWeightProp} it follows that $\mascP _{E,s_1}^0=\mascP _{E,s_2}
=\mascP _E$ when $s_1<1$ and $s_2\le 1$. For convenience we set
$\mascP^0_E(\rr d)=\mascP^0_{E,1}(\rr d)$.


\par

\par

\subsection{Gelfand-Shilov spaces}\label{subs:gs}

\par

We let $\mathscr F$ be the Fourier transform given by
$$
(\mathscr Ff)(\xi )= \widehat f(\xi ) \equiv (2\pi )^{-d/2}\int _{\rr
{d}} f(x)e^{-i\scal  x\xi }\, dx
$$
when $f\in L^1(\rr d)$. Here $\scal \cdo \cdo$ denotes the
usual scalar product on $\rr d$. 

\par


\par

\begin{defn}\label{def:gsspaces}
	The Gelfand-Shilov space $\Srr$ of Roumieu type ($\Sigma_s^\sigma(\rr{d})$ of Beurling type), $\sigma>0$, $s>0$, consists
	of all $f\in\mascS(\rr{d})$ such that
	\begin{equation}\label{eq:srra}
		|f(x)|\lesssim e^{-r|x|^\frac{1}{s}}
		\qquad \text{and}\qquad
		|(\mascF f)(\xi)|\lesssim e^{-r|\xi|^\frac{1}{\sigma}},\quad x, \xi\in\rr{d}
	\end{equation}
	for some $r>0$ (for all $r>0$). We set $\maclS_s = \maclS_s^s$ and
	$\Sigma_s=\Sigma _s^s$.
\end{defn}
The classes $\Srr$ and related generalizations were widely studied, and used in the applications
to partial differential equations, see for example \cite{BiGr03, Mi70, PiTe04, CoPiRoTe05, ChChKi96, GrZi04}. We recall the following characterisations of $\Srr$.

\par

\begin{prop}\label{thm:srrequiv}
	Let $s,\sigma>0$, $p\in[1,\infty]$ and
	let $f\in \mascS (\rr{d})$.
	Then the following conditions are equivalent:
	\begin{enumerate}
		\item $f\in \Srr$ ($f\in \Sigma _s^\sigma (\rr d)$);
		%
		%
		\item for some (every) $h>0$ it holds
		$$
		\|x^\alpha  f\|_{L^p}\lesssim h^{|\alpha|}\alpha!^s \quad \text{and}\quad
			\|\xi^\beta  \widehat f\|_{L^p}\lesssim h^{|\beta|}\beta!^\sigma ,
			\quad \alpha ,\beta\in\nn{d};
		$$
		\item for some (every) $h>0$ it holds
		$$
		\|x^\alpha f\|_{L^p}\lesssim h^{|\alpha|}\alpha!^s
		\quad \text{and}\quad 
		\|\partial^\beta f\|_{L^p}\lesssim h^{|\beta|}\beta!^\sigma,
		\quad \alpha ,\beta\in\nn{d};
		$$
		\item for some (every) $h>0$ it holds
		\[
			\|x^\alpha \partial^\beta f(x)\|_{L^p}\lesssim h^{|\alpha+\beta|}
			\alpha!^s\,\beta!^\sigma,
			\quad\alpha,\beta\in\nn{d};
		\]
		%
		%
		\item for some (every) $h,r>0$ it holds
		\label{pt:ptvi}
		\[
			\|e^{r|\cdo|^\frac{1}{s}}\partial^\alpha f\|_{L^p}\lesssim h^{|\alpha|}(\alpha!)^\sigma\,
			\alpha\in\nn{d}.
		\] 
	\end{enumerate}
\end{prop}

\par


\par

\begin{rem}
Let $h,s,s_0, \sigma ,\sigma _0\in \mathbf R_+$ be such that $s+\sigma >1$ and
$s_0+\sigma _0\ge 1$, and let $\mathcal S_{s,h}^\sigma (\rr d)$ be the set of all
$f\in C^\infty (\rr d)$ such that
\begin{equation*}
\nm f{\mathcal S_{s,h}^\sigma}\equiv \sup \frac {|x^\beta \partial ^\alpha
f(x)|}{h^{|\alpha + \beta |}\alpha !^\sigma  \beta !^s}
\end{equation*}
is finite. Here the supremum is taken over all $\alpha ,\beta \in
\mathbf N^d$ and $x\in \rr d$.

\par

Obviously $\maclS _{s,h}^\sigma (\rr d)$ is a Banach space which increases as $h$,
$s$ and $\sigma$ increase,
and is contained in $\mascS (\rr d)$, the set of Schwartz functions on $\rr d$.
Furthermore
$$
\maclS _{s,h}^\sigma (\rr d)\quad \text{and}\quad
\bigcup _{h>0}\maclS _{s_0,h}^{\sigma _0}(\rr d)
$$
are dense in $\mathscr S(\rr{d})$. Hence, the dual $(\maclS _{s,h}^\sigma )'(\rr d)$ of
$\mathcal S_{s,h}^\sigma (\rr d)$ is a Banach space which contains $\mathscr S'(\rr d)$.

\par

The spaces $\maclS _{s}^\sigma (\rr d)$ and
$\Sigma _s^\sigma (\rr d)$ are the inductive and projective
limits, respectively, of $\mathcal S_{s,h}^\sigma (\rr d)$ with respect
to $h$. This implies that
\begin{equation}\label{GSspacecond1}
\mathcal S_s^\sigma (\rr d) = \bigcup _{h>0}\mathcal S_{s,h}^\sigma (\rr d)
\quad \text{and}\quad \Sigma _{s}^\sigma (\rr d)
=\bigcap _{h>0}\maclS _{s,h}^\sigma (\rr d),
\end{equation}
and that the topology for $\mathcal S_s^\sigma (\rr d)$ is the
strongest possible one such that each inclusion map
from $\maclS _{s,h}^\sigma (\rr d)$ to $\maclS _s^\sigma (\rr d)$
is continuous. Moreover, any of the conditions (2)--(5) in Proposition
\ref{thm:srrequiv} induce the same topology for $\maclS _s^\sigma (\rr d)$
and $\Sigma _s^\sigma (\rr d)$.

\par

The Gelfand-Shilov distribution spaces $(\maclS _s^\sigma )'(\rr d)$
and $(\Sigma _s^\sigma )'(\rr d)$ are the projective and inductive limit
respectively of $(\mathcal S_{s,h}^\sigma )'(\rr d)$.  Hence
\begin{equation}\tag*{(\ref{GSspacecond1})$'$}
(\mathcal S_s^\sigma )'(\rr d) = \bigcap _{h>0}(\mathcal
S_{s,h}^\sigma )'(\rr d)\quad \text{and}\quad (\Sigma _s^\sigma )'(\rr d)
=\bigcup _{h>0} (\mathcal S_{s,h}^\sigma )'(\rr d).
\end{equation}
We have that $(\mathcal S_s^\sigma )'$
and $(\Sigma _s^\sigma )'$ are the topological duals of $\maclS _s^\sigma $
and $\Sigma _s^\sigma $, respectively (see \cite{Pi88}).
\end{rem}

\par

\begin{rem}
Let $s,\sigma >0$. Then $\Sigma _s^\sigma (\rr d)$ is a Fr{\'e}chet
space with seminorms $\nm \cdo{\mathcal S_{s,h}^\sigma }$, $h>0$.
Moreover, $\maclS _s^\sigma (\rr d)\neq \{ 0\}$ if and only if
$s+\sigma \ge 1$, and $\Sigma _s^\sigma (\rr d)\neq \{ 0\}$
if and only if $s+\sigma \ge 1$ and $(s,\sigma )\neq (\frac 12 ,\frac 12)$.
Moreover, if $\ep >0$ and $s+\sigma \ge 1$, then
$$
\Sigma _s ^\sigma (\rr d)\subseteq \maclS _s^\sigma (\rr d)\subseteq
\Sigma _{s+\ep}^{\sigma +\ep}(\rr d)\subseteq \mascS (\rr d)\subseteq
\mascS '(\rr d) \subseteq (\Sigma _{s+\ep}^{\sigma +\ep})'(\rr d)
\subseteq (\maclS _s^\sigma )'(\rr d).
$$
If in addition $(s,\sigma )\neq (\frac 12,\frac 12)$, then
$$
(\maclS _s^\sigma )'(\rr d)\subseteq (\Sigma _s ^\sigma )'(\rr d).
$$
\end{rem}

%
%

%

\par

The Gelfand-Shilov spaces are invariant or possess convenient
mapping properties under several basic
transformations. For example they are invariant under
translations, dilations, and under (partial) Fourier transformations.

\par

The Fourier transform $\mathscr F$ on $\mascS (\rr d)$ extends 
uniquely to homeomorphisms on $\mathscr S'(\rr d)$, $\mathcal
S_s'(\rr d)$ and $\Sigma _s'(\rr d)$, and restricts to 
homeomorphisms on $\mathcal S_s(\rr d)$
and $\Sigma _s(\rr d)$, and to a unitary operator on $L^2(\rr d)$.

\medspace

We also recall some mapping properties of Gelfand-Shilov
spaces under short-time Fourier transforms.
Let $\phi \in \mathscr S(\rr d)$ be fixed. For every $f\in
\mathscr S'(\rr d)$, the \emph{short-time Fourier transform} $V_\phi
f$ is the distribution on $\rr {2d}$ defined by the formula
\begin{equation}\label{defstft}
(V_\phi f)(x,\xi ) =\mathscr F(f\, \overline{\phi (\cdo -x)})(\xi ) =
(f,\phi (\cdo -x)e^{i\scal \cdo \xi}).
\end{equation}
We recall that if $T(f,\phi )\equiv V_\phi f$ when $f,\phi \in \maclS _{1/2}(\rr d)$,
then $T$ is uniquely extendable to sequentially continuous mappings
\begin{alignat*}{2}
T\, &:\, & \maclS _s'(\rr d)\times \maclS _s(\rr d) &\to
\maclS _s '(\rr {2d})\bigcap C^\infty (\rr {2d}),
\\[1ex]
T\, &:\, & \maclS _s'(\rr d)\times \maclS _s'(\rr d) &\to
\maclS _s '(\rr {2d}),
\end{alignat*}
and similarly when $\maclS _s$ and $\maclS _s'$ are replaced
by $\Sigma _s$ and $\Sigma _s'$, respectively, or by
$\mascS$ and $\mascS '$, respectively (cf. \cite{CPRT10,To14}).
We also note that $V_\phi f$ takes the form
\begin{equation}\tag*{(\ref{defstft})$'$}
V_\phi f(x,\xi ) =(2\pi )^{-d/2}\int _{\rr d}f(y)\overline {\phi
(y-x)}e^{-i\scal y\xi}\, dy
\end{equation}
when $f\in L^p_{(\omega )}(\rr d)$ for some $\omega \in
\mascP _E(\rr d)$, $\phi \in \Sigma _1(\rr d)$ and $p\ge 1$. Here
$L^p_{(\omega )}(\rr d)$, when $p\in (0,\infty ]$ and
$\omega \in \mascP _E(\rr d)$, is the set of all $f\in L^p_{loc} (\rr d)$ such
that $f\cdot \omega \in L^p(\rr d)$. 

\par

\subsection{Suitable function classes with Gelfand-Shilov regularity}
\par

The next result shows that for any $\omega\in\mascP _E(\rr d)$ one can find an equivalent 
weight $\omega_0$ which satisfies similar regularity estimates as functions in Gelfand-Shilov spaces.

\begin{prop}\label{Prop:EquivWeights}
Let $\omega \in \mascP _E(\rr d)$ and $s >0$.
Then there is an $\omega _0\in \mascP _E(\rr d)\cap C^\infty (\rr d)$
such that the following is true:
\begin{enumerate}
\item $\omega _0\asymp \omega $;

\vrum

\item $|\partial ^{\alpha}\omega _0(x)|\lesssim \omega _0(x) h^{|\alpha |}\alpha !^s
\asymp \omega (x) h^{|\alpha |}\alpha !^s$ for every $h>0$.
\end{enumerate}
\end{prop}

\par

\begin{proof}
We may assume that $s<1$. It suffices to prove that (2) should hold for some $h>0$.
Let $\phi _0\in \Sigma _{1-s}^s(\rr d)\setminus 0$, and let
$\phi =|\phi _0|^2$. Then $\phi \in \Sigma _{1-s}^s(\rr d)$, giving that
$$
|\partial ^\alpha \phi (x)| \lesssim h^{|\alpha |}e^{-c|x|^{\frac 1{1-s}}}\alpha !^s,
$$
for every $h>0$ and $c>0$. Now let $\omega _0=\omega *\phi$.

\par

We have
\begin{align*}
|\partial ^\alpha \omega _0(x)| &= \left | \int \omega (y)(\partial ^\alpha \phi )(x-y)
\, dy  \right |
\\[1ex]
&\lesssim 
h^{|\alpha |}\alpha !^s \int \omega (y)e^{-c|x-y|^{\frac 1{1-s}}} \, dy
\\[1ex]
&\lesssim 
h^{|\alpha |}\alpha !^s \int \omega (x+(y-x))e^{-c|x-y|^{\frac 1{1-s}}} \, dy
\\[1ex]
&\lesssim 
h^{|\alpha |}\alpha !^s \omega (x) \int e^{-\frac c2|x-y|^{\frac 1{1-s}}} \, dy
\asymp h^{|\alpha |}\alpha !^s \omega (x),
\end{align*}
where the last inequality follows from \eqref{Eq:ModWeightProp} and the fact that
$\phi$ is bounded by a super exponential function. This gives the first part of (2).

\par

The equivalences in (1) follows in the same way as in \cite{To14}.
More precisely, by \eqref{Eq:ModWeightProp} we have
\begin{align*}
\omega _0(x) &= \int \omega (y)\phi (x-y)\, dy = \int \omega (x+(y-x))
\phi (x-y)\, dy
\\[1ex]
&\lesssim \omega (x) \int e^{c|x-y|}\phi (x-y)\, dy \asymp \omega (x).
\end{align*}
In the same way, \eqref{Eq:ModWeightProp} gives
\begin{align*}
\omega _0(x) &= \int \omega (y)\phi (x-y)\, dy = \int
\omega (x+(y-x))\phi (x-y)\, dy
\\[1ex]
&\gtrsim \omega (x) \int e^{-c|x-y|}\phi (x-y)\, dy \asymp \omega (x),
\end{align*}
and (1) as well as the second part of (2) follow.
\end{proof}

\par

A weight $\omega _0$ which satisfies Proposition \ref{Prop:EquivWeights} (2) is
called \emph{elliptic} or \emph{$s$-elliptic}.

\par

\begin{defn}\label{Def:SymbClExp}
Let $s\ge 0$ and $\omega \in \mascP _E(\rr {d})$. The class $\Gamma ^{(\omega )}_s(\rr {d})$
($\Gamma ^{(\omega )}_{0,s}(\rr {d})$) consists of all $f\in C^\infty (\rr {d})$ such that
$$
|D^\alpha f(x)| \lesssim  \omega (x)\,h^{|\alpha |}\alpha !^s,  \quad x\in\rr{d},
$$
for \emph{some} $h>0$ (for \emph{every} $h>0$).
\end{defn}

\par

Evidently, by Proposition \ref{Prop:EquivWeights} it follows that
the family of symbol classes in Definition
\ref{Def:SymbClExp} is not reduced when the assumption $\omega
\in \mascP _E(\rr {2d})$
is replaced by the more restrictive assumption $\omega \in \mascP 
_{E,s}(\rr {2d})$ or by
$\omega \in \mascP ^0_{E,s}(\rr {2d})$.

\par

By similar arguments as in the proof of Proposition
\ref{Prop:EquivWeights}
we get the following analog of Proposition 2.3.16 in
\cite{Le}. The details are left for the reader.

\par

\begin{prop}\label{prop:conv.class}
    Let $s>1/2$, $\omega\in\mascP _E(\rr {2d})$, and $\phi\in\Sigma _s(\rr {2d})$.
    Then $\omega*\phi$ belongs to $\Gamma ^{(\omega )}_{0,s}$.
\end{prop}

\par

%
%
%
%


\par

The following definition is motivated by Lemma 2.6.13 in \cite{Le}.

\par

\begin{defn}
Let $s\ge 1$, $\omega \in \mascP _E(\rr d)$ and $\vartheta _0=1+|\log \omega |$.
Then $c$ is called \emph{comparable} to $\omega$ with respect to $s\ge 1$ if
\begin{enumerate}
\item $\nm {c-\log \omega}{L^\infty} <\infty$

\vrum

\item $c \in \Gamma^{(\vartheta _0)}_s(\rr d)$ and $\partial ^\alpha c \in
\Gamma^{( 1 )}_s(\rr d)$, when $|\alpha |=1$.
\end{enumerate} 
\end{defn}

\par

%

\begin{prop}
Let $\omega ,v\in \mascP _E(\rr d)$ be such that $v$ is submultiplicative and
$\omega$ is $v$-moderate. Also let
$$
v_1(x)\equiv 1+|\log v(x)|
\quad \text{and}\quad
\omega_1(x)\equiv 1+|\log \omega (x)|.
$$
Then $v_1$ is submultiplicative and $\omega_1$ is $v_1$-moderate,
satisfying \eqref{eq:2} with $1+\log C\ge1$ in place of $C$.
\end{prop}

\par

\begin{proof}
By \eqref{eq:2Next} we get
\begin{align*}
	v_1(x+y)&=1+\log v(x+y)\le 1 + \log v(x) + \log v(y)
	\\
	& \le (1+\log v(x))\,(1+\log v(y)) = v_1(x)\,v_1(y),
\end{align*}
and
\begin{align*}
	\omega_1(x+y)&=1+|\log\omega(x+y)|
	\\
	&\le 1+\log C+|\log\omega(x)|+\log v(y)
	\\
	&\le (1+\log C)(1+|\log\omega(x)|)\,(1+\log v(y))
	\\
	&\le (1+\log C) \, \omega_1(x)\,v_1(y),
\end{align*}
as claimed.
\end{proof}

\par

\begin{lemma}
Let $s\ge 1$, $\omega \in \mascP _E(\rr d)$ and $\vartheta_0 =1+|\log \omega |$. Then
the following is true:
\begin{enumerate}
\item there exists an elliptic
weight $\omega _0\in \mascP _E(\rr d)\cap \Gamma^{(\omega )}_s(\rr d)$ such that
$\omega \asymp \omega _0$ and $1+ |\log \omega _0| \in \mascP _E(\rr d)\cap
\Gamma^{(\vartheta_0)}_s(\rr d)$;

\vrum

\item there exists an element $c$ which is comparable with $\omega _0$.
\end{enumerate}
\end{lemma}

\par

\begin{proof}
If $\omega _1(x)$ is equal to $\omega (x)$ when $\omega (x)\ge e$ or
$\omega (x)\le e^{-1}$, and $3$ otherwise, then $\omega _1$ is equivalent to $\omega$.
The result now follows from Proposition \ref{Prop:EquivWeights} and its proof,
with $\omega _1$ in place of $\omega$.
\end{proof}

\par

\subsection{Modulation spaces} 

\par

Let $\phi \in \Sigma _1(\rr d)\setminus 0$, $p,q\in (0,\infty ]$
and $\omega \in\mascP _E(\rr {2d})$ be fixed. Then the
\emph{modulation space} $M^{p,q}_{(\omega )}(\rr d)$ consists of all
$f\in \Sigma _1'(\rr d)$ such that
\begin{equation}\label{modnorm}
\nm f{M^{p,q}_{(\omega )}} \equiv \Big ( \int \Big ( \int |V_\phi f(x,\xi
)\omega (x,\xi )|^p\, dx\Big )^{q/p} \, d\xi \Big )^{1/q} <\infty
\end{equation}
(with the obvious modifications when $p=\infty$ and/or
$q=\infty$). We set $M^p_{(\omega )}=M^{p,p}_{(\omega )}$, and
if $\omega =1$, then we set $M^{p,q}=M^{p,q}_{(\omega )}$
and $M^{p}=M^{p}_{(\omega )}$.

\par

The following proposition is a consequence of well-known facts
in \cite {Fe4,GaSa,Gc2,To20}. Here and in what follows, we let $p'$
denotes the conjugate exponent of $p$, i.{\,}e.
$$
p'
=
\begin{cases}
\infty & \text{when}\ p\in (0,1]
\\[1ex]
\displaystyle{\frac p{p-1}} & \text{when}\ p\in (1,\infty )
\\[2ex]
1 & \text{when}\ p=\infty \, .
\end{cases}
$$

\par

\begin{prop}\label{p1.4}
Let $p,q,p_j,q_j,r\in (0,\infty ]$ be such that $r\le \min (1,p,q)$,
$j=1,2$, let $\omega
,\omega _1,\omega _2,v\in\mascP _E(\rr {2d})$ be such that $\omega$
is $v$-moderate, $\phi \in M^r_{(v)}(\rr d)\setminus 0$, and let $f\in
\Sigma _1'(\rr d)$. Then the following is true:
\begin{enumerate}
\item $f\in
M^{p,q}_{(\omega )}(\rr d)$ if and only if \eqref {modnorm} holds,
i.{\,}e. $M^{p,q}_{(\omega )}(\rr d)$ is independent of the choice of
$\phi$. Moreover, $M^{p,q}_{(\omega )}$ is a Banach space under the
norm in \eqref{modnorm}, and different choices of $\phi$ give rise to
equivalent norms;

\vrum

\item if  $p_1\le p_2$,
$q_1\le q_2$ and $\omega _2\le C\omega _1$ for some constant $C$, then
\begin{alignat*}{3}
\Sigma _1(\rr d)&\subseteq &M^{p_1,q_1}_{(\omega _1)}(\rr
d) &\subseteq  & M^{p_2,q_2}_{(\omega _2)}(\rr d)&\subseteq 
\Sigma _1'(\rr d).
\end{alignat*}
\end{enumerate}
\end{prop}

\par

Proposition \ref{p1.4}{\,}(1) allows us to be rather vague about
to the choice of $\phi \in  M^r_{(v)}\setminus 0$ in
\eqref{modnorm}. For example, if $C>0$ is a constant and $\Omega$ is a
subset of $\Sigma _1'$, then $\nm a{M^{p,q}_{(\omega )}}\le C$ for
every $a\in \Omega$, means that the inequality holds for some choice
of $\phi \in  M^r_{(v)}\setminus 0$ and every $a\in
\Omega$. Evidently, for any other choice
of $\phi \in  M^r_{(v)}\setminus 0$, a similar inequality is true
although $C$ may have to be replaced by a larger constant, if necessary.

\par

Let $s,t\in \mathbf R$. Then the weights
\begin{equation}\label{e1.3}
(x,\xi )\mapsto  \langle \xi \rangle ^s\langle x\rangle ^t
\quad \text{and}\quad
(x,\xi )\mapsto  \langle {(x,\xi )} \rangle ^s
,\qquad x,\xi \in \rr d,
\end{equation}
are common in the applications. It follows that they belong to $\mascP
(\rr {2d})$ for every $s,t\in \mathbf R$. If $\omega \in\mascP (\rr {2d})$,
then $\omega$ is moderated by any of the weights in \eqref{e1.3}
provided $s$ and $t$ are chosen large enough.

\par

\begin{rem}\label{Rem:ExtWindows}
For modulation spaces of the form $M^{p,q}_{(\omega )}$ with  fixed
$p,q\in [1,\infty ]$ the norm equivalence   in Proposition
\ref{p1.4}(1)  can be extended to a larger class of windows.  In fact,
assume  that $\omega ,v\in \mascP _E(\rr {2d})$ with  $\omega  $
being $v$-moderate and
$$
1\leq r\le \min (p,p',q,q') \, .
$$
 Let $\phi \in M^r_{(v)}(\rr d)\setminus \{0\}$. Then a Gelfand-Shilov
 distribution  $f\in
 \Sigma _1'(\rr d )$ belongs to  $M^{p,q}_{(\omega )}(\rr d)$, if and
only if $V_\phi f\in L^{p,q}_{(\omega )}(\rr {2d})$. Furthermore,
different choices of $\phi\in M^r_{(v)}(\rr d)\setminus \{0\}$ in $\nm
{V_\phi f}{L^{p,q}_{(\omega )}}$ give rise to equivalent norms.
(Cf. Theorem 3.1 in \cite{To10}.)
\end{rem}

\par

\subsection{A broader family of modulation spaces}

\par

As announced in the introduction we consider in Section  \ref{sec2}
mapping properties for pseudo-differential operators when acting on
a broader class of modulation spaces, which are defined by imposing certain
types of translation invariant solid BF-space norms on the short-time
Fourier transforms. (Cf. \cite{Fe3,Fe4,Fe5,Fe5.5,Fe6,FG1,FG2}.)

\par

\begin{defn}\label{bfspaces1}
Let $\mascB \subseteq L^r_{loc}(\rr d)$ be a quasi-Banach
of order $r\in (0,1]$, and let $v \in\mascP _E(\rr d)$.
Then $\mascB$ is called a \emph{translation invariant
Quasi-Banach Function space on $\rr d$}, or \emph{invariant
QBF space on $\rr d$}, if there is a constant $C$ such
that the following conditions are fulfilled:
\begin{enumerate}
\item if $x\in \rr d$ and $f\in \mascB$, then $f(\cdo -x)\in
\mascB$, and 
\begin{equation}\label{translmultprop1}
\nm {f(\cdo -x)}{\mascB}\le Cv(x)\nm {f}{\mascB}\text ;
\end{equation}

\vrum

\item if  $f,g\in L^r_{loc}(\rr d)$ satisfy $g\in \mascB$ and $|f|
\le |g|$, then $f\in \mascB$ and
$$
\nm f{\mascB}\le C\nm g{\mascB}\text .
$$
\end{enumerate}
\end{defn}

\par

If $v$ belongs to $\mascP _{E,s}(\rr d)$
($\mascP _{E,s}^0(\rr d)$) , then $\mascB$ in Definition \ref{bfspaces1}
is called an \emph{invariant BF-space} of Roumieu type (Beurling type) of order $s$.

\par

We notice that the quasi-norm $\nm \cdo {\mascB}$ in Definition \ref{bfspaces1}
should satisfy
\begin{alignat}{2}
\nm {f+g}{\mascB} &\le 2^{\frac 1r-1}(\nm {f}{\mascB} + \nm {g}{\mascB})&
\quad f,g &\in \mascB .
\label{Eq:WeakTriangle1}
\intertext{By Aoki and Rolewi{\'c} in \cite{Aoki,Rol} it follows that there is
an equivalent quasi-norm to the previous one which additionally satisfies}
\nm {f+g}{\mascB}^r &\le \nm {f}{\mascB}^r + \nm {g}{\mascB}^r &
\quad f,g &\in \mascB .
\label{Eq:WeakTriangle2}
\end{alignat}
From now on we suppose that the quasi-norm of $\mascB$ has been chosen
such that both \eqref{Eq:WeakTriangle1} and \eqref{Eq:WeakTriangle2} hold true.

\par

It follows from (2) in Definition \ref{bfspaces1} that if $f\in
\mascB$ and $h\in L^\infty$, then $f\cdot h\in \mascB$, and
\begin{equation}\label{multprop}
\nm {f\cdot h}{\mascB}\le C\nm f{\mascB}\nm h{L^\infty}.
\end{equation}
If $r=1$, then $\mascB$ in Definition \ref{bfspaces1} is a Banach
space, and the condition (2)  means that a
translation invariant QBF-space is a solid BF-space in the sense of
(A.3) in \cite{Fe6}. 
The space $\mascB$ in Definition \ref{bfspaces1} is called an
\emph{invariant BF-space} (with respect to $v$) if $r=1$, and
Minkowski's inequality holds true, i.{\,}e.
\begin{equation}\label{Eq:MinkIneq}
\nm {f*\fy}{\mascB}\le C \nm {f}{\mascB}\nm \fy{L^1_{(v)}},
\qquad f\in \mascB ,\ \fy \in \Sigma _1 (\rr d)
\end{equation}
for some constant $C$ which is independent of
$f\in \mascB$ and $\fy \in \Sigma _1 (\rr d)$.

\par

\begin{example}\label{Lpqbfspaces}
Assume that $p,q\in [1,\infty ]$, and let $L^{p,q}_1(\rr {2d})$ be the
set of all $f\in L^1_{loc}(\rr {2d})$ such that
$$
\nm  f{L^{p,q}_1} \equiv \Big ( \int \Big ( \int |f(x,\xi )|^p\, dx\Big
)^{q/p}\, d\xi \Big )^{1/q}
$$
if finite. Also let $L^{p,q}_2(\rr {2d})$ be the set of all $f\in
L^1_{loc}(\rr {2d})$ such that
$$
\nm  f{L^{p,q}_2} \equiv \Big ( \int \Big ( \int |f(x,\xi )|^q\, d\xi
\Big )^{p/q}\, dx \Big )^{1/p}
$$
is finite. Then it follows that $L^{p,q}_1$ and $L^{p,q}_2$ are
translation invariant BF-spaces with respect to $v=1$.
\end{example}

\par

For translation invariant BF-spaces we make the
following observation.

\par

\begin{prop}\label{p1.4BFA}
Assume that $v\in\mascP _E(\rr {d})$, and that $\mascB$ is an
invariant BF-space with respect to $v$ such that 
\eqref{Eq:MinkIneq}
holds true. Then the
convolution mapping $(\fy ,f)\mapsto \fy *f$ from $C_0^\infty (\rr
d)\times \mascB$ to $\mascB$ extends uniquely to a continuous
mapping from
$L^1_{(v )}(\rr d)\times \mascB$ to $\mascB$, and 
\eqref{Eq:MinkIneq}
holds true for any $f\in \mascB$ and $\fy \in L^1_{(v)}(\rr d)$.
\end{prop}

\par

The result is a straightforward consequence of the fact that 
$C_0^\infty$
is dense in $L^1_{(v)}$.

\par

The quasi-Banach space $\mascB$ above is usually a
mixed quasi-normed Lebesgue space, given as follows.
Let $E$ be a non-degenerate parallelepiped in $\rr d$ which is 
spanned by the
ordered basis $\kappa (E)=\{ e_1,\dots,e_d \}$. That is,
$$
E = \sets{x_1e_1+\cdots+x_de_d}{(x_1,\dots,x_d)\in\rr d,\ 0\leq 
x_k\leq 1,\
k=1,\dots,d}.
$$
The corresponding lattice, dual parallelepiped and dual lattice are
given by
\begin{align*}
\Lambda _E &=\sets{j_1e_1+\cdots +j_de_d}{(j_1,\dots,j_d)\in
\zz d},
\\[1ex]
E' &=\sets {\xi _1e'_1+\cdots+\xi _de'_d}{(\xi _1,\dots ,\xi _d)
\in \rr d,
\ 0\leq \xi _k\leq 1,\ k=1,\dots ,d},
\intertext{and}
\Lambda'_E &= \Lambda_{E'}=\sets{\iota _1e'_1+\cdots +\iota _de'_d}
{(\iota _1,\dots ,\iota _d) \in \zz d},
\end{align*}
respectively, where the ordered basis $\kappa (E') =\{
e'_1,\dots ,e'_d\}$ of $E'$ satisfies
$$
\scal {e_j} {e'_k} = 2\pi \delta_{jk}
\quad \text{for every}\quad
j,k =1,\dots, d.
$$
Note here that the Fourier
analysis with respect to general biorthogonal bases has recently been
developed in \cite{RuTo}.

\par

The basis $e'_1,\dots ,e'_d$ is called the \emph{dual basis} of
$e_1,\dots ,e_d$.
We observe that there is a matrix $T_E$ such that
$e_1,\dots ,e_d$ and $e_1',\dots ,e_d'$ are the images of
the standard basis under $T_E$ and
$T_{E'}= 2\pi(T^{-1}_E)^t$, respectively.

\par

In the following we let
$$
\max \mabfq =\max (q_1,\dots ,q_d)
\quad \text{and}\quad
\min \mabfq =\min (q_1,\dots ,q_d)
$$
when $\mabfq =(q_1,\dots ,q_d)\in (0,\infty ]^d$,
and $\chi _\Omega$ be the characteristic function of $\Omega$.

\par

\begin{defn}\label{Def:MixedLebSpaces}
Let $E$ be a non-degenerate parallelepiped in $\rr d$
spanned by the ordered set $\kappa (E)\equiv \{ e_1,\dots ,e_d\}$
in $\rr d$, $\mabfp =(p_1,\dots ,p_d)\in (0,\infty ]^{d}$ and $r=\min (1,\mabfp )$.
If  $f\in L^r_{loc}(\rr d)$, then
$$
\nm f{L^{\mabfp }_{\kappa (E)}}\equiv
\nm {g_{d-1}}{L^{p_{d}}(\mathbf R)}
$$
where  $g_k(z_k)$, $z_k\in \rr {d-k}$,
$k=0,\dots ,d-1$, are inductively defined as
\begin{align*}
g_0(x_1,\dots ,x_{d}) &\equiv |f(x_1e_1+\cdots +x_{d}e_d)|,
\\[1ex]
\intertext{and}
g_k(z_k) &\equiv
\nm {g_{k-1}(\cdo ,z_k)}{L^{p_k}(\mathbf R)},
\quad k=1,\dots ,d-1.
\end{align*}
The space $L^{\mabfp }_{\kappa (E)}(\rr d)$ consists
of all $f\in L^r_{loc}(\rr d)$ such that
$\nm f{L^{\mabfp}_{\kappa (E)}}$ is finite, and is called
\emph{$E$-split Lebesgue space (with respect to $\mabfp$ and
$\kappa (E)$)}.
\end{defn}

\par

\begin{defn}\label{Def:MixedPhaseShiftLebSpaces}
Let $E_0\subseteq \rr d$ be a non-degenerate
parallelepiped with dual
parallelepiped $E'_0$, and $E\subseteq \rr {2d}$ be a
parallelepiped spanned by the ordered set $\kappa (E)\equiv \{ e_1,\dots ,e_{2d}\}$.
Then $E$ is called a \emph{phase-shift split parallelepiped} (with respect to $E_0$)
if $E$ is non-degenerate
and $d$ of the vectors in $\{ e_1,\dots ,e_{2d}\}$ span $E_0$
and the other $d$ vectors is the corresponding dual basis which span $E_0'$. 
\end{defn}

\par

\par

Next we consider the extended class of modulation spaces
which we are interested
in. 

\par

\begin{defn}\label{bfspaces2}
Assume that $\mascB$ is a translation
invariant QBF-space on $\rr {2d}$, $\omega \in\mascP _E(\rr {2d})$,
and that $\phi \in
\Sigma _1(\rr d)\setminus 0$. Then the set $M_{(\omega)} =
M(\omega ,\mascB )$ consists of all $f\in
\Sigma _1'(\rr d)$ such that
$$
\nm f{M_{(\omega )}}=\nm f{M(\omega ,\mascB )}
\equiv \nm {V_\phi f\, \omega }{\mascB}
$$
is finite.
\end{defn}

\par

Obviously, we have
$
M^{p,q}_{(\omega )}(\rr d)=M(\omega ,\mascB )$
when  $\mascB =L^{p,q}_1(\rr {2d})$ (cf. Example \ref{Lpqbfspaces}).
It follows that many properties which are valid for the classical modulation
spaces also hold for the spaces of the form $M(\omega ,\mascB )$.
For example we have the following proposition, which shows that
the definition of $M(\omega ,\mascB )$ is independent of the
choice of $\phi$ when $\mascB$ is a Banach space. We omit the proof
since it follows by similar arguments as
in the proof of Proposition 11.3.2 in \cite{Gc2}.
(See also \cite{PfTo} for topological aspects of $M(\omega ,\mascB)$.)

\par

\begin{prop}\label{p1.4BF}
Let $\mascB$ be an invariant BF-space with
respect to $v_0\in \mascP _E(\rr {2d})$ for $j=1,2$. Also let
$\omega ,v\in\mascP _E(\rr {2d})$ be such that $\omega$ is
$v$-moderate, $M(\omega ,\mascB )$ is the same as in Definition
\ref{bfspaces2}, and let $\phi \in M^1_{(v_0v)}(\rr d)\setminus
0$ and $f\in \Sigma _1'(\rr d)$. Then $M(\omega ,\mascB )$
is a Banach space, and $f\in M(\omega ,\mascB )$
if and only if $V_\phi f\, \omega \in \mascB$, and
different choices of $\phi$ gives rise to equivalent norms in
$M(\omega ,\mascB )$.
\end{prop}

\par

We refer to \cite {Fe3,Fe4,Fe5,Fe5.5,Fe6,FG1,FG2,GaSa,Gc2,RSTT,To20}
for more facts about modulation spaces.

\par

\subsection{Pseudo-differential operators}

\par

We use the notation $\GL (d,\Omega )$ for the set of $d\times d$-matrices with
entries in the set $\Omega$. Let $s\ge 1/2$, $a\in \maclS _s 
(\rr {2d})$, and $A\in \GL (d,\mathbf R)$ be fixed. Then, the
pseudo-differential operator $\op _A(a)$
is the linear and continuous operator on $\maclS _s (\rr d)$
given by
\begin{equation}\label{eq:psido}
	(\op _A(a)u)(x)=(2\pi)^{-d}\iint 
	a(x-A(x-y),\xi ) \, f(y)\, e^{i\scal {x-y}\xi}\,dyd\xi
\end{equation}
when $f\in\maclS _s(\rr{d})$. For general $a\in \maclS _s'(\rr {2d})$, the
pseudo-differential operator $\op _A(a)$ is defined as the continuous
operator from $\maclS _s(\rr d)$ to $\maclS _s'(\rr d)$ with
distribution kernel given by
\begin{equation}\label{atkernel}
K_{a,A}(x,y)=(2\pi )^{-d/2}(\mascF _2^{-1}a)(x-A(x-y),x-y).
\end{equation}
Here $\mascF _2F$ is the partial Fourier transform of $F(x,y)\in
\maclS _s'(\rr {2d})$ with respect to the $y$ variable. This
definition makes sense, since the mappings
\begin{equation}\label{homeoF2tmap}
\mascF _2\quad \text{and}\quad F(x,y)\mapsto F(x-A(x-y),y-x)
\end{equation}
are homeomorphisms on $\maclS _s'(\rr {2d})$.
In particular, the map $a\mapsto K_{a,A}$ is a homeomorphism on
$\maclS _s'(\rr {2d})$.

\par

\begin{rem}\label{Rem:KernelThm}
For any $K\in \maclS '_s(\rr {d_1+d_2})$, let $T_K$ be the
linear and continuous mapping from $\maclS _s(\rr {d_1})$
to $\maclS _s'(\rr {d_2})$, defined by the formula
\begin{equation}\label{pre(A.1)}
(T_Kf,g)_{L^2(\rr {d_2})} = (K,g\otimes \overline f )_{L^2(\rr {d_1+d_2})}.
\end{equation}
It is well-known that if $t\in \mathbf R$, then it follows from e.{\,}g.
\cite{LozPerTask,ChSiTo}
that the Schwartz kernel theorem also holds in the context of Gelfans-Shilov spaces.
That is, the mappings $K\mapsto T_K$ and $a\mapsto \op _t(a)$ are bijective
from $\maclS _s '(\rr {2d})$
to the set of linear and continuous mappings from $\maclS _s(\rr d)$ to
$\maclS _s'(\rr d)$. Similar facts hold true if $\maclS _s$ and $\maclS _s'$
are replaced by $\Sigma _s$ and $\Sigma _s'$, respectively (or by 
$\mascS$ and $\mascS '$, respectively).
\end{rem}

\par

As a consequence of Remark \ref{Rem:KernelThm} it follows that
for each $a_1\in \maclS _s '(\rr {2d})$ and $A_1,A_2\in
\GL (d,\mathbf R)$, there is a unique $a_2\in \maclS _s '(\rr {2d})$ such that
$\op _{A_1}(a_1) = \op _{A_2} (a_2)$. The relation between $a_1$ and $a_2$
is given by
\begin{equation}\label{calculitransform}
\op _{A_1}(a_1) = \op _{A_2}(a_2) \qquad \Leftrightarrow \qquad
a_2(x,\xi )=e^{i\scal {(A_1-A_2)D_\xi}{D_x }}a_1(x,\xi ).
\end{equation}
(Cf. \cite{Ho1}.) Note here that the right-hand side makes sense, since
it is equivalent to $\widehat a_2(\xi ,x)
= e^{i(A_1-A_2)\scal x \xi }\widehat a_1(\xi ,x)$,
and that the map $a\mapsto e^{i\scal {Ax} \xi }a$ is continuous on
$\maclS _s '$ when $A\in \GL (d,\mathbf R)$.

\par

Let $A\in \GL (d,\mathbf R)$ and $a\in \maclS _s '(\rr {2d})$ be
fixed. Then $a$ is called a rank-one element with respect to $A$, if
the corresponding pseudo-differential operator is of rank-one,
i.{\,}e.
\begin{equation}\label{trankone}
\op _A(a)f=(f,f_2)f_1, \qquad f\in \maclS _s(\rr d),
\end{equation}
for some $f_1,f_2\in \maclS _s '(\rr d)$. By
straightforward computations it follows that \eqref{trankone}
is fulfilled, if and only if $a=(2\pi
)^{d/2}W_{f_1,f_2}^{A}$, where $W_{f_1,f_2}^{A}$
it the $A$-Wigner distribution defined by the formula
\begin{equation}\label{wignertdef}
W_{f_1,f_2}^{A}(x,\xi ) \equiv \mascF (f_1(x+A\cdo
)\overline{f_2(x-(I-A)\cdo )} )(\xi ),
\end{equation}
which takes the form
$$
W_{f_1,f_2}^{A}(x,\xi ) =(2\pi )^{-d/2} \int
f_1(x+Ay)\overline{f_2(x-(I-A)y) }e^{-i\scal y\xi}\, dy,
$$
when $f_1,f_2\in \maclS _s (\rr d)$. Here $I\in \GL (d,\mathbf R)$ is
the identity matrix. By combining these facts
with \eqref{calculitransform} it follows that
\begin{equation}\label{wignertransf}
W_{f_1,f_2}^{A_2} = e^{i\scal {{\color{red}{(A_1-A_2)}}D_{\xi}}{D_x }}
W_{f_1,f_2}^{A_1},
%
%
%
\end{equation}
for each $f_1,f_2\in \maclS _s '(\rr d)$ and $A_1,A_2\in \GL (d,\mathbf R)$. Since
the Weyl case is particularly important, we set
$W_{f_1,f_2}^{A}=W_{f_1,f_2}$ when $A=\frac 12I$, i.{\,}e.
$W_{f_1,f_2}$ is the usual (cross-)Wigner distribution of $f_1$ and
$f_2$.

\par

For future references we note the link
\begin{multline}\label{tWigpseudolink}
(\op _A(a)f,g)_{L^2(\rr d)} =(2\pi )^{-d/2}(a,W_{g,f}^A)_{L^2(\rr {2d})},
\\[1ex]
a\in \maclS _s'(\rr {2d}) \quad\text{and}\quad f,g\in \maclS _s(\rr d)
\end{multline}
between pseudo-differential operators and Wigner distributions,
which follows by straightforward computations (see also e.{\,}g.
\cite{To24}).

\medspace

Next we discuss the Weyl
product, the twisted convolution and related objects. Let
$s\ge 1/2$ and let $a,b\in
\maclS _s '(\rr {2d})$. Then the Weyl product $a\wpr
b$ between $a$
and $b$ is the function or distribution which fulfills $\op ^w(a\wpr
b) = \op ^w(a)\circ \op ^w(b)$, provided the right-hand side
makes sense as a continuous operator from $\maclS _s(\rr d)$ to
$\maclS _s'(\rr d)$. More generally, if $A\in \GL (d,\mathbf R)$,
then the product $\wpr _A$ is defined by the formula
\begin{equation}\label{wprtdef}
\op _A(a\wpr _A b) = \op _A(a)\circ \op _A(b),
\end{equation}
provided the right-hand side makes sense as a continuous operator from
$\maclS _s (\rr d)$ to $\maclS _s '(\rr d)$, in which case $a$ and $b$
are called \emph{suitable} or \emph{admissible}. We also use the notation
$\wpr$ instead of $\wpr _A$ when $A= \frac 12 I$ (i.{\,}e. in the Weyl case).

\par

The Weyl product can also, in a convenient way, be expressed in terms
of the symplectic Fourier transform and the twisted convolution. More
precisely, let $s\ge 1/2$. Then the \emph{symplectic Fourier transform} for $a \in
\maclS _s (\rr {2d})$ is defined by the formula
\begin{equation*}
(\mascF _\sigma a) (X)
= \pi^{-d}\int a(Y) e^{2 i \sigma(X,Y)}\,  dY,
\end{equation*}
where $\sigma$ is the symplectic form given by
$$
\sigma(X,Y) = \scal y \xi -
\scal x \eta ,\qquad X=(x,\xi )\in \rr {2d},\ Y=(y,\eta )\in \rr {2d}.
$$

\par

We note that $\mascF _\sigma = T\circ (\mascF \otimes (\mascF ^{-1}))$, when
$(Ta)(x,\xi) =a(\xi ,x)$. In particular, ${\mascF _\sigma}$ is continuous on
$\maclS _s (\rr {2d})$, and extends uniquely to a homeomorphism on
$\maclS _s '(\rr {2d})$, and to a unitary map on $L^2(\rr {2d})$, since similar
facts hold for $\mascF$. Furthermore, $\mascF _\sigma^{2}$ is the identity
operator.

\par

Let $s\ge 1/2$ and $a,b\in \maclS _s (\rr {2d})$. Then the \emph{twisted
convolution} of $a$ and $b$ is defined by the formula
\begin{equation}\label{twist1}
(a \ast _\sigma b) (X)
= \left (\frac 2\pi \right )^{\frac d2}
\int a(X-Y) b(Y) e^{2 i \sigma(X,Y)}\, dY.
\end{equation}
The definition of $*_\sigma$ extends in different ways. For example,
it extends to a continuous multiplication on $L^p(\rr {2d})$ when $p\in
[1,2]$, and to a continuous map from $\maclS _s '(\rr {2d})\times
\maclS _s (\rr {2d})$ to $\maclS _s '(\rr {2d})$. If $a,b \in
\maclS _s '(\rr {2d})$, then $a \wpr b$ makes sense if and only if $a
*_\sigma \widehat b$ makes sense, and then
\begin{equation}\label{tvist1}
a \wpr b = (2\pi)^{-\frac d2} a \ast_\sigma (\mascF _\sigma {b}).
\end{equation}
We also remark that for the twisted convolution we have
\begin{equation}\label{weylfourier1}
\mascF _\sigma (a *_\sigma b) = (\mascF _\sigma a) *_\sigma b =
\check{a} *_\sigma (\mascF _\sigma b),
\end{equation}
where $\check{a}(X)=a(-X)$ (cf. \cite{To3,To10,To16}). A
combination of \eqref{tvist1} and \eqref{weylfourier1} gives
\begin{equation}\label{weyltwist2}
\mascF _\sigma (a\wpr b) = (2\pi )^{-\frac d2}(\mascF _\sigma
a)*_\sigma (\mascF _\sigma b).
\end{equation}

We now define the subspace of symbols in $\mascS '(\rr {2d})$ which give rise to
$L^2(\rr d)$-bounded pseudo-differential operators,
which will be useful in the sequel.

\par

\begin{defn}\label{def:swinf}
	The set $s_\infty ^w(\rr {2d})$ consists of all
	$a\in \mascS '(\rr {2d})$ such that $\op ^w(a)$ is linear and continuous on
	$L^2(\rr d)$, and we set
	$$
		\nm a{s_\infty ^w}\equiv \nm {\op ^w(a)}{L^2(\rr d)\to L^2(\rr d)}.
	$$
\end{defn}

\par 

\begin{rem}
	By Remark \ref{Rem:KernelThm} it follows that the map
	$a\mapsto \op ^w(a)$ is an isometric bijection from
	$s_\infty ^w(\rr {2d})$ to the set of linear continuous operators on
	$L^2(\rr d)$.
\end{rem}

\par

\begin{rem}
We remark that the relations in this subsection hold true after
$\maclS _s$, $\maclS _s'$ and $s\ge \frac 12$ are replaced by
$\Sigma _s$, $\Sigma _s'$ and $s>\frac 12$ respectively, in each
place.
\end{rem}

\par

Next we recall some algebraic properties and characterisations
of $\Gamma ^{(\omega )}_s(\rr {2d})$ and $\Gamma ^{(\omega )}_{0,s}(\rr {2d})$, and
begin with the following. We refer to \cite{CaTo} for its proof.

\par

\begin{prop}\label{Prop:Gamma(omega)}
Let $s\ge 1$, $\omega _j\in \mascP _{E,s}^0(\rr {2d})$, $A_j\in
\GL (d, \mathbf R)$ for $j=1,2$, and let
$\omega _{0,r}(X,Y)=\omega _0(X)e^{-r|Y|^{\frac 1s}}$ when
$r> 0$. Then the following is true:
\begin{enumerate}
\item If $a_1,a_2\in \Sigma _s'(\rr {2d})$ satisfy $\op _{A_1}(a_1)=\op
_{A_2}(a_2)$, then $a_1\in \Gamma ^{(\omega _0)}_{s}(\rr {2d})$ if and only if
$a_2\in \Gamma ^{(\omega _0)}_{s}(\rr {2d})$.

\vrum

\item $\Gamma ^{(\omega _1)}_{s}\wpr \Gamma ^{(\omega _2)}_{s} \subseteq
\Gamma ^{(\omega _1\omega _2)}_{s}$.

\vrum

\item $\displaystyle {\Gamma ^{(\omega _0)}_{s} =\bigcup _{r> 0}M^{\infty
,1}_{(1/\omega _{0,r})}=\bigcup _{r\ge 0}\splM ^{\infty
,1}_{(1/\omega _{0,r})}}$.
\end{enumerate}
\end{prop}

\par

\begin{prop}\label{Prop:Gamma(omega)B}
Let $s\ge 1$, $\omega _j\in \mascP _{E,s}(\rr {2d})$, $A_j\in
\GL (d, \mathbf R)$ for $j=1,2$, and let
$\omega _{0,r}(X,Y)=\omega _0(X)e^{-r|Y|^{\frac 1s}}$ when
$r> 0$. Then the following is true:
\begin{enumerate}
\item If $a_1,a_2\in \Sigma _s'(\rr {2d})$ satisfy $\op _{A_1}(a_1)=\op
_{A_2}(a_2)$, then $a_1\in \Gamma ^{(\omega _0)}_{0,s}(\rr {2d})$ if and only if
$a_2\in \Gamma ^{(\omega _0)}_{0,s}(\rr {2d})$.

\vrum

\item $\Gamma ^{(\omega _1)}_{0,s}\wpr \Gamma ^{(\omega _2)}_{0,s} \subseteq
\Gamma ^{(\omega _1\omega _2)}_{0,s}$.

\vrum

\item $\displaystyle {\Gamma ^{(\omega _0)}_{0,s} =\bigcap _{r> 0}M^{\infty
,1}_{(1/\omega _{0,r})}=\bigcap _{r\ge 0}\splM ^{\infty
,1}_{(1/\omega _{0,r})}}$.
\end{enumerate}
\end{prop}

\par

In time-frequency analysis one also considers mapping properties for 
pseudo-differential
operators between  modulation spaces  or with
symbols in  modulation spaces. Especially we need the following two
results, where the first
one is a generalisation of \cite[Theorem 2.1]{Ta} by Tachizawa, and the
second one is a
weighted version of \cite[Theorem 14.5.2]{Gc2}. We refer to \cite{To28}
for the proof of the first two propositions and to \cite{To28} for
the proof of the third one.

\par

\begin{prop}\label{Prop:p3.2}
Assume that $A\in \GL (d,\mathbf R)$, $s\ge 1$, $\omega ,\omega _0\in \mascP
_{E,s}^0(\rr {2d})$, $a\in
\Gamma ^{(\omega )}_{s}(\rr {2d})$, and that $\mathscr B$
is an invariant BF-space on $\rr {2d}$ of Beurling type. Then
$\op _A(a)$ is continuous from $M(\omega _0\omega ,\mathscr
B)$ to $M(\omega _0 ,\mathscr B)$, and also continuous on
$\maclS _s(\rr{d})$ and on $\maclS _s'(\rr{d})$.
\end{prop}

\par

\begin{prop}\label{Prop:p3.2B}
Assume that $A\in \GL (d,\mathbf R)$, $s\ge 1$, $\omega ,\omega _0\in \mascP
_{E,s}(\rr {2d})$, $a\in
\Gamma ^{(\omega )}_{0,s}(\rr {2d})$, and that $\mathscr B$
is an invariant BF-space on $\rr {2d}$ of Roumieu type. Then
$\op _A(a)$ is continuous from $M(\omega _0\omega ,\mathscr
B)$ to $M(\omega _0 ,\mathscr B)$, and also continuous on
$\Sigma _s(\rr{d})$ and on $\Sigma _s'(\rr{d})$.
\end{prop}

\par

\begin{prop}\label{Prop:pseudomod}
Assume that $p,q\in (0,\infty ]$, $r\le \min (p,q,1)$,
$\omega \in \mathscr P _{E}(\rr {2d}\oplus \rr {2d})$ and $\omega
_1,\omega _2\in \mathscr P_{E}(\rr {2d})$ satisfy
\begin{equation}\label{e5.9}
\frac {\omega _2(X-Y)}{\omega _1
(X+Y)} \le C \omega (X ,Y), \qquad \quad X,Y \in \rr{2d}, 
\end{equation}
for some constant $C$. If $a\in \splM ^{\infty ,r}_{(\omega )}(\rr
{2d})$, then $\op ^w(a)$ extends uniquely to a continuous map from
$M^{p,q}_{(\omega _1)}(\rr d)$ to $M^{p,q}_{(\omega _2)}(\rr d)$.
\end{prop}

\par

Finally we need the following result concerning mapping properties of
modulation spaces under the Weyl product. The result is a special case
of Theorem 0.3 in \cite{CoToWa}.

\par

\begin{prop}\label{Prop:Weylprodmod}
Assume that $\omega _j\in \mathscr P_{E}(\rr {2d}\oplus \rr {2d})$
for $j=0,1,2$ satisfy
\begin{equation}\label{Eq:weightprodmod}
\omega _0(X,Y)\le C\omega _1(X-Y+Z,Z)\omega _2(X+Z,Y-Z),
\end{equation}
for some constant $C>0$  independent of $X,Y,Z\in \rr {2d}$. Then the
map $(a,b)\mapsto a\wpr b$ from $\Sigma _1(\rr {2d})\times \Sigma _1(\rr {2d})$
to $\Sigma _1(\rr {2d})$ extends uniquely to a mapping  from $\splM
^{\infty ,1}_{(\omega _1)}(\rr {2d})\times \splM ^{\infty
  ,1}_{(\omega _2)}(\rr {2d})$ to $\splM ^{\infty ,1}_{(\omega
  _0)}(\rr {2d})$. 
\end{prop}

\par

In the proof of our main theorem we will need the following
consequence of Proposition~\ref{Prop:Weylprodmod}.

\par

\begin{prop} \label{Prop:CorWeyl}
Let $s\ge 1$, $\omega _0, v_0, v_1 \in \mathscr{P}_{E,s}^0
(\rr {2d}\oplus \rr {2d})$ be such that $\omega _0$ is $v_0$-moderate.
Set $\vartheta = \omega _0^{1/2}$, and  
\begin{eqnarray}
\omega _1(X,Y)&=&\frac{v_0(2Y)^{1/2}v_1(2Y)}{\vartheta (X+Y) \vartheta(X-Y)} \, , \notag
\\[1ex]
\omega _2(X,Y) &=&\vartheta (X-Y)\vartheta(X+Y)v_1(2Y) \, ,
\notag
\\[1ex]
v_2(X,Y) &= & v_1(2Y)\, .
\end{eqnarray}
Then 
\begin{eqnarray}
  \label{Eq:hn1}
  \Gamma ^{(1/\vartheta )}_{s} \wpr \mathcal{M}^{\infty , 1} _{(\omega _1)}
  \wpr \Gamma ^{(1/\vartheta )}_{s} & \subseteq & \mathcal{M}^{\infty , 1}
  _{(v_2)}\, , \label{Eq:ch123}
  \\[1ex]
  \Gamma ^{(1/\vartheta )}_{s} \wpr \mathcal{M}^{\infty , 1} _{(v_2)}
  \wpr \Gamma ^{(1/\vartheta )}_{s} & \subseteq & \mathcal{M}^{\infty , 1}
  _{(\omega _2)} \, . \label{Eq:ch124}
\end{eqnarray}
\par

The same holds with $\mascP _{E,s}$ and $\Gamma _{0,s}^{(1/\vartheta )}$
in place of $\mascP _{E,s}^0$ and $\Gamma _{s}^{(1/\vartheta )}$
respectively, at each occurrence.
\end{prop}

\par

\begin{proof}
Since $\Gamma ^{(1/\vartheta )}_{s} = \bigcup _{r\ge 0}\splM ^{\infty
,1}_{(\vartheta _r)} $ with $\vartheta _r(X,Y) = \vartheta (X)
e^{r|Y|^{\frac 1s}}$ (Proposition~\ref{Prop:Gamma(omega)}(3)), it suffices
to argue with the symbol class $\splM ^{\infty , 1}
_{(\vartheta _r)}$ for some sufficiently large $r$ instead of
$\Gamma ^{(1/\vartheta )}_{s}$.

\par

Introducing the intermediate weight  
$$
\omega _3(X,Y) = \frac {v_1(2Y)\vartheta (X+Y)}{\omega _0(X-Y)}.
$$
we will show that for suitable $r$
\begin{align}
\omega _3(X,Y) &\le C\omega _1(X-Y+Z,Z)\vartheta
_r(X+Z,Y-Z)\label{Eq:omegacond2}
\\[1ex] 
v_1(2Y) &\le C\vartheta _r(X-Y+Z,Z)\omega
_3(X+Z,Y-Z)\label{Eq:omegacond3} \, .
\end{align}
Proposition~\ref{Prop:Weylprodmod} applied to \eqref{Eq:omegacond2} shows that
$\splM ^{\infty ,1} _{(\omega _1)} \wpr \Gamma ^{(1/\vartheta )}_{s} \subseteq
\splM ^{\infty ,1} _{(\omega _3)}$, and \eqref{Eq:omegacond3} implies  that
$\Gamma ^{(1/\vartheta )}_{s}  \wpr \splM ^{\infty ,1} _{(\omega _3)} \subseteq
\splM ^{\infty ,1} _{(v_2)}$,  and \eqref{Eq:ch123} follows. 

\par

Since $\vartheta $ is $v_0 ^{1/2}$-moderate and $v_0\in \mascP _{E,s}^0$,
we have
$$
\vartheta (X-Y)^{-1} \leq v_0(2Z)^{1/2} \vartheta (X-Y+2Z)^{-1}
\quad \text{and}\quad
\vartheta (X+Y) \leq \vartheta (X+Z) e^{r|Y-Z|^{\frac 1s}}
$$
for suitable $r> 0$. Using these inequalities
repeatedly in the following, a  straightforward computation yields
\begin{align*}
\omega _3(X,Y) & = \frac {v_1(2Y)\vartheta (X+Y)}{\vartheta (X-Y)^2}
\\[1ex]
& \le C_1 \frac {v_0(2Z) ^{1/2}v_1(2Z)\vartheta (X+Z) 
e^{r|Y-Z|^{\frac 1s}}}{\vartheta(X-Y+2Z)\vartheta(X-Y)}
\\[1ex]
& =C_1 \omega _1(X-Y+Z,Z)\vartheta _r(X+Z,Y-Z),
\end{align*}
for some $C_1>0$ and $r>0$.

\par

Likewise we obtain 
\begin{align*}
v_1(2Y) &= \frac {\vartheta(X-Y) v_1(2Y)\vartheta
  (X-Y)}{\vartheta(X-Y) ^2}
\\[1ex]
&\le C_1 \frac {\vartheta(X-Y) v_0(2Y)^{1/2}v_1(2Y)\vartheta (X+Y)}{\vartheta(X-Y)^2}
\\[1ex]
&\le C_2 \frac {\vartheta(X-Y+Z)
e^{r|Z|^{\frac 1s}}v_0(2(Y-Z))^{1/2}v_1(2(Y-Z))\vartheta(X+Y)}{\vartheta(X-Y+2Z)^2}
\\[1ex]
& = C_2 \vartheta _r(X-Y+Z,Z)\omega _3(X+Z,Y-Z) \, .
\end{align*}

\par

The twisted convolution relation \eqref{Eq:ch124} is proved similarly. Let 
$$
\omega _4(X,Y) = \vartheta(X-Y)v_1(2Y)
$$
be the intermediate weight. Then the inequality 
\begin{align*}
\omega _4(X,Y) & = \vartheta(X-Y)v_1(2Y) \le C\vartheta (X-Y+Z)
e^{r|Z|^{\frac 1s}}v_1(2(Y-Z)) 
\\[1ex]
&= C \vartheta _r(X-Y+Z,Z)v_2(X+Z,Y-Z)
\end{align*}
implies that $\Gamma ^{(1/\vartheta   )} _{s}\wpr \splM ^{\infty ,1} _{(v_2)}
\subseteq  \splM ^{\infty ,1} _{(\omega _4)}$. 

\par

Similarly we obtain
\begin{align*}
\omega _2(X,Y) &\le C \vartheta(X-Y)v_1(2Z)\vartheta (X+Z)e^{r|Z-Y|^{\frac 1s}}
\\[1ex]
&= C\omega _4(X-Y+Z,Z)\vartheta(X+Z)e^{r|Z-Y|^{\frac 1s}}
\\[1ex]
& = C\omega _4(X-Y+Z,Z)\vartheta _r(X+Z,Y-Z),
\end{align*}
and thus $\mathcal  M^{\infty ,1} _{(\omega _4)} \wpr \Gamma ^{(1/\vartheta  )}_{s}
\subseteq \mathcal  M^{\infty ,1} _{(\omega _2)}$.

\par

The case $\mascP _{E,s}$ and $\Gamma _{0,s}^{(1/\vartheta )}$
in place of $\mascP _{E,s}^0$ and $\Gamma _{s}^{(1/\vartheta )}$
respectively, at each occurrence, is treated in similar ways and is
left for the reader.
\end{proof}

\par

\subsection{The Wiener Algebra Property}

\par


\par

As a further crucial tool in our study of  the isomorphism property of
Toeplitz operators we need to combine these continuity results with
convenient invertibility properties. The so-called Wiener algebra property of
certain symbol classes asserts that the inversion of a pseudo-differential
operator preserves the symbol class and is often referred to as the
spectral invariance of a symbol class.

\par

\begin{prop}\label{Thm:specinv}
Let $A\in \GL (d,\mathbf R)$. Then the following are true:
\begin{enumerate}
\item If $s>1$, $a\in \Gamma ^{(1)}_{0,s}(\rr {2d})$ and $\op _A(a)$ is invertible
on $L^2(\rr d)$, then $\op _A(a)^{-1}=\op _A(b)$ for some $b \in \Gamma ^{(1)}_{0,s}(\rr {2d})$. 

\vrum

\item If $s\ge 1$, $a\in \Gamma ^{(1)}_{s}(\rr {2d})$ and $\op _A(a)$ is invertible
on $L^2(\rr d)$, then $\op _A(a)^{-1}=\op _A(b)$ for some $b \in \Gamma ^{(1)}_{s}(\rr {2d})$. 

\vrum

\item If $s\ge 1$, $v_0\in \mathscr P_{E,s}^0(\rr {2d})$ is submultiplicative, $v(X,Y)\equiv v_0(Y)$,
$X,Y\in \rr {2d}$, $a\in M^{\infty ,1}_{(v)}(\rr {2d})$ and  $\op _A(a)$ is invertible on
$L^2(\rr d)$, then $\op _A(a)^{-1}=\op _A(b)$, for some
$b\in M^{\infty ,1}_{(v)}(\rr {2d})$.

\vrum

\item If $s> 1$, $v_0\in \mathscr P_{E,s}(\rr {2d})$ is submultiplicative, $v(X,Y)\equiv v_0(Y)$,
$X,Y\in \rr {2d}$, $a\in M^{\infty ,1}_{(v)}(\rr {2d})$ and  $\op _A(a)$ is invertible on
$L^2(\rr d)$, then $\op _A(a)^{-1}=\op _A(b)$, for some
$b\in M^{\infty ,1}_{(v)}(\rr {2d})$.
\end{enumerate}
\end{prop}

\par

\begin{proof}
The results follows essentially from \cite[Corollary 5.5]{Gc3} or \cite{Gc4}. More precisely,
Suppose $s>1$, $a\in \Gamma ^{(1)}_{s}(\rr {2d})$, $\op _A(a)$ is invertible
on $L^2(\rr d)$, and let $v_r(X,Y)=e^{r|Y|^{frac 1s}}$ when $r\ge 0$. Then
$a\in M^{\infty,1}_{(v_r)}(\rr {2d})$ for some $r>0$. By \cite[Corollary 5.5]{Gc3},
$\op (M^{\infty,1}_{(v_r)}(\rr {2d}))$ is a Wiener algebra, giving that
$\op (a)^{-1}=\op (b)$ for some $b\in M^{\infty ,1}_{(v_r)}(\rr {2d})\subseteq
\Gamma ^{(1)}_{s}(\rr {2d})$. This gives (2) in the case $s>1$.

\par

If instead $s=1$, then by \cite[Theorem 4.4]{FeGaTo} there is an $r_0>0$ such that
$\op (a)^{-1}=\op (b)$ for some $b\in M^{\infty ,1}_{(v_{r_0})}(\rr {2d})\subseteq
\Gamma ^{(1)}_{1}(\rr {2d})$,
and (2) follows for general $s\ge 1$.

\par

By similar arguments, (1), (3) and (4) follow. The details are left for the reader.
%
\end{proof}

\par

\subsection{Toeplitz Operators.}  Fix a symbol  $a\in
\mathscr S(\rr {2d})$ and a window $\fy \in \mathscr S(\rr d)$. Then the
Toeplitz operator $\tp _{\fy}(a)$ is defined by the formula
\begin{equation}\label{toeplitz}
(\tp _{\fy}(a)f_1,f_2)_{L^2(\rr d)} = (a\, V_{
\fy}f_1,V_{ \fy}f_2)_{L^2(\rr {2d})}\, ,
\end{equation}
when $f_1,f_2\in \mathscr S(\rr d)$. Obviously, $\tp _{\fy}(a)$ is
well-defined and extends uniquely to a  continuous operator from $\mathscr S'(\rr d)$ to $\mathscr
S(\rr d)$.

\par

The definition of Toeplitz operators can be extended to more general
classes of windows and symbols by using appropriate estimates for the 
short-time Fourier transforms in \eqref{toeplitz}. 

\par

We state two possible ways of extending \eqref{toeplitz}. 
The first result follows from \cite[Corollary
4.2]{CG1} and its proof, and the second result is a special case
of \cite[Theorem 3.1]{TB}. We use the notation $\mathcal
L(V_1,V_2)$ for the set of linear and continuous mappings from the
topological vector space  $V_1$ into the topological vector space
$V_2$. We also set
\begin{equation}\label{omega0t}
  \omega _{0,t}(X,Y)=v_1(2Y)^{t-1}\omega _0(X) \qquad \text{ for } X,Y
  \in \rr {2d} \, .
\end{equation}

\par

\begin{prop}\label{Tpcont1}
Let $0\le t\le 1$, $p,q\in [1,\infty]$, and  $\omega ,\omega _0,v_1,v_0\in
\mathscr P(\rr {2d})$ be such that $v_0$ and $v_1$ are submultiplicative,
$\omega _0$ is $v_0$-moderate and $\omega $ is $v_1$-moderate. Set 
\begin{equation*}
v=v_1^tv_0 \quad \text{and}\quad \vartheta = \omega _0^{1/2} \,,  
\end{equation*}
and  let  $\omega _{0,t}$ be as in \eqref{omega0t}. Then the following
are true:

\par

\begin{enumerate}
\item The definition of $(a,\fy )\mapsto \tp _\fy (a)$ from $\mathscr
S(\rr {2d})\times \mathscr S(\rr d)$ to $\mathcal L(\mathscr S(\rr
d),\mathscr S'(\rr d))$ extends uniquely to a continuous map from
$\mathcal M^\infty _{(1/\omega _{0,t})}(\rr {2d})\times M^1_{(v)}(\rr
d)$ to $\mathcal L(\mathscr S(\rr d),\mathscr S'(\rr d))$.

\vrum

\item If $\fy \in M^{1}_{(v )}(\rr d)$ and $a\in \mathcal
M^{\infty }_{(1/\omega _{0,t})}(\rr {2d})$, then $\tp _\fy (a)$
extends uniquely to a continuous map from $M_{(\vartheta \omega
)}^{p,q}(\rr d)$ to $M_{(\omega /\vartheta )}^{p,q}(\rr d)$.
\end{enumerate}
\end{prop}

\par

\begin{prop}\label{Tpcont2}
Let $\omega ,\omega _1,\omega _2, v\in \mathscr P(\rr
{2d})$ be such that $\omega _1$ is $v$-moderate, $\omega_2$ is
$ v$-moderate and $\omega =\omega _1/\omega _2$. Then the
following are true:
\begin{enumerate}
\item The mapping $(a,\fy )\mapsto \tp _\fy (a)$ extends uniquely to
a continuous map from
$L^\infty _{(\omega )}(\rr {2d})\times M^2_{(v)}(\rr d)$ to
$\mathcal L(\mathscr S(\rr d),\mathscr S'(\rr d))$.

\vrum

\item If $a\in L^\infty _{(1/\omega )}(\rr {2d})$ and $\fy \in
M^2_{(v)}(\rr d)$, then  $\tp _\fy (a)$ extends uniquely to a
continuous operator from $M^2_{(\omega _1)}(\rr d)$ to
$M^2_{(\omega _2)}(\rr d)$.
\end{enumerate}
\end{prop}

\par

\subsection{Weyl formulation of Toeplitz operators.} We finish
this section by recalling some important relations
between Weyl operators, Wigner distributions, and Toeplitz
operators. For instance, the Weyl symbol of a Toeplitz operator
is the convolution between the Toeplitz symbol and a Wigner
distribution. More precisely,  if  $a\in \mathscr S(\rr {2d})$ and
$\fy \in \mathscr S(\rr d)$, then
\begin{equation}
\label{toeplweyl}
\tp _\fy (a) = (2\pi )^{-d/2}\op ^w(a*W_{\fy  ,\fy} ) \, .  
\end{equation}

\par

Our analysis of Toeplitz operators is
based on the pseudo-differential operator representation  given
by \eqref{toeplweyl}. Furthermore, any extension of the definition of
Toeplitz operators to cases which are not covered by Propositions
\ref{Tpcont1} and \ref{Tpcont2} is based on this representation. Here
we remark that this leads to situations were certain mapping
properties for the pseudo-differential operator representation make
sense, whereas  similar interpretations are difficult or impossible to
make in the framework of \eqref{toeplitz} (see Remark
\ref{extensionremark} in Section \ref{sec2}). We refer to \cite{To9}
or Section \ref{sec2} for extensions of Toeplitz operators in context
of pseudo-differential operators.

\par

\section{Confinement of the symbol classes  $\Gamma ^{(\omega )}_s(\rr {d})$ and
$\Gamma ^{(\omega)}_{0,s}(\rr {d})$} \label{sec2}

\par

In this section
we introduce and discuss basic properties for confinements for
symbols in $\Gamma ^{(\omega _0)}_{s}$ and in $\Gamma ^{(\omega _0)}_{0,s}$.
These considerations follow lines similar to the discussions in \cite{BoCh,Le2},
but are here adapted to symbols that possess Gevrey regularity. 
In particular, this requires the deduction of various types of
delicate estimates, concerning the compositions of symbols that are confined in different ways.

\par

We recall that if $\sigma$ is the (standard) symplectic form on $\rr {2d}$, namely,
$$
\sigma (X,Y) = \scal y\xi -\scal x\eta ,\qquad X=(x,\xi )\in \rr {2d},\ Y=(y,\eta )\in \rr {2d},
$$
then the \emph{symplectic Fourier transform} $\mascF _\sigma$ is defined by
$$
(\mascF _\sigma a)(X) = \widehat a(X)\equiv \pi ^{-d}\int a(Y)e^{2i\sigma (X,Y)}\, dY,
$$
and the twisted convolution $*_\sigma$ is given by
$$
(a*_\sigma b)(X)\equiv \left ( \frac 2\pi \right ) ^{\frac d2}\int a(X-Y)b(Y)e^{2i\sigma (X,Y)}\, dY,
$$
for suitable $a,b\in \maclS _{1/2}'(\rr {2d})$. The twisted convolution is linked to the Weyl product
by the formula
$$
a\wpr b = (2\pi )^{-\frac d2} a*_\sigma (\mascF _\sigma b),
$$
hence,
$$
(a\wpr b)(X) = \pi ^{-d}\int a(X-Y)\widehat b(Y)e^{2i\sigma (X,Y)}\, dY.
$$
We also note that $\mascF _\sigma ^2$ is the identity map.

\par

In what follows we let $a_Y = a(\cdo -Y)$ when $a\in \maclS _{1/2}'(\rr {2d})$ and $Y\in \rr {2d}$,
and in analogous ways, $b_Y$, $\phi _Y$, $\fy _Y$, $\psi _Y$ etc. are defined when $b,\phi
,\fy ,\psi \in \maclS _{1/2}'(\rr {2d})$. For admissible $a$ and $b$ we have
\begin{equation}\label{Eq:TranslWeylPr}
(a\wpr b)_Y = a_Y\wpr b_Y .
\end{equation}
We also recall that if $\fy \in \maclS _s(\rr {2d})$, then there are functions
$\phi ,\psi \in \maclS _s(\rr {2d})$ such that $\fy =\phi \wpr \psi $
(cf. \cite{ChSiTo,ToKrNiNo}). The same is true
if $\maclS _s$ is replaced by $\Sigma _s$ or by $\mascS$. In particular,
by choosing $\fy$ such that
$\int \fy (X)\, dX=1$, \eqref{Eq:TranslWeylPr} gives the following.

\par

\begin{prop}
Let $s\ge \frac 12$. Then there are $\phi ,\psi \in \maclS _s(\rr {2d})$ such that
\begin{equation}\label{eq:pu}
	\int _{\rr {2d}} \psi _Y\wpr\phi _Y \, dY=1.
\end{equation}
The same holds true with $\Sigma _s$ or $\mascS$ in place of $\maclS _s$,
provided $s>\frac{1}{2}$.
\end{prop}

\par

For independent translations in Weyl products we have the following.

\par

\begin{prop}\label{Prop:DilWeylprod}
Let $s\ge \frac 12$ and let $\phi ,\psi \in \maclS _s(\rr {2d})$. Then
$$
(\phi _Y\wpr \psi _Z)(X) = \Psi (X-Y,X-Z)
$$
for some $\Psi \in \maclS _s(\rr {2d}\times \rr {2d})$.
The same holds true with $\Sigma _s$ or $\mascS$ in place of $\maclS _s$.
\end{prop}

\par

\begin{proof}
We only prove the result when $\phi ,\psi \in \maclS _s(\rr {2d})$. The other
cases follow by similar arguments and are left for the reader.

\par

We have
\begin{multline*}
(\phi _Y\wpr \psi _Z)(X) = \pi ^{-d}\int \phi (X-Y-Y_1)\widehat \psi (Y_1)
e^{2i\sigma (Y_1,Z)} e^{2i\sigma (X,Y_1)}\, dY_1
\\[1ex]
=
\pi ^{-d}\int \phi ((X-Y)-Y_1)\widehat \psi (Y_1)
e^{2i\sigma (X-Z,Y_1)}\, dY_1
=\Psi (X-Y,X-Z),
\end{multline*}
where
$$
\Psi (X,Z)= \pi ^{-d}\int \phi (X-Y_1)\widehat \psi (Y_1)
e^{2i\sigma (Z,Y_1)}\, dY_1.
$$
We note that
$$
\Psi =(\mascF _{\sigma ,2} \circ T)(\phi \otimes \widehat \psi ),
$$
where $(T\Phi )(X,Z)= \Phi (X-Z,Z)$ when $\Phi \in \maclS _s(\rr {2d}\times \rr {2d})$,
and $\mascF _{\sigma ,2}\Phi$ is the partial symplectic Fourier transform of
$\Phi (X,Z)$ with respect to the $Z$ variable.
Since $(\phi ,\psi)\mapsto \phi \otimes \widehat \psi$ is continuous from
$\maclS _s(\rr {2d})\times \maclS _s(\rr {2d})$ to $\maclS _s(\rr {2d}\times \rr {2d})$,
and $T$ and $\mascF _{\sigma ,2}\Phi$ are continuous on $\maclS _s(\rr {2d}\times \rr {2d})$,
it follows that $\Psi \in \maclS _s(\rr {2d}\times \rr {2d})$.
\end{proof}

\par

Since $\Psi $ in Proposition \ref{Prop:DilWeylprod} belongs to similar types of spaces as $\phi$
and $\psi$, it follows that estimates of the form
$$
|D^\alpha \Psi (X,Y)| \lesssim h^{|\alpha |}\alpha !^s e^{-(|X|^{\frac 1s}+|Y|^{\frac 1s})/h}
$$
hold true. In particular, the following is an immediate consequence of Proposition \ref{Prop:DilWeylprod}
and some standard manipulations in Gelfand-Shilov theory.

\par

\begin{cor}\label{cor:DerEstLocWeylprod}
Let $s\ge \frac 12$. If $\phi ,\psi \in \maclS _s(\rr {2d})$ ($\phi ,\psi \in \Sigma _s(\rr {2d})$), then
\begin{equation}\label{Eq:DerEstLocWeylprod}
|D_X^\alpha D_Y^\beta D_Z^\gamma (\phi _Y\wpr \psi _Z)(X)|
\lesssim h^{|\alpha +\beta +\gamma |}(\alpha !\beta !\gamma !)^s e^{-(|X-Y|^{\frac 1s}+|X-Z|^{\frac 1s})/h}
\end{equation}
for some $h>0$ (for every $h>0$).
\end{cor}

\par

\begin{proof}
Obviously, as stated in the previous Proposition \ref{Prop:DilWeylprod}, we can write $\phi _Y\wpr
\psi _Z$ as $\Psi (X-Y,X-Z)$ for some $\Psi\in \maclS _s(\rr {2d}
\times \rr {2d})$. Thus 
\begin{align*}
|D_X^\alpha D_Y^\beta &D_Z^\gamma \Psi (X-Y,X-Z)|
=
\left |D_X^\alpha
\left( D_1^\beta D_2^\gamma \Psi \right)(X-Y,X-Z) \right |
\\
&\leq
\sum_{\delta\leq\alpha}\binom{\alpha}{\delta} \left |
\left( D_1^{\beta+\delta} D_{2}^{\gamma+\alpha-\delta} \Psi \right)(X-Y,X-Z)
\right |
\\
&\leq h^{|\alpha+\beta+\gamma|}\sum_{\delta\leq\alpha}
\binom{\alpha}{\delta}
\left((\beta+\delta)!(\gamma+\alpha-\delta)!\right)^s e^{-r\left(|X-Y|^{\frac{1}{s}}+|X-Z|^{\frac{1}{s}}\right)}.
\end{align*}
Moreover, we have
\[
\sum_{\delta\leq\alpha}\binom{\alpha}{\delta}\left((\beta+\delta)!(\gamma+\alpha-\delta)!\right)^s\leq 2^{|\alpha|}4^{s|\alpha+\beta+\gamma|}\left(\alpha!\beta!\gamma!\right)^s.
\]
Indeed, using the fact that $\sum\limits_{\delta\leq\alpha}\binom{\alpha}{\delta}=2^{|\alpha|}$ and that $n!\leq 2^n(n-k)!k!$, which implies  $(n+k)!\leq 2^{n+k}n!k!$, then 
\[
(\alpha+\beta+\gamma)!= \prod\limits_{j=1}^d(\alpha_j+\beta_j+\gamma_j)! \leq \prod\limits_{j=1}^d 4^{\alpha_j+\beta_j+\gamma_j}\alpha_j!\beta_j!\gamma_j!=4^{|\alpha+\beta+\gamma|}\alpha!\beta!\gamma!.
\]
Thus, inequality \eqref{Eq:DerEstLocWeylprod} holds with $2\cdot 4^sh$ in place of $h$.
\end{proof}
\par

The next fundamental result is a consequence of Theorem 4.12 in \cite{CaTo}.

\begin{prop}\label{prop:prcont}
	Let $s\ge\frac{1}{2}$ and $\vartheta\in\mascP _E(\rr{2d})$.
	Then, the map $(\phi,a)\mapsto \phi\wpr a$ is continuous from $\Sigma_s(\rr{2d})\times\Gamma_s^{(\vartheta)}(\rr{2d})$
	to $\maclS_s(\rr{2d})$.
\end{prop}

\par

The next lemma concerns uniform estimates of the Weyl product
between elements in sets $\Omega _{j,Y_j}$, $Y_j\in \rr {2d}$, $j=1,2$,
such that related sets $\Omega _j^\cup$, given by
\begin{equation}\label{Eq:BoundedSymbSets}
\Omega _j^\cup \equiv
\bigcup _{Y\in \rr {2d}} \sets {a(\cdo +Y)}{a\in \Omega _{j,Y}},
\end{equation}
are bounded in $\maclS _s(\rr {2d})$ or in $\Sigma _s(\rr {2d})$, $j=1,2$.

\par

\begin{lemma}\label{Lemma:BoundedGSFamWeylComp}
Let $s\ge \frac 12$ and let  $\Omega _{j,Y}\subseteq
\mascS (\rr {2d})$, $Y\in \rr {2d}$, $\Omega _j^\cup$ be as in
\eqref{Eq:BoundedSymbSets}, $a_j\in \Omega _j^\cup$
and choose $Y_j\in \rr {2d}$ such that
$a_j\in \Omega _{j,Y_j}$, $j=1,2$. Then the following is true:
\begin{enumerate}
\item if $\Omega _j^\cup$ are bounded
in $\maclS _s(\rr {2d})$, then there are constants $C>0$ and $h>0$
which are independent of $Y_j\in \rr {2d}$ and $a_j$, $j=1,2$, and such
that
\begin{align}
|(a_1^{(\alpha _1)}\wpr a_2^{(\alpha _2)})(X)|
&\le
Ch^{|\alpha _1+\alpha _2|}(\alpha _1!\alpha _2!)^s
e^{-(|X-Y_1|^{\frac 1s}+|X-Y_2|^{\frac 1s}+|Y_1-Y_2|^{\frac 1s})/h}
\label{Eq:UniformEst1}
\intertext{and}
|D^\alpha(a_1\wpr a_2)(X)|
&\le
Ch^{|\alpha |}\alpha !^s
e^{-(|X-Y_1|^{\frac 1s}+|X-Y_2|^{\frac 1s}+|Y_1-Y_2|^{\frac 1s})/h}
\label{Eq:UniformEst2}
\end{align}

\vrum

\item if $\Omega _j^\cup$ are bounded
in $\Sigma _s(\rr {2d})$, then for every $h>0$, there is a constant $C>0$
which is independent of $Y_j\in \rr {2d}$ and $a_j$, $j=1,2$, and such
that \eqref{Eq:UniformEst1} and \eqref{Eq:UniformEst2} hold.
\end{enumerate}
\end{lemma}

\par

\begin{proof}
We only prove (2). The assertion (1) follows by similar arguments and is
left for the reader.

\par

We have $D_X(a_1\wpr a_2) = (D_Xa_1)\wpr a_2+a_1\wpr (D_Xa_2)$,
which implies that the Leibnitz rule is valid for the Weyl product. Hence,
\eqref{Eq:UniformEst2} follows for every $h>0$ if we prove that
\eqref{Eq:UniformEst1} holds for every $h>0$.

\par

Let $Y=Y_1$, $Z=Y_2$, $a=a_1$, $b=a_2$, $a_Y=a(\cdo +Y)$
and $b_Z=b(\cdo +Z)$.
Then
\begin{multline*}
a_1\wpr a_2(X)
=
\pi ^{-d} \int a_Y((X-Y)-Y_1) \mascF _\sigma (b_Z(\cdo -Z))(Y_1)
e^{2i\sigma (X,Y_1)}\, dY_1
\\[1ex]
=
\pi ^{-d} \int a_Y((X-Y)-Y_1) \mascF _\sigma (b_Z)(Y_1)
e^{2i\sigma (X-Z,Y_1)}\, dY_1
\\[1ex]
=\mascF _\sigma (a_Y(X_1-\cdo )(\mascF _\sigma b_Z))(X_2),
\end{multline*}
where $X_1=X-Y$ and $X_2=X-Z$. That is
$a_1\wpr a_2(X) = G(X_1,X_2)$,
where
$$
G = G_{Y,Z}=(\mascF _{\sigma ,2}\circ T)(a_Y\otimes b_Z).
$$
Here $\mascF _{\sigma ,2}F$ is the partial symplectic Fourier transform
of $F(\cdo ,X_2)$ with respect to the variable $X_2\in \rr {2d}$, and
$$
(TF)(X,Y) = F(X-Y,Y).
$$
Since both $\mascF _{\sigma ,2}$ and $T$ are continuous on
$\Sigma _s(\rr {2d}\times \rr {2d})$, it follows from the boundedness
of the sets $\Omega _j^\cup$ that for every $h>0$, there is a constant
$C_h$ which is independent of $a_j\in \Omega _j^\cup$ such that
$$
|D_{X_1}^{\alpha}D_{X_2}^{\beta}G(X_1,X_2)|
\le
C_h h^{|\alpha +\beta|}(\alpha \beta )^se^{-(|X_1|^{\frac 1s}+|X_2|^{\frac 1s})/h}.
$$
By straightforward computations it follows that
$$
|D_{X_1}^{\alpha}D_{X_2}^{\beta}G(X_1,X_2)|
=
|(a^{(\alpha )}\wpr b^{(\beta )})|,
$$
since
$$
|Y_1-Y_2|^{\frac 1s}\le c(|X-Y_1|^{\frac 1s}+|X-Y_2|^{\frac 1s})
$$
for some constant $c$ which only depends on $s$. The estimate
\eqref{Eq:UniformEst1} follows from these relations.
\end{proof}

\par

%
%
%
%
%
\begin{lemma}\label{lem:conf}
Let $s\ge \frac 12$, $\phi ,\psi \in \Sigma _s(\rr {2d})$, $\omega ,\vartheta
\in \mascP _{E}(\rr {2d})$, 
$\phi _Y=\phi (\cdo -Y)$, and
$\psi _Z=\psi (\cdo -Z)$. Then the following is true:
\begin{enumerate}
\item if $a\in \Gamma ^{(\omega )}_s(\rr {2d})$
($a\in \Gamma ^{(\omega )}_{0,s}(\rr {2d})$), then
\begin{align}
|D^\alpha _XD^\beta _Y(\phi _Ya)(X)|
\lesssim h_1^{|\alpha |}h_2^{|\beta |}(\alpha !\beta !)^s
e^{-|X-Y|^{\frac 1s}/h_1}\min (\omega (X),\omega (Y))
\label{Eq:LocEst1}
\intertext{and}
|D^\alpha _XD^\beta _Y(\phi _Y\wpr a)(X)|
\lesssim h_1^{|\alpha |}h_2^{|\beta |}(\alpha !\beta !)^s
e^{-|X-Y|^{\frac 1s}/h_1}\min (\omega (X),\omega (Y)),
\label{Eq:LocEst2}
\end{align}
for some $h_1>0$ (for every $h_1>0$) and every $h_2>0$; 

\vrum

\item if $a_1\in \Gamma ^{(\omega )}_s(\rr {2d})$ and
$a_2\in \Gamma ^{(\vartheta )}_s(\rr {2d})$
($a_1\in \Gamma ^{(\omega )}_{0,s}(\rr {2d})$ and
$a_2\in \Gamma ^{(\vartheta )}_{0,s}(\rr {2d})$),
then
\begin{multline*}
|D^\alpha _X D^\beta _Y D^\gamma _Z((\phi _Ya_1)\wpr (\psi _Za_2))(X)|
\\[1ex]
\lesssim h_1^{|\alpha +\beta |}h_2^{|\gamma |}(\alpha !\beta !\gamma !)^s
e^{-(|X-Y|^{\frac 1s}+|X-Z|^{\frac 1s}+|Y-Z|^{\frac 1s})/h_1}
\min _{X_1,X_2\in \{ X,Y,Z \} }\big (\omega (X_1)
\vartheta (X_2) \big ),
\end{multline*}
for some $h_1>0$ (for every $h_1>0$) and every $h_2>0$.
\end{enumerate}
\end{lemma}

\par

\begin{proof}
We only consider the case when $a\in \Gamma ^{(\omega )}_{0,s}(\rr {2d})$
and $b\in \Gamma ^{(\vartheta )}_{0,s}(\rr {2d})$. The other cases
follow by similar arguments and are left for the reader.

\par

Let
$$
\Psi (X,Y) = \phi (X-Y)a(X).
$$
By Leibniz rule we get
\begin{multline*}
|D^\alpha _XD^\beta _Y\Psi (X,Y)| \le \sum _{\gamma \le \alpha}
{\alpha \choose \gamma}
|\phi ^{(\alpha +\beta -\gamma)}(X-Y) a^{(\gamma )}(X)|
\\[1ex]
\lesssim
2^{|\alpha |}\sup _{\gamma \le \alpha}\left ( 
h^{|\alpha +\beta |}((\alpha +\beta -\gamma )!\gamma !)^s
e^{-|X-Y|^{\frac 1s}/h}\omega (X)
\right )
\\[1ex]
\le
(2^{1+s}h)^{|\alpha +\beta |}(\alpha !\beta !)^s
e^{-|X-Y|^{\frac 1s}/h}\omega (X)
\lesssim
(2^{1+s}h)^{|\alpha +\beta |}(\alpha !\beta !)^s
e^{-|X-Y|^{\frac 1s}/(2h)}
\omega (Y),
\end{multline*}
for every $h>0$ which is chosen small enough. Here we have used the
fact that for some $c>0$
$$
\omega (X)\lesssim \omega (Y)e^{c|X-Y|}\lesssim
\omega (Y)e^{|X-Y|^{\frac 1s}/(2h)},
$$
since $\omega$ is a moderate function. This gives \eqref{Eq:LocEst1}.

\par

%

Next we prove (2). Let 
\begin{align*}
\Omega _{1,Y} &= \Sets
{X\mapsto \frac {D_Y^\beta (\phi _Ya)}{h^{|\beta |}\beta !^s\omega (Y)}}
{\beta \in \nn d}
\\[1ex]
\Omega _{2,Z} &= \Sets
{X\mapsto \frac {D_Z^\gamma (\psi _Zb)}{h^{|\gamma |}\gamma !^s\vartheta (Z)}}
{\gamma \in \nn d},
\end{align*}
and let $\Omega _1^\cup$ and $\Omega _2^\cup$ be as in
\eqref{Eq:BoundedSymbSets}. By \eqref{Eq:LocEst1} it follows that
$\Omega _1^\cup$ and $\Omega _2^\cup$ are bounded in
$\Sigma _s(\rr {2d})$. Hence, Lemma
\ref{Lemma:BoundedGSFamWeylComp} shows that
\begin{multline*}
\left |
D_X^\alpha \left ( 
\left (
\frac {D_Y^\beta (\phi _Ya)}{h^{|\beta |}\beta !^s\omega (Y)}
\right )
\wpr
\left (
\frac {D_Z^\gamma (\psi _Zb)}{h^{|\gamma |}\gamma !^s\vartheta (Z)}
\right )
\right )(X)
\right |
\\[1ex]
\lesssim h^{|\alpha |}\alpha !^s e^{-(|X-Y|^{\frac 1s}+|X-Z|^{\frac 1s}
+|Y-Z|^{\frac 1s})/h}
\end{multline*}
for every $h>0$, or equivalently,
\begin{multline*}
| D_X^\alpha D_Y^\beta D_Z^\gamma
((\phi _Ya)\wpr (\psi _Zb))(X)|
\\[1ex]
\lesssim
h^{|\alpha +\beta +\gamma|}(\alpha ! \beta ! \gamma !)^s
e^{-(|X-Y|^{\frac 1s}+|X-Z|^{\frac 1s}
+|Y-Z|^{\frac 1s})/h}\omega (Y)\vartheta (Z).
\end{multline*}
The assertion now follows from the latter estimate and the fact
that $\omega$ and $\vartheta$ are moderate weights, giving that
$$
\omega (Y)\lesssim
\omega (X)e^{|X-Y|^{\frac 1s}/(2h)}
\lesssim
\omega (Z)e^{(|X-Y|^{\frac 1s}+|X-Z|^{\frac 1s})/(2h)},
$$
and similarly for $\vartheta$.

\par

Finally, the estimate in (2) also holds for $(\phi _Yb)\wpr (\psi _Za)$
in place of $(\phi _Ya)\wpr (\psi _Zb)$,
and \eqref{Eq:LocEst2} follows from this estimate by letting $\gamma =0$,
$b=1$, $\vartheta =1$, and then integrate with respect to $Z$. The proof
is complete.
\end{proof}

\par

Lemmas \ref{Lemma:BoundedGSFamWeylComp} and \ref{lem:conf} imply the
following characterisation of $\Gamma_s^{(\omega)}(\rr{2d})$.

\par

\begin{prop}\label{prop:charact}
    Let $s>1/2$, $\omega\in \mascP_E(\rr {2d})$, $a\in\Sigma_1'(\rr{2d})$, 
    $\phi\in\Sigma_s(\rr{2d})$ have non-vanishing integrals, and let $\phi_Y
    =\phi(\cdo-Y)$. Then the following conditions are equivalent:
    \begin{enumerate}
        \item $a\in\Gamma_s^{(\omega)}$;
        \item $\phi_Y a$ is smooth and satisfies \eqref{Eq:LocEst1} for some
        $h_1>0$ and every $h_2>0$;
        \item $\phi_Y\wpr a$ is smooth and satisfies \eqref{Eq:LocEst2} for
        some $h_1>0$ and every $h_2>0$;
        \item 
            \begin{align}
                 |D^\alpha _X(\phi _Ya)(X)|
                 \lesssim h_1^{|\alpha |}\alpha !^s
                  e^{-|X-Y|^{\frac 1s}/h_1}\min (\omega (X),\omega (Y)),
            \end{align}
            for some $h_1>0$;
        \item
            {\begin{align}
                 |D^\alpha _X(\phi _Y\wpr a)(X)|
                 \lesssim h_1^{|\alpha |}\alpha !^s
                  e^{-|X-Y|^{\frac 1s}/h_1}\min (\omega (X),\omega (Y)),
            \end{align}for some $h_1>0$.}
    \end{enumerate}
\end{prop}

\par

\begin{lemma}\label{lem:confbis}
Let $s\ge\frac{1}{2}$. Assume that $\{a_1(\cdo+Y,Y)\}_{Y\in\rr{2d}}$ and $\{a_2(\cdo+Z,Z)\}_{Z\in\rr{2d}}$ are bounded families in $\maclS_s(\rr{2d})$. Then,
for any $X,Y,Z\in\rr{2d}$, $\alpha\in\nn{2d}$,
\begin{equation}\label{eq:estsharp}
	|D^\alpha _X (a_1(\cdo+Y,Y)\wpr  a_2(\cdo+Z,Z) )(X)|
	\lesssim h^{|\alpha |}\alpha! ^s e^{-r(|X-Y|^{\frac 1s}+|X-Z|^{\frac 1s})},  
\end{equation}
for some $h>0$.
\end{lemma}

\par

\begin{proof}
The result is a straightforward consequence of Corollary \ref{cor:DerEstLocWeylprod}
and its proof, since the involved constants on the right hand side of
\eqref{Eq:DerEstLocWeylprod} depend continuously on $\phi$ and
$\psi$ in $\maclS_s(\rr{2d}).$
\end{proof}

\par

We notice that, by straightforward computation, for some other $h, r>0$,
\eqref{eq:estsharp} gives, for any $X,Y,Z\in\rr{2d}$, $\alpha\in\nn{2d}$,
\begin{equation}\label{eq:estsharpbis}
	|D^\alpha _X (a_1(\cdo+Y,Y)\wpr  a_2(\cdo+Z,Z) )(X)|
	\lesssim h^{|\alpha |}\alpha! ^s
	e^{-r(|X-Y|^{\frac 1s}+|X-Z|^{\frac 1s}+|Y-Z|^{\frac 1s})}.
\end{equation}

\par

\subsection{A family related to $\Gamma ^{(1)}_s$ and $\Gamma ^{(1)}_{0,s}$}
Let $I_R=[-R,R]$ and $E^0=E^0_{h,s}=L^\infty(I_R\times\rr{2d};
	s_\infty ^w(\rr {2d}))$, with the symbol subspace $s_\infty ^w(\rr {2d})$
	from Definition \ref{def:swinf}.
	We shall consider
	suitable decreasing family $\{E^n_{h,s}\}_{n=0}^\infty$ of
	Banach spaces. To this aim, let
	\[
		G_n=\{(Y,T_1,\dots,T_n)\in\rr{2d(n+1)}\colon Y,T_j\in\rr{2d}
		\text{ with } |T_j|\le 1,\ j=1,\dots, n\} ,
		 n\in\zz {}_+.
	\]
	We define $E^n_{h,s}$, $n\ge1$, as the set of all $a\in E^0$ such that
	\[
		\|a\|^{(n)}=
		\sup_{1\le k\le n}\sup_{t\in I_R}\sup_{(Y,T_1,\dots,T_k)\in G_k}
		\frac{\| 
			\scal {T_1}{D_X} \cdots \scal{T_k}{D_X}a(t,Y,\cdo ) 
		\|_{s_\infty ^w}
		}
		{h^k (k!)^s}
		<\infty,
	\]
	with the norm
	\[
		\nm a{E^n_{h,s}}\equiv \|a\|_{E^0}+\nmm a ^{(n)}.
	\]
	We also let $E^\infty _{h,s}$ be the set of all
	$$
	a\in \bigcap _{n\ge 0}E^n_{h,s}
	$$
	such that
	$$
	\nm a{E^\infty _{h,s}} \equiv \sup _{n\ge 0}\nm a{E^n _{h,s}}
	$$
	is finite.
	
\par

\begin{lemma}\label{Lemma:EhnBanach}
    Let $n\ge 0$. Then $E^n_{h,s}$ is a Banach space.
\end{lemma}

\par

\begin{proof}
    Let $\{a_j\}_{j\ge0}$ be a Cauchy sequence in $E^n_{h,s}$, $n\ge1$. By
    definition, this sequence clearly has a limit $a\in E^0$,
	and for some $X\mapsto b_k(t,Y,T_1,\dots ,T_k,X)\in
	s_\infty ^w(\rr {2d})$ we have
$$
\lim _{j\to \infty} \sup \frac{\| 
			\scal {T_1}{D_X} \cdots \scal{T_k}{D_X}a_j(t,Y,\cdo ) 
		- b_k(t,Y,T_1,\dots ,T_k,\cdo )\|_{s_\infty ^w}
		}
		{h^k (k!)^s}=0,
$$
where the supremum is taken over all
$$
k\in \{ 1,\dots ,n\} ,\quad t\in I_R
\quad \text{and}\quad
(Y,T_1,\dots ,T_k)\in G_k.
$$
	We need to prove that $a\in E^n_{h,s}$, and $a_j\to a$ in $E^n_{h,s}$.
	
	\par
	
%
The conditions here above are equivalent to
\begin{gather}
\lim _{j\to \infty} \sup _{t\in I_R}\sup _{Y\in \rr {2d}}
\nm {a_j(t,Y,\cdo )-a(t,Y,\cdo )}{s_\infty ^w}
=0\label{Eq:Limit0Case}
\intertext{and}
\lim _{j\to \infty} \sup \frac 
{\nm {(-1)^k\scal {T_1}D\cdots \scal {T_k}Da_j(t,Y,\cdo )
-b_k(t,Y,T_1,\dots ,T_k,\cdo )}{s_\infty ^w}}{h^k (k!)^s}
=0,\label{Eq:Limit1TonCase}
\end{gather}
where the latter supremum should be taken over all
$$
k\in \{ 1,\dots ,n\} ,\quad t\in I_R
\quad \text{and}\quad
(Y,T_1,\dots ,T_k)\in G_k.
$$

\par

Since $s_\infty ^w(\rr {2d})$ is continuously embedded in $\mascS '(\rr {2d})$,
it follows from \eqref{Eq:Limit0Case} and \eqref{Eq:Limit1TonCase}
that
$$
X\mapsto (-1)^k\scal {T_1}{D_X}\cdots \scal {T_k}{D_X} a_j(t,Y,X)
$$
has the limit
$$
X\mapsto (-1)^k\scal {T_1}{D_X}\cdots \scal {T_k}{D_X} a(t,Y,X)
$$
in $\mascS '(\rr {2d})$, and the limit
$$
X\mapsto b_k(t,Y,T_1,\dots ,T_k,X)
$$
in $s_\infty ^w(\rr {2d})$, and thereby in $\mascS '(\rr {2d})$, as $j$
tends to $\infty$. Hence
$$
b_k(t,Y,T_1,\dots ,T_k,X) = (-1)^k\scal {T_1}{D_X}\cdots \scal {T_k}{D_X}
a(t,Y,X)
$$
and it follows that $E_{h,s}^n$ is a Banach space for every $h>0$, $s>0$ and
integer $n\ge 0$.
\end{proof}

\par

The spaces $E^\infty _{h,s}$ can be related
to $\Gamma ^{(1)}_s$ and $\Gamma ^{(1)}_{0,s}$, as the following
lemma shows. The details are left for the reader.

\par

\begin{lemma}\label{Lemma:GammaLink}
Let $a\in L^\infty (I_R\times \rr {2d};s^w_\infty (\rr {2d}))$. Then
$\{ a(t,Y,\cdo ) \} _{t\in I_R,Y\in \rr {2d}}$ is a uniformly
bounded family in $\Gamma ^{(1)}_s(\rr {2d})$
({\,}$\Gamma ^{(1)}_{0,s}(\rr {2d})$), if and only if
$$
\nm a{E^\infty _{h,s}}<\infty 
$$
for some $h>0$ (for every $h>0$).
\end{lemma}

\par

Later on we also need the following result of differential
equations with functions depending on a real variable with
values in $E^\infty _{h,s}$. The proof is omitted since
the result can be considered as a part of the standard theory
of ordinary differential equations of first order in Banach spaces.

\par

\begin{lemma}\label{Lemma:ODE}
Suppose $s\ge 0$ and $T>0$, and let $\maclK$ be an operator
from $E^\infty _{h,s}$ to $E^\infty _{h,s}$ for every
$h>0$ such that
\begin{equation}\label{Eq:KEstLemma}
\nm {\maclK a}{E^\infty _{h,s}} \le C
\nm {a}{E^\infty _{h,s}},\qquad a\in E^\infty _{h,s},
\end{equation}
for some constant $C$ which only depend on $h>0$. Then
$$
\frac{dc(t)}{dt} =\maclK(c(t)) ,\qquad c(0)\in E^\infty _{h,s},
$$
has a unique solution $t\mapsto c(t)$ from $[-T,T]$ to
$E^\infty _{h,s}$ which satisfies
$$
\nm {c(t)}{E^\infty _{h,s}}\le \nm {c(0)}{E^\infty _{h,s}}
e^{CT},
$$
where $C$ is the same as in \eqref{Eq:KEstLemma}.
\end{lemma}

\par

%
%

\par

\section{One-parameter group of elliptic symbols in
the classes $\Gamma ^{(\omega )}_s(\rr {d})$ }\label{sec3}

\par

In this section we show that for suitable $s$ and $\omega _0$, there are
elements $a\in \Gamma ^{(\omega _0)}_{s}$ and $b\in \Gamma
^{(1/\omega _0)}_{s}$
such that $a\wpr b = b\wpr a =1$. This is essentially a consequence of
Theorem \ref{thm:symbeqdiff},
where it is proved that
the evolution equation \eqref{Eq:IntrSymeqDiff2} has a unique solution
$a(t,\cdo )$ which belongs to $\Gamma ^{(\omega\vartheta^t )}_{s}$,
thereby deducing needed semigroup properties for scales
of pseudo-differential operators.
Similar facts hold for corresponding Beurling type spaces
(cf. Theorem \ref{thm:symbeqdiff2}).

\par

First we have the following result on certain logarithms of weight 
functions.

\par

\begin{thm}\label{thm:logomega}
	Let $\omega\in\mascP_E(\rr{2d})\cap\Gamma^{(\omega)}_{s_0}(\rr{2d})$,
	$s_0\in(0, 1]$, $v\in \mascP_E(\rr{2d})$, be such that $\omega$ is 
	$v$-moderate, $\vartheta(X)=1+\log v(X)$ 
	and let
	\[
		c(X,Y)=\log \frac{\omega(X+Y)}{\omega(Y)}.
	\]
	Then,
	\begin{itemize}
		\item $\{c(\cdo ,Y)\}_{Y\in\rr{2d}}$ is a uniformly bounded family in 
		$\Gamma^{(\vartheta)}_s(\rr{2d})$, $s\ge1$;
		\item for $\alpha\not=0$,
		$\{(\partial^\alpha_X c)(\cdo ,Y)\}_{Y\in\rr{2d}}$ is a uniformly
		bounded 
		family in $\Gamma^{(1)}_s(\rr{2d})$, $s\ge1$.
	\end{itemize}
\end{thm}
For the proof of Theorem \ref{thm:logomega} we will need the following multidimensional version
of the well-known Fa{\`a} di Bruno formula for the derivatives of composed functions. It can be found, 
e.g., setting $q=p=1$, $n=2d$, in equations (2.3) and (2.4) in \cite{Gzyl}.
\begin{lemma}\label{lem:faadb}
	Let $f\colon\rr{}\to\rr{}$, $g\colon\rr{2d} \to \rr{}$. Then, for any $\alpha\in\nn{2d}$, $\alpha\not=0$,
	\begin{equation}\label{eq:faadb}
		\frac{\partial^\alpha f(g(x))}{\alpha!}=\sum_{1\le k\le |\alpha|}\frac{f^{(k)}(g(x))}{k!}
		\sum_{\substack{\beta_1+\cdots+\beta_k=\alpha \\ \beta_j\not=0,\,  j=1,\dots,k}}
		\prod_{1\le j\le k} \frac{(\partial^{\beta_j}g)(x)}{\beta_j!}.
	\end{equation}
\end{lemma}
We will also need the next \textit{factorial estimate}, for expressions involving
decompositions of $\alpha\in\nn{2d}$, $\alpha\not=0$,
into the sum of $k$ nontrivial multiindeces $\beta_j$, $j=1,\dots,k$, as in
\eqref{eq:faadb}, and corresponding products of (powers of) factorials.
\begin{lemma}\label{lem:estfaadb}
	Let $s_0\in(0,1]$, $\alpha\in\nn{2d}$, $\alpha\not=0$. Then, for a suitable $C_0>0$,
	depending only in $d$,
	\begin{equation}\label{eq:estfaadb}
		\sum_{1\le k\le |\alpha|}
		\frac{1}{k}
		\sum_{\substack{\beta_1+\cdots+\beta_k=\alpha \\ \beta_j\not=0,\,  j=1,\dots,k}}
		\prod_{1\le j\le k}(\beta_j!)^{s_0-1}\lesssim C_0^{|\alpha|}.
	\end{equation}
\end{lemma}
The proof of Lemma \ref{lem:estfaadb} can be found in the Appendix.
\begin{proof}[Proof of Theorem \ref{thm:logomega}]
	We have to show that $c(\cdo , Y)$ satisfies $\Gamma^{(\vartheta)}_s$
	estimates, uniformly with respect to $Y\in\rr{2d}$. 
	If $c(X,Y)\ge0$, then it follows by submultiplicativity of $\omega$, that
	\begin{align*}
		c(X,Y)&=\log\omega(Y+X)-\log\omega(Y)\lesssim
		\log\omega(Y)+\log v(X)-\log\omega(Y)
		\\
		&\le \vartheta(X),
	\end{align*}
	for any $Y\in\rr{2d}$. Again by moderateness, when $c(X,Y)\le0$,
	recall that $\omega(X+Y)\ge\frac{\omega(Y)}{v(X)}$, so that
	\[
		c(X,Y)\gtrsim \log\frac{\omega(Y)}{v(X)}-\log\omega(Y)\ge-\log v(X)
		\ge-\vartheta(X),
	\]
	and we can conclude $|c(X,Y)|\lesssim\vartheta(X)$, $X\in\rr{2d}$.
	Now, for $\alpha\in\nn{2d}$, $\alpha\not=0$, \eqref{eq:faadb} gives
	\[
		\partial^\alpha_X c(X,Y)=
		\alpha! \sum_{1\le k\le |\alpha|}\frac{(-1)^{k+1}}{k\,[\omega(X+Y)]^k}
		\sum_{\substack{\beta_1+\cdots+\beta_k=\alpha \\ \beta_j\not=0,\,  j=1,\dots,k}}
		\prod_{1\le j\le k}
			\frac{(\partial^{\beta_j}\omega)(X+Y)}{\beta_j!}
	\]
	We can then estimate, in view of \eqref{eq:estfaadb},
	\begin{align*}
		|\partial^\alpha_X c(X,Y)|
		&\lesssim \alpha! \sum_{1\le k\le |\alpha|}
		\frac{1}{
		k\,[\omega(X+Y)]^{k}}
		\sum_{\substack{\beta_1+\cdots+\beta_k=\alpha \\ \beta_j\not=0,\,  j=1,\dots,k}}
		\prod_{1\le j\le k}
		\frac{\omega(X+Y) h^{|\beta_j|}(\beta_j!)^{s_0}}{\beta_j!}
		\\
		&=
		h^{|\alpha|} \, \alpha! 
		\sum_{1\le k\le |\alpha|}
		\frac{1}{k}
		\sum_{\substack{\beta_1+\cdots+\beta_k=\alpha \\ \beta_j\not=0,\,  j=1,\dots,k}}
		\prod_{1\le j\le k} (\beta_j!)^{s_0-1}\lesssim (C_0h)^{|\alpha|} (\alpha!)^{s},
	\end{align*}
	which concludes the proof.
\end{proof}
\begin{prop}\label{prop:confest}
	Assume $s>\frac{1}{2}$ and $\omega(X)\lesssim e^{r|X|^\frac{1}{s}}$ for some $r>0$.
	Let $\{a(\cdo,Y)\}_{Y\in\rr{2d}}$ be a uniformly bounded family in $\Sigma_s(\rr{2d})$
	and $\{c(\cdo,Z)\}_{Z\in\rr{2d}}$ be a bounded family in $\Gamma^{(\omega)}_s(\rr{2d})$.
	Then, 
	\[
		\{a(\cdo,Y)\wpr c(\cdo,Z)\}_{Y,Z\in\rr{2d}}
		\text{ and }
		\{c(\cdo,Z)\wpr a(\cdo, Y)\}_{Y,Z\in\rr{2d}}
	\]
	are bounded families in $\maclS_s(\rr{2d})$.
\end{prop}

\par

\begin{proof}
An immediate consequence of \eqref{Eq:LocEst2} in Lemma \ref{lem:conf} is the following
\begin{equation}\label{Eq:LocEst2bis}
    |D_X^\alpha(\phi\wpr a)(X)|\leq C h^{|\alpha|}\alpha!^s e^{-r|X|^{\frac{1}{s}}},
\end{equation}
which implies that, for $\phi\in\Sigma_s$ and $a\in\Gamma ^{(\omega )}_s$,
\eqref{Eq:LocEst2bis} holds if and only if $\phi\wpr a$ belongs to $\maclS_s$.
Then by the proof of \eqref{Eq:LocEst2bis} it follows that the constant $C,h$
and $r$ can be chosen depending continuously on $\phi\in\Sigma _s(\rr {2d})$
and $a\in \Gamma ^{(\omega )}_s(\rr {2d})$. Hence if $\Omega_1$ is a bounded
family in $\Sigma _s(\rr {2d})$ and $\Omega_2$ is a bounded family in
$\Gamma^{(\omega )}_s(\rr {2d})$, then $\{\phi\wpr a\}_{\phi \in \Omega
_1,a\in\Omega_2}$ is a bounded family in $\maclS_s(\rr{2d})$.
\end{proof}

\par

The following result can be found, e.g., in \cite{To3}.
\begin{lemma}\label{lem:estoprnorms}
		Let $a\in\mascS '(\rr{2d})$. Then,
		\begin{align}\label{eq:estopa}
			\| a\|_{s_\infty ^w}
			&\le C
			\sum_{|\alpha|\le d+1}\|\partial^\alpha a\|_{L^\infty}
			\intertext{and}
			\label{eq:estopb}
			\|a\|_{L^\infty}
			&\le C
			\sum_{|\alpha|\le 2d+1}
			\|\partial^\alpha a\|_{s_\infty ^w},
		\end{align}
    for some constant $C>0$ depending only on the dimension
		$d$.
\end{lemma}

\par

\begin{prop}\label{prop:symbdefequiv}
	Let $a\in \mascS^\prime(\rr{2d})$, $s\ge\frac{1}{2}$ and set
	$b_{\alpha\beta}(X)=\partial^\alpha(X^\beta a(X))$ when
	$\alpha,\beta \in\nn {2d}$.
	Then the following conditions are equivalent:
	\begin{enumerate}
		\item $a\in\maclS_s(\rr{2d})$;
		\item for some $h>0$ it holds 
			\[
			\|b_{\alpha\beta}\|_{L^\infty}
			\lesssim h^{|\alpha+\beta|}(\alpha!\beta!)^s,
			\qquad \alpha,\beta\in\nn {2d}\text ;
			\]
		\item for some $h>0$ it holds 
			\[
			\|b_{\alpha\beta}\|_{s_\infty ^w}
			\lesssim h^{|\alpha+\beta|}(\alpha!\beta!)^s,
			\qquad \alpha,\beta\in\nn {2d}\text .
			\]
	\end{enumerate}
\end{prop}

\par

\begin{proof}
	The equivalence between (1) and (2) is well-known. The proof of 
	the equivalence
	of (2) and (3) follows by a straightforward application of
	Lemma \ref{lem:estoprnorms}. In fact, assume that (2) holds true.
	Then \eqref{eq:estopa} gives
	\begin{align*}
			\| b_{\alpha\beta}\|_{s_\infty ^w}
			&\le C
			\sum_{|\gamma|\le d+1}
			\|\partial^\gamma b_{\alpha\beta}\|_{L^\infty(\rr{2d})}
			\lesssim
			\sum_{|\gamma|\le d+1}
			h^{|\alpha+\beta+\gamma|}((\alpha+\gamma)!\beta! )^s
			\\
			&=h^{|\alpha+\beta|}(\alpha!\beta!)^s
			\sum_{|\gamma|\le d+1}
			h^{|\gamma|}(\gamma!)^s
			\left (\frac{(\alpha+\gamma)!}{\alpha!\,\gamma!}
			\right )^s
			\\
			&\lesssim
			(2^sh)^{|\alpha+\beta|}(\alpha!\beta!)^s,
	\end{align*}
	with a constant depending only on $d$ and $h$,
	since
	%
	\begin{align*}
	\sum_{|\gamma|\le d+1}
			h^{|\gamma|}(\gamma!)^s
			\left(\frac{(\alpha+\gamma)!}{\alpha!\,\gamma!}\right)^s
			\le 
			C_1\cdot 2^{s(|\alpha |+d+1)}
			\le C_2 2^{s|\alpha +\beta|},
	\end{align*}
	where the constants $C_1$ and $C_2$ only depend on $d$ and $h$.
	Then, (3) holds true, as claimed. The proof of the 
	converse follows by similar argument,
	employing \eqref{eq:estopb}.
\end{proof}

\par

We have now the following.

\par
\begin{thm}\label{thm:gamma1}
		Let $a\in\mascS^\prime(\rr{2d})$ and $s>0$. Then the following
		conditions are equivalent:
		\begin{enumerate}
		\item $a\in\Gamma^{(1)}_s(\rr{2d})$;
		\item there exists $h>0$ such that
			\[
			\|\partial^\alpha a\|_{L^\infty(\rr{2d})}
			\lesssim h^{|\alpha|}(\alpha!)^s,
			\]
			for all $\alpha\in\zzp{2d}$;
		\item there exists $h>0$ such that
			\[
			\|\partial^\alpha a\|_{s_\infty ^w}
			\lesssim h^{|\alpha|}(\alpha!)^s,
			\]
			for all $\alpha\in\zzp{2d}$;
		\item there exists $h>0$ such that
			\[
				\| 
				  \scal {T_1}{D_X} \cdots \scal {T_m}{D_X} a
				\|_{s_\infty ^w}
				\lesssim h^m (m!)^s,
			\]
			for any $T_1,\dots,T_m\in\rr{2d}$ such that $|T_j|\le1$,
			$j=1,\dots,m$, $m\ge1$.
	\end{enumerate}
\end{thm}
\begin{proof}
	The equivalence between (1) and (2) is well known.	The equivalence of
	(2) and (3) is proved by an argument
	completely similar to the one employed in the proof of
	Proposition \ref{prop:symbdefequiv}, using Lemma	\ref{lem:estoprnorms}.
	It remains to prove only the equivalence
	with (4). Assume that (3) holds true, and let
	\[
		T_j=\sum_{l=1}^d t_{jl} \be_l + \sum_{l=1}^d \tau_{jl} \beps_l, 
	\]
	for the standard basis $(\be_l)_{l=1,\dots,d}$ of $\rr{d}$
	and the dual basis $(\beps_l)_{l=1,\dots,d}$. Recall that
	\[
		\scal {T_k}{D_X}a=
		  \sum_{l=1}^d t_{kl}\frac{\partial a}{\partial x_l}
		+
		  \sum_{l=1}^d \tau_{kl}\frac{\partial a}{\partial\xi_l},
	\]
	so that the symbol $\scal {T_1}{D_X} \cdots \scal {T_m}{D_X}a$
	is in the span of symbols of the form 
	\[
		\left(\prod_{j=1}^m t_j^{\beta_j}\tau_j^{\gamma_j}\right)
		\partial^\beta_x\partial^\gamma_\xi a,
		\quad \sum_{j=1}^m\beta_j=\beta, \sum_{j=1}^m\gamma_j=\gamma,
		|\beta+\gamma|=m,
	\]
	where the summation contains at most $(2d)^m$ terms.
	Since $|T_j|\le 1$, $j=1,\dots,m$, by \eqref{eq:estopa} we find
	\begin{align*}
				\| 
				   \scal {T_1}{D_X} \cdots \scal {T_m}{D_X} a
				\|_{s_\infty ^w}
				&\le (2d)^m\sup_{|\alpha|=m}
				\|\partial^\alpha a\|_{s_\infty ^w}
				\\
				&\lesssim\sup_{|\alpha|=m}
				\sum_{|\gamma|\le d+1}h^{|\alpha+\gamma|}
				\left ( (\alpha+\gamma)!\right )^s
				\\
				&=\sup_{|\alpha|=m} h^{|\alpha|}(\alpha!)^s
				\sum_{|\gamma|\le d+1} h^{|\gamma|}(\gamma !)^s
				\left(\frac{(\alpha+\gamma)!}{\alpha!\,\gamma !}\right)^s
				\\
				&\lesssim (2^{s+1}h)^m (m!)^s,
	\end{align*}
  	which proves that (4) holds true. 
  	As in the proof of Proposition \ref{prop:symbdefequiv}, the converse
  	implication is obtained by a completely similar argument, employing 
  	\eqref{eq:estopb}.
\end{proof}

\par

We are now ready to state and prove the first main result of this section, the following
Theorem \ref{thm:symbeqdiff}.
It deals with the existence of one-parameter groups of
pseudo-differential operators, obtained
as solutions to suitable evolution equations.

\par

\begin{thm}\label{thm:symbeqdiff}
	Let $s\ge1$, $\omega,\vartheta\in\mascP^0_{E,s}(\rr{2d})$ be such that 
	$\omega\in\Gamma^{(\omega)}_s(\rr{2d})$ and
	$\vartheta\in\Gamma^{(\vartheta)}_s(\rr{2d})$,
	and let
	$a_0\in\Gamma^{(\omega)}_s(\rr{2d})$, 
	$b\in\Gamma^{(1)}_s(\rr{2d})$.
	Then, there exists a unique smooth map
	$(t,X)\mapsto a(t,X)\in\cc{}$ 
	such that $a(t,\cdo)\in\Gamma_s^{(\omega\,\vartheta^t)}(\rr{2d})$
	for all $t\in\rr{}$, and
	\begin{equation}\label{eq:symeqdiff}
			\begin{cases}
				(\partial_t a)(t,\cdo)=(b+\log\vartheta)\wpr a(t, \cdo)
				\\
				a(0,\cdo)=a_0.
			\end{cases}
	\end{equation}
	If $\omega\equiv a_0\equiv 1$, then $a(t,X)$ also satisfies 
	\begin{equation}\label{eq:symeqdiffbis}
			\begin{cases}
				(\partial_t a)(t,\cdo)=a(t,\cdo)\wpr (b+\log\vartheta)
				\\
				a(0,\cdo)=a_0,
			\end{cases}
	\end{equation}
	and
	\begin{equation}\label{eq:sharprodsemigroupppt}
		a(t_1,\cdo)\wpr a(t_2,\cdo)=a(t_1+t_2,\cdo), \quad
		a(t,\cdo )\in \Gamma^{(\vartheta^t)}_s(\rr{2d}),\quad
		t,t_1,t_2\in \mathbf R.
	\end{equation}
\end{thm}

\par

\begin{proof}
	First suppose that a solution $a(t,X)$ of \eqref{eq:symeqdiff} exists.
	Then,
	\[
		a(t,X)=a_0(X)+\int_0^t c(u,X)\, du
	\]
	with
	\[
		c(t,\cdo)=(b + \log\vartheta)\wpr a(t,\cdo)
		\in \Gamma_s^{(\omega \langle \log\vartheta\rangle \vartheta^t)}
		(\rr{2d}),
	\]
	in view of Theorem \ref{thm:logomega} and the properties of the Weyl product in the $\Gamma_s^{(\omega)}(\rr{2d})$ classes, see \cite{CaTo}.
	%
	%
	This implies that the map
	$t\mapsto a(t,\cdo)$ is $C^1$ from $[-R,R]$ into the symbol space
	\[
		\Gamma_s^{(\omega \langle \log\vartheta\rangle (\vartheta+\vartheta^{-1})^R)}(\rr{2d}).
	\]
	 Choose $s_0<s$, and 
	$\phi,\psi\in\maclS_{s_0}(\rr{2d})$ such that \eqref{eq:pu} holds true. Let
	\begin{equation}\label{c_1}
		c_1(t,Y,\cdo )=\omega(Y)^{-1}\vartheta(Y)^{-t}\,\phi_Y\wpr a(t,\cdo).
	\end{equation}
	By Lemma \ref{lem:conf} (1) we have that, for any $Y\in\rr{2d}$,
	$t\mapsto c_1(t,Y,\cdo )$ is a $C^1$ map
	from $[-R,R]$ into $\maclS_s(\rr{2d})$. 
	Moreover, 
	%
	\begin{align*}
		\partial_t c_1(t,Y ,\cdo )&=\omega(Y)^{-1}		\vartheta(Y)^{-t}\,\phi_Y\wpr (b+\log\vartheta)		\wpr a(t,\cdo)-c_1(t,Y ,\cdo ) \,\log\vartheta(Y).
	\end{align*}
	Let 
	\[
		f(Y,X)=b(X)+\log\frac{\vartheta(X)}{\vartheta(Y)}.
	\]
	Then,
	\begin{align*}
		(\partial_t c_1)(t,Y ,\cdo )&=\omega(Y)^{-1}\		\vartheta(Y)^t\int \phi_Y\wpr f(Y,\cdo )\wpr \psi_Z\wpr \phi_Z\wpr a(t,\cdo)
		\,dZ
		\\
		&=\int K_{Y,Z}(t,\cdo)\wpr c_1(t,Z ,\cdo )\,dZ
	\end{align*}
	with
	\begin{equation}\label{Eq:KKernelExp1}
		K_{Y,Z}(t,\cdo)=\frac{\omega(Z)\,\vartheta(Z)^t}
		{\omega(Y)\,\vartheta(Y)^t}\phi_Y\wpr f(Y,\cdo )\wpr \psi_Z,
	\end{equation}
	and
	\[
		c_1(0,Y,\cdo )=\omega(Y)^{-1}\phi_Y\wpr a_0.
	\]
	We also need to consider the similar equation where $f(Y,\cdo )$ is replaced
	by $f(Z,\cdo )$, that is
	\begin{equation}\label{c_2}
		\partial_t c_2(t,Y ,\cdo )=\int \tK_{Y,Z}(t,\cdo)\wpr c_2(t,Z ,\cdo )\,dZ,
	\end{equation}
	where
	\[
		\tK_{Y,Z}(t,\cdo)=\frac{\omega(Z)\,\vartheta(Z)^t}
		{\omega(Y)\,\vartheta(Y)^t}\phi_Y\wpr f(Z,\cdo )\wpr \psi_Z,
	\]
	and
	\begin{equation}\label{c_2(0)}
		c_2(0,Y,\cdo )=c_1(0,Y,\cdo )=\omega(Y)\phi_Y\wpr a_0.
	\end{equation}

	We consider the operators $\maclK$ and $\widetilde \maclK$
	on $E^0$, defined by
	\begin{align*}
		(\maclK a)(t,Y,X) &=\int (K_{Y,Z}(t,\cdo )\wpr a(t,Z,\cdo ))(X)\, dZ,
    \intertext{and}
		(\widetilde \maclK a)(t,Y,X)
		&=\int (\widetilde K_{Y,Z}(t,\cdo )\wpr a(t,Z,\cdo ))(X)\, dZ,
	\end{align*}
	and show that
	\begin{equation}\label{Eq:KtildeKEst}
    \begin{alignedat}{2}
	\nm {\maclK a}{E_{h,s}^n}&\le C(n+1)\nm a{E_{h,s}^n},&
	\quad
	\nm {\maclK a}{E_{h,s}^\infty}&\le C\nm a{E_{h,s}^\infty},
    \\[1ex]
    \nm {\widetilde \maclK a}{E_{h,s}^n}&\le C(n+1)\nm a{E_{h,s}^n} &
    \quad \text{and}\qquad
    \nm {\widetilde \maclK a}{E_{h,s}^\infty} &\le C\nm a{E_{h,s}^\infty}
    \end{alignedat}
	\end{equation}
	for some constant $C$, which is independent of $h$, $n$ and $s$.
	
	\par

In order to prove this, it is convenient to let $\maclP _k$ be the
family of all subsets of $\{1,\dots ,k\}$,
$k\ge 1$. For each $P\in\maclP _k$, $a\in s_\infty ^w(\rr {2d})$, we set
\[	
	H(a, P)=
	\begin{cases}
		a                                                                   & \text{ when } P=\emptyset,
		\\
		\scal {T_{j_1}}{D_X} \cdots \scal {T_{j_l}}{D_X} a & \text{ when }
		P=\{j_1 < \cdots < j_l\}, \; l\le k.
	\end{cases}
\]
We now estimate
\[
		\frac{\| 
			(\scal {T_1}{D_X} \cdots \scal {T_k}{D_X} \maclK a)(t,Y,\cdo ) 
		\|_{s_\infty ^w(\rr {2d})}
		}
		{h^k (k!)^s}.
\]
when $a\in E^n_{h,s}$. Since
\begin{multline*}
	(\scal {T_1}{D_X} \cdots \scal {T_k}{D_X} \maclK a)(t,Y,X)
	\\
	= 
	\scal {T_1}{D_X} \cdots \scal {T_k}{D_X}
    \int (K_{Y,Z}(t,\cdo)\wpr a(t,Z,\cdo ))(X)\,dZ
	\\
	=\sum_{P\in\maclP _k}
	\int  (H(K_{Y,Z}(t,\cdo), P)\wpr H(a(t,Z,\cdo ),P^c))(X)  \,dZ,
\end{multline*}
we find
%
\begin{multline}\label{Eq:EnMapCont}
			\frac{\| 
			(\scal {T_1}{D_X} \cdots \scal {T_k}{D_X} \maclK a)(t,Y,\cdo) 
		\|_{s_\infty ^w}
		}
		{h^k (k!)^s}
		\\
		\le
		\sum _{l=0}^k \sum _{|P| =l}
		{k \choose l}^{-s}
		\int
		\frac{\|H(K_{Y,Z}(t,\cdo), P)\|_{s_\infty ^w}}{h^l\,l!^s}
		\cdot \frac{\|H(a(t,Z,\cdo ),P^c)\|_{s_\infty ^w}}{h^{k-l} ((k-l)!)^s}  \,dZ
		\\
		\le
		\sum _{l=0}^k \sum _{|P| =l}\|a\|_{E^{k-l}_{h,s}} 
		{k \choose l}^{-s}
		\int  \frac{\|H(K_{Y,Z}(t,\cdo ), P)\|_{s_\infty ^w}}{h^l\,l!^s}  \,dZ
		\\
		\lesssim
		\|a\|_{E^{n}_{h,s}}\sum _{l=0}^k \sum _{|P| =l}
		{k \choose l}^{-1}
		\int  \frac{\|H(K_{Y,Z}(t,\cdo ), P)\|_{s_\infty ^w}}{h^l\,l!^s}  \,dZ ,
\end{multline}
where the last inequality follows from the fact that $s\ge 1$ and
$\|a\|_{E^{n}_{h,s}}$ increases with $n$.

\par

	We have now to estimate $\|H(K_{Y,Z}(t,\cdo), P)\|_{s_\infty ^w}$, and
	study the different quantities on the right-hand side of \eqref{Eq:KKernelExp1}.
	Since $\omega$ and $\vartheta$ belong to $\mascP _{E,s}^0$, it
	follows that for every $r>0$,
	\begin{multline}\label{Eq:WeightEstimates}
		\frac{\omega(Z)\,\vartheta(Z)^t}
		{\omega(Y)\,\vartheta(Y)^t} = \frac{\omega(Z)}{\omega(Y)}
		\left(\frac{\vartheta(Z)}{\vartheta(Y)}\right)^t
		\lesssim e^{r|Y-Z|^\frac{1}{s}} \,  \left( e^{r|Y-Z|^\frac{1}{s}}
		\right)^t
		\\
		\lesssim  e^{r(1+t)|Y-Z|^\frac{1}{s}}, \quad Y,Z\in\rr{2d}.
	\end{multline}
	For the Weyl product in \eqref{Eq:KKernelExp1} we have
	%
	\begin{align*}
		\phi_Y\wpr f(Y,\cdo )&=\phi(\cdo-Y)\wpr \Big(b(\cdo)+\log\frac{\vartheta(\cdo)}{\vartheta(Y)}\Big)
		\\
		&=\Big(\phi\wpr b(\cdo+Y) \Big)_Y+
		\phi(\cdo-Y)\wpr\Big(\log\frac{\vartheta(\cdo)}{\vartheta(Y)}\Big)
		\\
		&=\Big(\phi\wpr b(\cdo+Y) \Big)_Y+
		\Big(\phi\wpr \log\frac{\vartheta(\cdo+Y)}{\vartheta(Y)}\Big)_Y.
	\end{align*}
	By Theorem \ref{thm:logomega} and Proposition \ref{prop:confest},
	\begin{equation}\label{eq:phi.sharp.f.bound}
		\Big\{ \phi\wpr b(\cdo+Y) \Big\}_{Y\in\rr{2d}}
		\quad
		\text{and}\quad \Big\{\phi\wpr \log\frac{\vartheta(\cdo+Y)}{\vartheta(Y)}\Big\}_{Y\in\rr{2d}}
	\end{equation}
	are uniformly bounded families in $\maclS_s(\rr{2d})$.
	%
	Note that
	\[
		a_2(Z,X)=\psi_Z(X)\quad \Rightarrow \quad  \{a_2(Z,\cdo 
		+Z)\}_{Z\in\rr{2d}}=\{\psi\}_{Z\in\rr{2d}} ,
	\]
	which is evidently a uniformly bounded family in $\maclS_s(\rr{2d})$.
	Combining this last observation with
	the computations on $\phi_Y\wpr f(Y,\cdo )$ above, using Lemma
	\ref{lem:confbis} and \eqref{eq:estsharpbis}, we finally obtain,
	for some $h,r_0>0$,
	\begin{equation}\label{eq:symkest}
		\begin{aligned}
			|D^\alpha _X (\phi_Y\wpr f(Y,\cdo )\wpr \psi_Z)(X)|\lesssim \,\,&h^{|\alpha |}\alpha! ^s e^{-r_0(|X-Y|^{\frac 1s}+|X-Z|^{\frac 1s}+|Y-Z|^{\frac 1s})},
			\\
			&X,Y,Z\in\rr{2d},\alpha\in\nn{2d}.
		\end{aligned}
	\end{equation}

	\par
	
	By Theorem \ref{thm:gamma1}, \eqref{Eq:WeightEstimates} and
	\eqref{eq:symkest} we get for all $P\in\maclP _k$
	$Y,Z\in\rr{2d}$ and some $r_0,h>0$ that
	\[
		\|H(K_{Y,Z}(t,\cdo ), P)\|_{s_\infty ^w}\le Ch^{l}
		l! ^s e^{-r_0|Y-Z|^{\frac 1s}}, \quad l=|P|,
	\]
	where $C$ is independent of $k$. Hence, \eqref{Eq:EnMapCont} gives
	\[
		\|\maclK a(t,Y,\cdo )\|_{s_\infty ^w}
		\lesssim
		\|a\|_{E^{n}_{h,s}},
	\]
	and 
	\begin{multline*}
		\frac{\| 
			\scal {T_1}{D_X} \cdots \scal {T_k}{D_X} \maclK a(t,Y,\cdo ) 
		\|_{s_\infty ^w}
		}
		{h^k (k!)^s}
		\\
		\le
		C_1\|a\|_{E^{n}_{h,s}}\sum _{l=0}^k \sum _{|P| =l}
		{k \choose l}^{-1}
		\int  e^{-r_0|Y-Z|^{\frac 1s}}\, dZ
		\\
		=
		C_2\|a\|_{E^{n}_{h,s}}\sum _{l=0}^k \sum _{|P| =l}
		{k \choose l}^{-1} =C_2(k+1)\|a\|_{E^{n}_{h,s}}
		,\quad 1\le k\le n,
	\end{multline*}	
	as claimed, where $C_1$ and $C_2$ are independent of $k$, $n$ and $h$.
	
	\par
	
	By a completely similar argument, an analogous result can be obtained for
	$\widetilde{\maclK}$. In fact, by similar arguments that lead to
	\eqref{eq:phi.sharp.f.bound} it follows that 
	$$
	\left \{b(\cdo+Z)\wpr\psi \right \}_{Z\in\rr{2d}}
	\quad \text{and}\quad
	\left \{ \log
	\frac{\vartheta(\cdo+Z)}{\vartheta(Z)}\wpr \psi \right \}
	_{Z\in\rr{2d}}
	$$
	are bounded in
	$\maclS_s(\rr{2d})$, given that \eqref{eq:symkest} holds with $f(Z,\cdo )$ in
	place of $f(Y,\cdo )$. This gives the first and third inequalities in
	\eqref{Eq:KtildeKEst}. From these estimates we get
	$$
	\nm {\maclK a}{E_{h,s}^{n+1}}\le C\nm a{E_{h,s}^n}
	\quad \text{and}\quad
	\nm {\widetilde \maclK a}{E_{h,s}^{n+1}}\le C\nm a{E_{h,s}^n},
	$$
	which give the other inequalities in \eqref{eq:symkest}.
	
	\par
	
	We have proven that for any $T>0$, then
	$$
	\nm {\maclK_t}{E_{h,s}^n\to E_{h,s}^n}\le C(n+1),\qquad |t|\le T,
	$$
	where $C$ is independent of $n$. As a
	consequence, since $\omega(Y )^{-1}\phi _Y \wpr a_0$ belongs to
	$E_{h,s}^\infty$, the  equation
		\begin{equation}\label{eq:symeqdiffbis2}
	    	\frac{dc_1}{d t}=\maclK c_1,\qquad
	    	c_1(0)=\omega(Y)^{-1}\phi _Y \wpr a_0
		\end{equation}
		has a unique solution on $\re$ belonging to
		$E_{h,s}^\infty$, in view of Lemma \ref{Lemma:ODE}.
	By Proposition \ref{prop:charact} it follows that $c_1(t,Y,\cdo 
	)\in \Gamma^{(1)}_s(\rr{2d})$, uniformly in $Y$ and for bounded $t$.
	
	\par
	
	In order to prove the uniqueness of the solution $a$ of
	\eqref{eq:symeqdiff}, first we assume the existence and
	by what we have proven above i.e. that $c_1(t,Y,\cdo )$ in
	\eqref{c_1} satisfies \eqref{eq:symeqdiffbis2} which
	implies the uniqueness of the solution of \eqref{eq:symeqdiff},
	since
	\begin{equation}
	    a(t,\cdo )=\int_{\rr{2d}} \psi_Y\wpr \phi_Y\wpr a(t,\cdo )
	    \,dY
	    =
	    \int_{\rr{2d}} \omega(Y)\vartheta(Y)^t\psi_Y\wpr
	    c_1(t,Y,\cdo ) \,dY.
	\end{equation}	
	To prove the existence of a solution of \eqref{eq:symeqdiff}, we
	consider the solution $c_2(t,Y,\cdo )$ of \eqref{c_2} with the
	initial data \eqref{c_2(0)}, and we let
	\begin{equation}
	a(t,\cdo )=\int_{\rr{2d}} \omega(Y)\vartheta(Y)^t
	\psi_Y\wpr c_2(t,Y,\cdo )\,dY.	    
	\end{equation}
	By Propositions \ref{prop:charact} and
	\ref{prop:conv.class}, the family $\{\psi_Y\wpr c_2(t,Y,\cdo
	)\}_{Y\in\rr{2d}}$ belongs to $\maclS _s$ and $a(t,\cdo )$
	belongs to $\Gamma_s^{(w\vartheta^t)}$. Moreover,
	\begin{align*}
	    \frac{d a(t,\cdo )}{dt}&=\int_{\rr{2d}}
	    \omega(Y)\vartheta(Y)^t\log\vartheta(Y)\psi_Y\wpr
	    c_2(t,Y,\cdo )\,dY
	    \\
	    &\qquad+\int_{\rr{2d}}\int_{\rr{2d}}\omega(Y)
	    \vartheta(Y)^t\psi_Y\wpr \tK_{Y,Z}(t,\cdo)\wpr
	    c_2(t,Z,\cdo )\,dY\,dZ
	    \\
	    &=\int_{\rr{2d}} \omega(Z)\vartheta(Z)^t
	    \log\vartheta(Z)\psi_Z\wpr
	    c_2(t,Z,\cdo )\,dZ
	    \\
	    &\qquad+\int_{\rr{2d}}\int_{\rr{2d}}
	    \omega(Z)\,\vartheta(Z)^t\psi_Y\wpr
		\phi_Y\wpr f(Z,\cdo )\wpr \psi_Z\wpr
		c_2(t,Z,\cdo )\,dY\,dZ
		\\
		&=\int_{\rr{2d}} \omega(Z)\vartheta(Z)^t(b+\log\vartheta)\psi_Z
		\wpr c_2(t,Z,\cdo )\,dZ
		\\
		&=(b+\log\vartheta)\wpr a(t,\cdo ),
	\end{align*}
	with the initial data 
	\[
	a(0,\cdo )=\int_{\rr{2d}}\omega(Y)\psi_Y
	\wpr (\omega(Y)^{-1}\phi_Y\wpr a_0)\,dY=a_0,
	\]
	which provide a solution of \eqref{eq:symeqdiff}.
	
	\par
		
	In order to prove the last part we consider the unique
	solution $a(t,\cdo )$ of \eqref{eq:symeqdiff} with the
	initial data $a(0,\cdo )\equiv 1$. If $\omega\equiv 1$,
	then for $u\in\re$ the mappings
	$$
	t\mapsto a(t+u,\cdo )\quad \text{and}\quad  t\mapsto
	a(t,\cdo )\wpr a(u,\cdo )
	$$
	are both solutions of \eqref{eq:symeqdiff} with
	value $a(u,\cdo )$ at $t=0$, and
	\begin{equation}\label{eq:semigroups.property}
	    	a(t+u,\cdo )=a(t,\cdo )\wpr a(u,\cdo ),
	\end{equation}
	by the uniqueness property for the solution of
	\eqref{eq:symeqdiff}. 

    \par

	Using \eqref{eq:semigroups.property} we have for all
	$t \in \re$, $a(t,\cdo )\wpr a(-t,\cdo )=1$. Taking
	the derivative we get
	$$
	0=\frac{d}{dt}(a(t,\cdo )\wpr a(-t,\cdo ))
	=
	(b+\log\vartheta)\wpr a(t,\cdo )\wpr a(-t,\cdo )
	-
	a(t,\cdo )\wpr(b+\log\vartheta)\wpr a(-t,\cdo ).
	$$
	That is $(b+\log\vartheta)=a(t,\cdo )\wpr(b+\log\vartheta)
	\wpr a(-t,\cdo )$, implying the commutation for the sharp
	product of $a(t,\cdo )$ with $(b+\log\vartheta)$, and the
	result follows. 
	\end{proof}

	\par
	
	By similar argument as for the previous result we get the following.
	
	\par
	
	\begin{thm}\label{thm:symbeqdiff2}
	Let $s\ge1$, $\omega,\vartheta\in\mascP_{E,s}(\rr{2d})$ be such that 
	$\omega\in\Gamma^{(\omega)}_s(\rr{2d})$ and
	$\vartheta\in\Gamma^{(\vartheta)}_s(\rr{2d})$,
	and let
	$a_0\in\Gamma^{(\omega)}_{0,s}(\rr{2d})$, 
	$b\in\Gamma^{(1)}_{0,s}(\rr{2d})$.
	Then, there exists a unique smooth map
	$(t,X)\mapsto a(t,X)\in\cc{}$ 
	such that $a(t,\cdo)\in\Gamma_s^{(\omega\,\vartheta^t)}(\rr{2d})$
	for all $t\in\rr{}$, and $a(t,\cdo)$ satisfies \eqref{eq:symeqdiff}.
	
	\par
	
	Moreover, if $\omega\equiv a_0\equiv 1$, then $a(t,X)$ also satisfies
	\eqref{eq:symeqdiffbis}
	and
	\begin{equation*}
		a(t_1,\cdo)\wpr a(t_2,\cdo)=a(t_1+t_2,\cdo), \quad
		a(t,\cdo )\in \Gamma^{(\vartheta^t)}_{0,s}(\rr{2d}),\quad
		t,t_1,t_2\in \mathbf R.
	\end{equation*}
	\end{thm}
%
%
%

\par

\section{Lifting of pseudo-differential operators and Toeplitz
operators on modulation spaces}\label{sec4}

\par

In this section we apply the group properties in Theorems \ref{thm:symbeqdiff} and
\ref{thm:symbeqdiff2} to deduce lifting properties of pseudo-differential operators
on modulation spaces. Thereafter we combine these results by the Wiener
property of certain pseudo-differential operators with symbols in suitable
modulation spaces to get lifting properties for Toeplitz operators with
weights as their symbols.

\par

We begin to apply Theorems \ref{thm:symbeqdiff} and
\ref{thm:symbeqdiff2} to get the following.

\par

\begin{thm}\label{thm:identification}
Let $s\ge 1$, $\mabfp \in (0,\infty ]^{2d}$, 
$A\in \GL (d,\mathbf R)$, $\omega \in \mathscr P_{E,s}^0(\rr {2d})$,
and let $\mascB$ be an invariant
BF-space on $\rr {2d}$, or $\mascB = L^{\mabfp ,E}(\rr {2d})$
for some phase-shift split parallelepiped $E$ in $\rr {2d}$. Then
the following are true:
\begin{enumerate}
\item There exist $a\in \Gamma ^{(\omega )}_{s}(\rr {2d})$ and $b\in \Gamma
^{(1/\omega )}_{s}(\rr {2d})$ such that
\begin{equation}\label{abinverse}
\op _A(a)\circ \op _A(b) =\op _A(b)\circ \op _A(a) =\operatorname
{Id}_{\maclS _s'(\rr d)}.
\end{equation}
Furthermore, $\op _A(a)$ is an isomorphism from $M(\omega
_0,\mathscr B)$ onto $M(\omega _0/\omega ,\mathscr B )$, for
every $\omega _0\in \mathscr P_{E,s}^0(\rr {2d})$.

\vrum

\item Let $a_0\in
\Gamma ^{(\omega )}_{s}(\rr {2d})$ be such that $\op _A(a_0)$ is an 
isomorphism from $M^2_{(\omega _1)}(\rr d)$ to $M^2_{(\omega
_1/\omega )} (\rr d)$ for some $\omega _1\in \mathscr P_{E,s}^0(\rr
{2d})$. Then $\op _A(a_0)$
is an isomorphism from $M(\omega
_2 ,\mathscr B)$ to $M(\omega _2/\omega ,\mathscr B)$, for
every  $\omega _2\in \mathscr P_{E,s}^0(\rr {2d})$. Furthermore,
the inverse of $\op _A(a_0)$ is equal to
$\op _A(b_0)$ for some $b_0\in \Gamma ^{(1/\omega )}_{s}(\rr {2d})$.
\end{enumerate}
\end{thm}

\par

\begin{thm}\label{thm:identification2}
Let $s>1$, $\mabfp \in (0,\infty ]^{2d}$, 
$A\in \GL (d,\mathbf R)$, $\omega \in \mathscr P_{E,s}(\rr {2d})$,
and let $\mascB$ be an invariant
BF-space on $\rr {2d}$, or $\mascB = L^{\mabfp ,E}(\rr {2d})$
for some phase-shift split parallelepiped $E$ in $\rr {2d}$. Then
the following are true:
\begin{enumerate}
\item There exist $a\in \Gamma ^{(\omega )}_{0,s}(\rr {2d})$ and $b\in \Gamma
^{(1/\omega )}_{0,s}(\rr {2d})$ such that
\begin{equation}\label{Eq:abinverse}
\op _A(a)\circ \op _A(b) =\op _A(b)\circ \op _A(a) =\operatorname
{Id}_{\Sigma _s'(\rr d)}.
\end{equation}
Furthermore,
$\op _A(a)$ is an isomorphism from $M(\omega
_0,\mathscr B)$ onto $M(\omega _0/\omega ,\mathscr B )$, for
every $\omega _0\in \mathscr P_{E,s}(\rr {2d})$.

\vrum

\item Let $a_0\in
\Gamma ^{(\omega )}_{0,s}(\rr {2d})$ be such that $\op _A(a_0)$ is an 
isomorphism from $M^2_{(\omega _1)}(\rr d)$ to $M^2_{(\omega
_1/\omega )} (\rr d)$ for some $\omega _1\in \mathscr P_{E,s}(\rr
{2d})$. Then $\op _A(a_0)$
is an isomorphism from $M(\omega
_2 ,\mathscr B)$ to $M(\omega _2/\omega ,\mathscr B)$, for
every  $\omega _2\in \mathscr P_{E,s}(\rr {2d})$. Furthermore,
the inverse of $\op _A(a_0)$ is equal to
$\op _A(b_0)$ for some $b_0\in \Gamma ^{(1/\omega )}_{0,s}(\rr {2d})$.
\end{enumerate}
\end{thm}

\par

We only prove Theorem \ref{thm:identification2}. Theorem \ref{thm:identification}
follows by similar arguments and is left for the reader.

\par

\begin{proof}[Proof of Theorem \ref{thm:identification2}]
The existence of $a\in \Gamma ^{(\omega )}_{0,s}(\rr {2d})$ and
$b\in \Gamma ^{(1/\omega )}_{0,s}(\rr {2d})$ such that
\eqref{Eq:abinverse} holds is guaranteed by Theorem \ref{thm:symbeqdiff2}.
By \cite[Theorems 2.3 and 2.6]{To28} it follows that
\begin{alignat}{5}
&\op _A(a) & \, &:& \, &M(\omega _0,\mascB ) & &\to & &M(\omega _0/\omega ,\mascB )
\label{Eq:PseudoOnModSpace1}
\intertext{and}
&\op _A(b) & \, &:& \, &M(\omega _0/\omega ,\mascB ) & &\to & &M(\omega _0 ,\mascB )
\label{Eq:PseudoOnModSpace2}
\end{alignat}
are continuous. By \eqref{Eq:abinverse} and the fact that $M(\omega _0,\mascB )$
and $M(\omega _0/\omega ,\mascB )$ are contained in $\Sigma _s'(\rr {2d})$, it
follows that \eqref{Eq:PseudoOnModSpace1} and \eqref{Eq:PseudoOnModSpace2}
are homeomorphic, and (1) follows.

\par

(2) It suffices to prove the result in the Weyl case, $A=\frac 12I$, in view of Proposition
\ref{Prop:Gamma(omega)}. By (1), we may find
$$
a_1\in \Gamma ^{(\omega _1)}_{0,s},\quad b_1\in \Gamma ^{(1/\omega _1)}_{0,s},
\quad a_2\in \Gamma ^{(\omega _1/\omega )}_{0,s},\quad
b_2\in \Gamma ^{(\omega /\omega _1)}_{0,s}
$$
satisfying the following properties:
\begin{itemize}
\item $\op ^w(a_j)$ and $\op ^w(b_j)$ are inverses to each others on
$\Sigma _s'(\rr d)$ for $j=1,2$;

\vrum

\item For arbitrary  $\omega _2\in \mascP _{E,s}(\rr {2d})$, the mappings
\begin{equation}\label{4homeomorphisms}
\begin{aligned}
\op ^w (a_1)\, &:\, M^2_{(\omega _2)}\to M^2_{(\omega
_2/\omega _1)},
\\[1ex]
\op ^w (b_1)\, &:\, M^2_{(\omega _2)}\to M^2_{(\omega
_2\omega _1)},
\\[1ex]
\op ^w (a_2)\, &:\, M^2_{(\omega _2)}\to M^2_{(\omega
_2\omega /\omega _1)},
\\[1ex]
\op ^w (b_2)\, &:\, M^2_{(\omega _2)}\to M^2_{(\omega
_2\omega _1/\omega )}
\end{aligned}
\end{equation}
are isomorphisms.
\end{itemize}

\par

In particular, $\op ^w(a_1)$ is an isomorphism from $M^2_{(\omega _1 )}$ to
$L^2$, and $\op ^w(b_1)$ is an isomorphism from $L^2$ to $M^2_{(\omega _1)}$. 

\par

Now set $c=a_2\wpr a \wpr b_1$. Then by \cite[Theorem 4.14]{CaTo},
the symbol $c$ satisfies
$$
c=a_2\wpr a \wpr b_1\in \Gamma ^{(\omega _1/\omega )}_{s}\wpr \Gamma ^{(\omega )}_{s}
\wpr \Gamma ^{(1/\omega _1)}_{s}\subseteq \Gamma ^{(1)}_s. 
$$
Furthermore,  $\op ^w(c)$ is a composition of three isomorphisms and
consequently  $\op ^w(c)$ is boundedly invertible on  $L^2$. 

\par 

By Proposition \ref{Thm:specinv} (2), $\op ^w(c)^{-1}=\op ^w(c_1)$ for some
$c_1\in \Gamma ^{(1)}_{0,s}$. Hence,
by (1) it follows that $\op ^w(c)$
and $\op ^w(c_1)$ are isomorphisms on $M(\omega _2,\mathscr
B)$, for each $\omega _2\in \mascP _{E,s}(\rr {2d})$. Since  $\op ^w(c)$
and  $\op ^w(c_1)$ are bounded on every $M(\omega ,\mathscr{B})$,
the factorization of the identity $ \op ^w(c) \op ^w(c_1) = \operatorname{Id}$
is well-defined on every $M(\omega ,\mathscr{B})$. Consequently, $
\op ^w(c)$ is an isomorphism on $M(\omega , \mathscr{B})$.

\par

Using the inverses of $a_2$ and $b_1$, we now find that 
$$
\op ^w(a)=\op ^w(b_2)\circ \op ^w(c)\circ \op ^w(a_1) 
$$
is a composition of isomorphisms from the domain space $M(\omega
  _2,\mathscr B)$ onto the image space 
$M(\omega _2/\omega ,\mathscr B)$ (factoring through some
intermediate spaces)  for every  $\omega _2\in
\mascP _{E,s}(\rr {2d})$ and every invariant BF-space
$\mathscr B$. This proves the isomorphism assertions for $\op ^w(a)$.

\par

Finally, the inverse of $\op ^w(a)$ is given by
$$
\op ^w(b_1)\circ \op ^w(c_1)\circ \op ^w(a_2).
$$
which is a Weyl operator with symbol in $\Gamma ^{(1/\omega)}_{0,s}$, and
the result follows.
\end{proof}

\par

\section{Mapping properties for Toeplitz operators}\label{sec5}

\par

In this section we study the  isomorphism properties of  Toeplitz
operators between modulation spaces. We  first
state  results for  Toeplitz operators that  are well-defined in
the sense of \eqref{toeplitz} and Propositions \ref{Tpcont1} and
\ref{Tpcont2}. Then we state and prove more general results
for  Toeplitz operators that  are   defined  only  in the framework of
pseudo-differential calculus.

\par

We start with the following result about  Toeplitz operators with smooth symbols.

\par

\begin{thm}\label{locidentification}
Let $s\ge 1$ $\omega ,\omega _0,v\in \mascP _{E,s}^0(\rr {2d})$  be
such that $\omega _0\in \Gamma ^{(\omega _0)}_{s}(\rr {2d})$ and
that $\omega _0$ is $v$-moderate, and let $\mascB$
  be an invariant BF-space on $\rr {2d}$ or $\mascB =
  L^{\mabfp ,E}(\rr {2d})$ for some phase-shift split parallelepiped
  $E$ in $\rr {2d}$.
If $\phi \in M^1_{(v)}(\rr d)$, then $\tp _\phi (\omega _0)$ is an
isomorphism from $M(\omega ,\mascB )$ to $M(\omega /\omega
_0 ,\mascB )$.
\end{thm}

\par

In the next result we relax our restrictions on the weights but impose
more restrictions on $\mascB$.

\par

\begin{thm}\label{locidentification2}
Let $s> 1$, $0\le t\le 1$, $p,q\in [1,\infty]$, and $\omega ,\omega
_0,v_0,v_1\in \mascP _{E,s}(\rr {2d})$ be such that  $\omega _0$ is $v_0$-moderate
and  $\omega $ is
$v_1$-moderate. Set  $v=v_1^tv_0$, $\vartheta =\omega _0^{1/2}$ and let
$\omega _{0,t}$ be the same as in \eqref{omega0t}.
If  $\phi \in M^{1}_{(v )}(\rr d)$ and $\omega _0\in
\splM ^{\infty}_{(1/\omega _{0,t})}(\rr {2d})$,  then $\tp _\phi (\omega _0)$ is an
isomorphism from $M_{(\vartheta
\omega )}^{p,q}(\rr d)$ to $M_{(\omega /\vartheta )}^{p,q}(\rr
d)$.
\end{thm}

\par

Before the proofs of Theorems \ref{locidentification} and
\ref{locidentification2} we state  the following consequence of Theorem
\ref{locidentification2} which  was the original  goal of our
investigations.

\par

\begin{cor}\label{locidentification3}
Let $s\ge 1$, $\omega ,\omega _0,v_1,v_0\in \mascP _{E,s}(\rr {2d})$ and
that $\omega _0$ is
$v_0$-moderate and $\omega $ is $v_1$-moderate. Set  $v=v_1v_0$
and $\vartheta =\omega _0^{1/2}$. If  $\phi \in
M^{1}_{(v)}(\rr d)$, then  $\tp
_\phi (\omega _0)$ is an isomorphism from $M_{(\vartheta \omega
)}^{p,q}(\rr d)$ to $M_{(\omega /\vartheta )}^{p,q}(\rr d)$
simultaneously for all  $p,q\in [1,\infty ]$.
\end{cor}

\par

\begin{proof}
Let $\omega _1\in \mascP _{E,s}(\rr {2d})\cap \Gamma ^{(\omega _1)}_{0,s}(\rr {2d})$
be such that
$C^{-1}\le\omega _1/\omega _0\le C$, for some constant $C$. Hence,
$\omega /\omega _0\in L^\infty \subseteq M^{\infty}$. By Theorem 2.2
in \cite{To9}, it follows that $\omega =\omega _1 \cdot (\omega
/\omega _1)$ belongs to $M^{\infty}_{(\omega _2)}(\rr {2d})$, when
$\omega _2(x,\xi ,\eta ,y)=1/\omega _0(x,\xi )$. The result now
follows by setting $t=1$ and $q_0=1$ in Theorem
\ref{locidentification2}.
\end{proof}

\par

In the proofs of Theorems \ref{locidentification} and
\ref{locidentification2} we consider Toeplitz operators as defined by
an extension of the form~\eqref{toeplitz}.
Later on we  present extensions of
these theorems (cf. Theorems \ref{locidentification}$'$ and
\ref{locidentification2}$'$ below) for those readers who accept to use
pseudo-differential 
calculus to extend the definition of Toeplitz operators. Except for
the interpretation of $\tp _\phi (\omega) $  the proofs of Theorems
\ref{locidentification} and 
\ref{locidentification2} are identical to those of Theorems
\ref{locidentification}$'$ and 
\ref{locidentification2}$'$. 

\par

We need some preparations and start with the following lemma.

\par

\begin{lemma}\label{Abijections}
Let $s\ge 1$,  $\omega ,v\in \mascP _{E,s}(\rr {2d})$ be such  that $\vartheta =
\omega ^{1/2}$ is $v$-moderate.  Assume that $\phi \in M^2_{(v)}$.
Then $\tp _\phi (\omega )$  is an isomorphism from
$M^2_{(\vartheta )}(\rr d)$ onto $M^2_{(1/\vartheta )}(\rr d)$.
\end{lemma}

\par

\begin{proof}
Recall from Remark \ref{Rem:ExtWindows}  that for  $\phi \in M^2_{(v)}(\rr
d)\setminus \{0\}$ the expression $\nm {V_\phi f\cdot \vartheta }{L^2}$
defines an  equivalent norm
on $M^2_{(\vartheta )}$. Thus the occurring  STFTs with respect to
$\phi  $ are well defined. 

\par

Since $ \tp _\phi (\omega ) $ is bounded from  $M^2_{(\vartheta )}$ to
$M^2_{(1/\vartheta )}$ by Proposition~\ref{Tpcont2}, the estimate
\begin{equation}
  \label{eq:23}
\|\tp _\phi (\omega )  f \|_{M^2_{(1/\vartheta )}} \lesssim \|f\|_{M^2_{(\vartheta )}}\,  
\end{equation}
holds for all $f\in M^2_{(\vartheta )}$. 

\par

Next we  observe that
\begin{equation}\label{Aomegaident}
(\tp _\phi (\omega ) f,g)_{L^2(\rr d)} = (\omega V_\phi f,V_\phi
g)_{L^2(\rr {2d})} = (f,g)_{M^{2,\phi }_{(\vartheta )}},
\end{equation}
for $f,g\in M^{2}_{(\vartheta )}(\rr d)$ and $\phi \in M^2_{(v)}(\rr
d)$. The duality of modulation spaces (Proposition~\ref{p1.4}(3)) now yields
the following identity:
\begin{eqnarray}
  \|f\|_{M^{2}_{(\vartheta )}} &\asymp & \sup _{\nm g{M^{2}_{(\vartheta )}}=1}
    |(f,g)_{M^2_{(\vartheta)}}| \notag \\
&\asymp & \sup _{\nm g{M^{2}_{(\vartheta )}}=1} |(\topo f, g )_{L^2}|
\asymp \|\topo f\|_{M^2_{(1/\vartheta )}} \label{eq:24} \, .
\end{eqnarray}
A combination of \eqref{eq:23} and \eqref{eq:24} shows that $\nm
f{M^2_{(\vartheta )}}$ and $\| \topo f
\|_{M^2_{(1/\vartheta )}}$ are equivalent norms on $M^2_{(\vartheta )}$. 

\par

In particular, $\tp _\phi (\omega )$ is one-to-one from 
$M^2_{(\vartheta )}$ to $M^2_{(1/\vartheta )}$ with closed range.
Since $\tp _\phi (\omega )$ is self-adjoint with respect to $L^2$, it
follows by duality that  $\tp _\phi (\omega )$ has dense
range in $M^2_{(1/\vartheta )}$. Consequently, $\topo $ is onto
$M^2_{(1/\vartheta )}$. By Banach's theorem $\topo $ is an isomorphism from
$M^2_{(\vartheta )}$ to $M^2_{(1/\vartheta )}$. 
\end{proof}

\par

We  need a further  generalization  of Proposition
\ref{Tpcont1} to more general classes of symbols and windows. Set 
\begin{equation}\label{Tomega}
\omega _1(X,Y)=\frac{v_0(2Y)^{1/2}v_1(2Y)}{\omega
_0(X+Y)^{1/2}\omega _0(X-Y)^{1/2}}.
\end{equation}

\par

\renewcommand{\rubrik}{Proposition \ref{Tpcont1}$'$}

\begin{tom}
Let $s\ge 1$, $0\le t\le 1$, $p,q,q_0\in [1,\infty]$, and  $\omega ,\omega _0,v_0,v_1\in
\mascP _{E,s}(\rr {2d})$ be such that $v_0$ and $v_1$ are submultiplicative,
$\omega _0$ is
$v_0$-moderate and $\omega $ is $v_1$-moderate. Set 
\begin{equation*}
r_0=2q_0/(2q_0-1),\quad v=v_1^tv_0 \quad \text{and}\quad \vartheta = \omega _0^{1/2} \,,  
\end{equation*}
and  let  $\omega _{0,t}$ and $\omega _1$ be as in \eqref{omega0t} and
\eqref{Tomega}. Then the following are true:

\par

\begin{enumerate}
\item The definition of $(a,\phi )\mapsto \tp _\phi (a)$ from $\Sigma
_s(\rr {2d})\times \Sigma _s(\rr d)$ to $\mathcal L(\Sigma _s(\rr
d),\Sigma _s'(\rr d))$ extends uniquely to a continuous map from
$\splM ^{\infty ,q_0}_{(1/\omega _{0,t})}(\rr {2d})\times
M^{r_0}_{(v)}(\rr d)$ to $\mathcal L(\Sigma _s(\rr d),\Sigma _s
'(\rr d))$.

\vrum

\item If $\phi \in M^{r_0}_{(v)}(\rr d)$ and $a\in \splM
^{\infty ,q_0}_{(1/\omega _{0,t})}(\rr {2d})$, then $\tp _\phi (a)
=\op ^w(a_0)$ for some $a_0\in \splM ^{\infty ,1}_{(\omega _1)}(\rr
{2d})$, and $\tp _\phi (a)$ extends uniquely to a continuous map
from $M_{(\vartheta \omega )}^{p,q}(\rr d)$ to $M_{(\omega
/\vartheta )}^{p,q}(\rr d)$.
\end{enumerate}
\end{tom}

\par

For the proof we need the following result, which follows from
\cite[Proposition 2.1]{To8} and its proof.

\par

\begin{lemma}\label{Prop2.1inTo8}
Assume that $s\ge 1$, $q_0,r_0\in [1,\infty ]$ satisfy $r_0=2q_0/(2q_0-1)$. Also
assume that $v \in \mascP _{E,s}(\rr {2d})$ is submultiplicative, and
that $\kappa ,\kappa _0\in \mascP _{E,s}(\rr {2d}\oplus \rr {2d})$
satisfy 
\begin{equation}\label{kappacond1}
\kappa _0(X_1+X_2,Y)\le C\kappa (X_1,Y)\, v(Y+X_2)v(Y-X_2)\quad X_1,X_2,Y\in \rr {2d},
\end{equation}
for some constant $C>0$. Then the  map $(a,\phi )\mapsto \operatorname 
{Tp}_{\phi} (a)$ from $\Sigma _s(\rr {2d})\times \Sigma _s(\rr {d})$
to $\mathcal L(\Sigma _s(\rr {d}),\Sigma _s'(\rr {d}))$ extends
uniquely to a continuous mapping from $\splM ^{\infty
  ,q_0}_{(\omega )}(\rr {2d})\times M^{r_0}_{(v)}(\rr d)$  to $\mathcal L(\Sigma _s(\rr
{d}),\Sigma _s'(\rr {d}))$. Furthermore, if $\phi \in
M^{r_0}_{(v)}(\rr d)$ and $a\in \splM ^{\infty ,q_0}_{(\kappa 
)}(\rr {2d})$, then $\operatorname {Tp}_{\phi} (a) = \op ^w(b)$ for some
$b\in \splM ^{\infty ,1}_{(\kappa _0)}$.
\end{lemma}

\par

\begin{proof}[Proof of Proposition \ref{Tpcont1}{\,}${}^\prime$]
We show that the conditions on the involved parameters and weight
functions satisfy  the conditions of  Lemma \ref{Prop2.1inTo8}.

\par

First we observe that
$$
v_j(2Y)\le Cv_j(Y+X_2)v_j(Y-X_2),\quad j=0,1
$$
for some constant $C$ which is independent of $X_2,Y\in \rr {2d}$,
because $v_0$ and $v_1$ are submultiplicative. Refering back to \eqref{Tomega} this gives
\begin{multline*}
\omega _1(X_1+X_2,Y) = \frac {v_0(2Y)^{1/2}v_1(2Y)}{\omega
_0(X_1+X_2+Y)^{1/2}\omega _0(X_1+X_2-Y)^{1/2}}
\\[1ex]
\le C_1 \frac
{v_0(2Y)^{1/2}v_1(2Y)v_0(X_2+Y)^{1/2}v_0(X_2-Y)^{1/2}}{\omega
_0(X_1)}
\\[1ex]
= C_1 v_1(2Y)^{1-t}\frac
{v_0(2Y)^{1/2}v_1(2Y)^t v_0(X_2+Y)^{1/2}v_0(X_2-Y)^{1/2}}{\omega
_0(X_1)}
\\[1ex]
\le C_2v_1(2Y)^{1-t} \frac {v_1(X_2+Y)^t v_1(X_2-Y)^t v_0(X_2+Y)v_0(X_2-Y)}{\omega
_0(X_1)}.
\end{multline*}
Hence
\begin{equation}\label{omegacond1}
\omega _1(X_1+X_2,Y) \le C\frac
{v_1(2Y)^{1-t}v(X_2+Y)v(X_2-Y)}{\omega _0(X_1)}.
\end{equation}
By letting $\kappa _0 = \omega _1$ and $\kappa =1/\omega _{0,t}$, it
follows that \eqref{omegacond1} agrees with \eqref{kappacond1}. The
result now follows from Lemma \ref{Prop2.1inTo8}. 
\end{proof}

\par

In the remaining part  of the paper we interpret  $\tp
_\phi (a)$  as the extension  of a  Toeplitz
operator provided by Proposition \ref{Tpcont1}$'$. (See also Remark
\ref{extensionremark} below for more comments.)

\par

Proposition \ref{Tpcont1}$'$ can be applied to Toeplitz operators
with smooth weights as symbols.

\par

\begin{prop}\label{Aomegaproperties}
Assume that $s\ge 1$, $\omega _0\in \mascP _{E,s}^0(\rr {2d})$ be such that
$\omega _0\in \Gamma ^{(\omega _0)}_{s}(\rr {2d})$,
that $v \in \mascP _{E,s}^0(\rr {2d})$ is submultiplicative,  and that
$\omega _0 ^{1/2}$ is $v$-moderate. If $\phi \in M^2_{(v)}(\rr d)$, then   $\tp _\phi
(\omega _0) =\op ^w(b)$ for some $b\in \Gamma ^{(\omega _0)}_s(\rr {2d})$.
\end{prop}

\par

\begin{proof}
By Propositions~\ref{Prop:Gamma(omega)} and \ref{Prop:Gamma(omega)B} with $t=0$
we have $\omega _0 \in \splM
^{\infty ,1}_{(1/\omega _{0,r_0})}(\rr {2d})$ for some  $r_0\ge 0$, where
$\omega _{0,r_0}(X,Y)=\omega _0(X)e ^{-r_0|Y|^{\frac 1s}}$. Furthermore, by
letting $v_1(Y)=e^{r_0|Y|^{\frac 1s}}$, and
$\omega _1$ in \eqref{Tomega} we have
$$
\omega _1(X,Y)\gtrsim \frac{e^{r_0|2Y|^{\frac 1s}}v(2Y)^{1/2}}{\omega
_0(X+Y)^{1/2}\omega _0(X-Y)^{1/2}} \gtrsim \frac{e^{r_0|Y|^{\frac 1s}}}{\omega _0(X)}.
$$
Proposition \ref{Tpcont1}$'$ implies that existence of  some $b \in
\splM ^{\infty ,1}_{(1/\omega _{0,r_0})}(\rr {2d}) \subseteq
\Gamma ^{(\omega _0)}_{s}(\rr {2d})$.
\end{proof}

%
%
%
%
%
%
%
%

\par

The following generalization of Theorem \ref{originalgoal} is
an immediate consequence of Theorem \ref{thm:identification}, Lemma \ref{Abijections} and
Proposition \ref{Aomegaproperties}, since it follows by straightforward computations that
$\maclS _s\subseteq M^2_{(v)}$ when $v$ satisfies the hypothesis in Proposition
\ref{Aomegaproperties}.


\par

\renewcommand{\rubrik}{Theorem \ref{originalgoal}$'$}

\begin{tom}
  Let $s\ge 1$, $\omega ,\omega _0\in \mascP _{E,s}^0(\rr {2d})$,
  $\mabfp \in (0,\infty ]^{2d}$, $\mascB$
  be an invariant BF-space on $\rr {2d}$ or $\mascB =
  L^{\mabfp ,E}(\rr {2d})$ for some phase-shift split parallelepiped
  $E$ in $\rr {2d}$, and let $\phi \in
  \maclS _s(\rr{d})$. Then the Toeplitz operator $\tp _\phi (\omega
  _0)$ is an isomorphism from $M(\omega ,\mascB )$ onto
  $M(\omega /\omega _0,\mascB )$.
\end{tom}

\medspace

\begin{rem}\label{extensionremark}
As remarked and stated before, there are different ways to extend the
definition of  a Toeplitz operator $\tp _\phi (a)$  (from  $\phi \in \Sigma _1
(\rr d)$ and $a\in \Sigma _1(\rr {2d})$) to more general classes of
symbols and windows. For example, Propositions
\ref{Tpcont1} and \ref{Tpcont2} are  based on the ``classical''
definition \eqref{toeplitz} of such operators and  a straightforward
extension  of the $L^2$-form on $\Sigma _1$.
Proposition~\ref{Aomegaproperties} interprets  $\topo $ as a
\psdo .  Let us  emphasize
that in this context the bilinear form \eqref{toeplitz} may not be well
defined,  even  when $\phi \in M^2_{(v)}(\rr d)$ and  $\omega \in \mascP
_{E,s}^0(\rr {2d})$ satisfies $\omega \in \Gamma ^{(\omega )}_{s}(\rr {2d})$.
(See also \cite[Remark 3.8]{GrTo1}.)

\par


\par

In Theorems \ref{locidentification}$'$ and \ref{locidentification2}$'$
below, we extend the definition of Toeplitz operators within the
framework of pseudo-differential calculus and  we 
interpret  Toeplitz operators as pseudo-differential operators. With
this understanding, the stated mapping properties are well-defined.

\par

The reader who is not interested in full generality  or does not accept Toeplitz
operators that  are not defined directly by an extension of
 \eqref{toeplitz} may only consider the case when the
windows belong to $M^1_{(v)}$ and stay with Theorems~\ref{locidentification} 
and \ref{locidentification2}.  
For the more general window classes in Theorems
\ref{locidentification}$'$ and \ref{locidentification2}$'$ below,
however,  one
should then interpret the involved operators as ``pseudo-differential
operators that extend Toeplitz operators''.
\end{rem}

\par

The following generalization of Theorem \ref{locidentification} is
an immediate consequence of Theorem \ref{thm:identification}, Lemma
\ref{Abijections} and Proposition
\ref{Aomegaproperties}.

\par

\renewcommand{\rubrik}{Theorem \ref{locidentification}$'$}

\begin{tom}
Let $s\ge 1$, $\omega ,v,v_0\in \mascP _{E,s}^0(\rr {2d})$ be
such that $\omega _0\in \Gamma ^{(\omega _0)}_{s}(\rr {2d})$ and
that $\omega _0$ is $v$-moderate, and let $\mascB$
  be an invariant BF-space on $\rr {2d}$ or $\mascB =
  L^{\mabfp ,E}(\rr {2d})$ for some phase-shift split parallelepiped
  $E$ in $\rr {2d}$.
If $\phi \in M^2_{(v)}(\rr d)$, then $\tp _\phi (\omega _0)$ is an
isomorphism from $M(\omega ,\mathscr B)$ to $M(\omega /\omega
_0,\mathscr B)$.
\end{tom}

\par

Theorem \ref{locidentification}$'$ holds only  for smooth weight functions. In order to   relax
the conditions on the weight function
$\omega _0$,  we use the  Wiener algebra property of  $\splM ^{\infty ,1}_{(v)}$ instead of
$\Gamma ^{(1)}_{s}$. On the other hand, we have  to restrict our results
to modulation spaces of the form $M^{p,q}_{(\omega )}$ instead of
$M(\omega ,\mascB )$.

\par

\renewcommand{\rubrik}{Theorem \ref{locidentification2}$'$}

\begin{tom}
Let $s> 1$, $0\le t\le 1$, $p,q,q_0\in [1,\infty]$ and $\omega ,\omega
_0,v_0,v_1\in \mascP _{E,s}(\rr {2d})$ be such that $\omega _0$ is
$v_0$-moderate and  $\omega $ is
$v_1$-moderate. Set  $r_0=2q_0/(2q_0-1)$,  $v=v_1^tv_0$, $\vartheta =\omega _0^{1/2}$ and
let $\omega _{0,t}$ be the same as in \eqref{omega0t}.
If  $\phi \in M^{r_0}_{(v )}(\rr d)$ and $\omega _0\in \splM
^{\infty ,q_0}_{(1/\omega _{0,t})}$,  then $\tp _\phi (\omega _0)$ is
an isomorphism from $M_{(\vartheta
\omega )}^{p,q}(\rr d)$ to $M_{(\omega /\vartheta )}^{p,q}(\rr d)$.
\end{tom}

\par

\begin{proof}
First we note that the Toeplitz
operator $\tp _\phi (\omega _0)$ is an isomorphism from
$M^2_{(\vartheta )}$ to $M^2_{(1/\vartheta )}$ in view of Lemma
\ref{Abijections}. With $\omega
_1$ defined in~\eqref{Tomega}, Proposition \ref{Tpcont1}$'$ implies
that there  exist  $b\in
\splM ^{\infty ,1}_{(\omega 
_1)}$ and $c\in \maclS _s'(\rr {2d})$ such that
$$
\tp _\phi (\omega _0) = \op ^w(b)\quad \text{and}\quad \tp _\phi (\omega
_0)^{-1}=\op ^w(c) \, .
$$
Let $\omega _2$ be the ``dual'' weight defined as 
\begin{equation}
  \label{eq:dmod1'}
  \omega _2(X,Y) =\vartheta (X-Y)\vartheta (X+Y) v_1(2Y).
\end{equation}
We shall prove that $c\in \splM ^{\infty ,1}_{(\omega _2)}(\rr
{2d})$. 
Let us assume for now that we have already proved the existence of
such a symbol $c$. Then  we may proceed as follows. 

\par

After checking~\eqref{e5.9}, we can apply  Proposition \ref{Prop:pseudomod} and
find that each of  the mappings
\begin{equation}\label{op(b)(d)}
\op ^w(b)\, :\, M^{p,q}_{(\omega \vartheta )} \to M^{p,q}_{(\omega /\vartheta )}\quad \text{and}\quad
\op ^w(c)\, :\, M^{p,q}_{(\omega /\vartheta )} \to M^{p,q}_{(\omega \vartheta )}
\end{equation}
is  well-defined and continuous.

\par

In order to apply Proposition~\ref{Prop:Weylprodmod}, we next check
condition~\eqref{Eq:weightprodmod} for the weights $\omega _1$,  $\omega _2$, and
$$
\omega _3(X,Y) = \frac {\vartheta (X+Y)}{\vartheta (X-Y)}.
$$
In fact we have
\begin{eqnarray*}
\lefteqn{\omega _1(X-Y+Z,Z)\omega _2(X+Z,Y-Z)}\\[1ex]
&=& \Big (\frac {v_0(2Z)^{1/2}v_1(2Z)}{\vartheta (X-Y+2Z)\vartheta
  (X-Y)}\Big )\cdot \big (\vartheta (X-Y+2Z)\vartheta (X+Y)
v_1(2(Y-Z)) \big ) 
\\[1ex]
&= &  \frac {v_0(2Z)^{1/2}v_1(2Z)v_1(2(Y-Z))\,  \vartheta (X+Y)}{\vartheta (X-Y)}
\\[1ex]
&\gtrsim &  \frac {\vartheta (X+Y)}{\vartheta (X-Y)} =\omega _3(X,Y)\, .
\end{eqnarray*}
Therefore Proposition \ref{Prop:Weylprodmod}  shows that the Weyl symbol
of $\op ^w(b)\circ \op ^w(c)$ belongs to $\splM ^{\infty
  ,1}_{(\omega _3)}(\rr {2d})$, or equivalently, $b\wpr c\in \splM
^{\infty ,1}_{(\omega _3)}$. Since $\op ^w(b)$ is an isomorphism from
$M^2_{(\vartheta )}$ to $M^2_{(1/\vartheta )}$ with inverse $\op
^w(c)$, it follows that $b\wpr c =1$ and that the constant symbol $1$  belongs to
$\splM ^{\infty ,1}_{(\omega _3)}$. By similar arguments it follows that
$c\wpr b=1$. Therefore the identity operator $\mathrm{Id}= \op ^w (b)
\circ \op ^w(c) $ on $M^{p,q}_{(\omega \vartheta )}$ factors through
$M^{p,q}_{(\omega /\vartheta )}$,  and thus $\op ^w(b)= \tp _\phi (\omega
_0)$ is an isomorphism
from $M^{p,q}_{(\omega \vartheta )}$ onto $M^{p,q}_{(\omega /\vartheta )}$
with inverse $\op ^w(c)$. This proves the assertion. 

\par

It remains to prove that $c\in \splM ^{\infty ,1}_{(\omega
 _2)}(\rr {2d})$. Using once again the basic result in Section \ref{sec3}
(cf. Theorem \ref{thm:symbeqdiff2}),
we choose $a\in \Gamma ^{(1/\vartheta )}_{0,s}(\rr {2d})$ and
$a_0\in 
\Gamma ^{(\vartheta )}_{0,s}(\rr {2d})$ such that the map
$$
\op ^w(a)\, :\, L^2(\rr d)\to M^2_{(\vartheta )}(\rr d)
$$
is an isomorphism with inverse $\op ^w(a_0)$. By Propositions
\ref{Prop:Gamma(omega)} and \ref{Prop:Gamma(omega)B}, $\op ^w(a)$ is
also bijective from $M^2_{(1/\vartheta )}(\rr d)$ to $L^2(\rr
d)$. Furthermore, by Theorem~\ref{thm:identification2} it follows that $a\in \splM
^{\infty ,1}_{(\vartheta _r)}$ when $r\ge 0$,  
where
$$
\vartheta _r(X,Y)=\vartheta (X)e^{r|Y|^{\frac 1s}}.
$$
Let $b_0=a\wpr b \wpr a$. From  Proposition~\ref{Prop:CorWeyl} we know that 
\begin{equation}\label{b0mod}
b_0\in \splM ^{\infty ,1}_{(v_2)}(\rr {2d}),\quad \text{where}\quad v_2(X,Y)=v_1(2Y)
\end{equation}
is submultiplicative and depends on $Y$ only. 
Since $\op ^w(b)$ is
bijective from $M^2_{(\vartheta )}$ to $M^2_{(1/\vartheta )}$ by
Lemma \ref{Abijections} (2),  $\op ^w(b_0)$ is  bijective and
continuous  on  $L^2$. 

\medspace

Since $v_2$ is submultiplicative and in $\mascP _{E,s}(\rr{2d})$,
$\splM ^{\infty ,1}_{(v_2)}$ is a 
Wiener algebra by Proposition \ref{Thm:specinv}.  Therefore, the Weyl symbol $c_0$
of the inverse to the bijective operator $\op ^w(b_0)$ on $L^2$
belongs to $\splM ^{\infty  ,1}_{(v_2)}(\rr {2d})$. 

\par

Since
$$
\op ^w (c_0) = \op ^w (b_0)\inv = \op ^w(a)\inv \op ^w(b)\inv
\op ^w(a)\inv ,
$$
we find
$$
\op ^w (a_0) = \op ^w (b)\inv = \op ^w(a) \op ^w(c_0)\op ^w(a),
$$
or equivalently, 
\begin{equation}\label{d0dcond}
a_0 = a\wpr c_0\wpr a ,\quad \text{where} \quad a\in \Gamma ^{(1/\vartheta )}_{0,s}
\quad \text{and}\ c_0\in \splM ^{\infty ,1}_{(v_2)}\, .
\end{equation}
The definitions of the weights are chosen such that  Proposition 
\ref{Prop:CorWeyl} implies that $a_0 \in  \splM ^{\infty ,1}_{(\omega _2)}$,
and the result follows.
\end{proof}

\par

\section{Examples on bijective pseudo-differential
operators on modulation spaces}\label{sec6}

\par

In this section we construct explicit isomorphisms between modulation
spaces with different weights. Applying the  results of the previous
sections, these may be in the form of  pseudo-differential operators
or of Toeplitz operators.

\par






\par

\begin{prop}\label{isompseudos3}
Let $s\ge 1$, $\omega _0,\omega \in \mascP _{E,s}^0(\rr
{2d})$, and let $\mascB$
  be an invariant BF-space on $\rr {2d}$ or $\mascB =
  L^{\mabfp ,E}(\rr {2d})$ for some phase-shift split parallelepiped
  $E$ in $\rr {2d}$. For $ \lambda = (\lambda _1$, $\lambda _2) \in \rr{2}_+ $ 
let $\Phi _{\lambda}  $ be the Gaussian 
$$
\Phi _\lambda  (x,\xi ) = Ce^{-(\lambda _1|x|^2+\lambda _2|\xi |^2)}\, .
$$

\par

\begin{enumerate}
\item $\omega _0 \ast \Phi _\lambda $ belongs to $\mascP
_{E,s}^0(\rr {2d})\cap \Gamma ^{(\omega _0)}_{0,1}$ for all
$\lambda \in \rr{2}_+ $ and
$$
\omega _0  \ast \Phi _\lambda \asymp \omega _0 .
$$

\vrum

\item If $\lambda _1\cdot \lambda _2 <1$, then there exists $\nu \in
  \rr{2}_+ $ and a Gauss function $\phi$ on $\rr{d}$ such that
  $\op ^w(\omega _0 \ast \Phi _\lambda ) =
  \mathrm{Tp}_\phi (\omega _0 \ast \Phi _\nu)$ is bijective from
  $M(\omega ,\mathscr B)$ to $M(\omega /\omega _0,\mathscr
  B)$ for all $\omega \in \mascP _{E,s}(\rr {2d})$.

\vrum

\item If $\lambda _1\cdot \lambda _2 \le 1$ and in addition $\omega _0
  \in \Gamma ^{(\omega _0)}_{s}(\rr {2d})$, then $\op ^w(\omega _0 \ast \Phi
  _\lambda ) = \mathrm{Tp}_\phi (\omega _0)$ is bijective from
  $M(\omega ,\mathscr B)$ to $M(\omega /\omega _0,\mathscr
  B)$ for all $\omega \in \mascP_{E,s}(\rr {2d})$.
\end{enumerate}
\end{prop}

\par

\begin{proof}
The assertion (1) follows easily from the definitions.

\par

(2) Choose $\mu _j>\lambda _j$ such that $\mu _1\cdot \mu _2=1$. Then
the Gaussian $\Phi _\mu$ is a multiple of a Wigner
  distribution, precisely $\Phi _\mu = c W(\phi , \phi )$
  with $\phi (x)=e^{-\mu _1|x|^2/2}$.   By
the semigroup property of Gaussian functions (cf. e.g., \cite{Fo,Gc2}) there
exists another Gaussian, namely $\Phi _\nu$,  such that $\Phi _\lambda
= \Phi _\mu \ast \Phi _\nu $. Using \eqref{toeplweyl}, this factorization implies that the
 Weyl operator with symbol $\omega _0 \ast \Phi _\lambda$  is in fact a Toeplitz operator, namely
 \begin{eqnarray*}
\mathrm{Op}^w (\omega _0 \ast \Phi _\lambda) &=& \mathrm{Op}^w (\omega
_0 \ast \Phi _\nu \ast \Phi _\mu) \\
& = &  \mathrm{Op}^w (\omega _0 \ast \Phi _\nu \ast cW(\phi , \phi )) \\
& = & c(2\pi )^{d/2}\mathrm{Tp}_\phi (\omega _0 \ast \Phi _\nu ).
\end{eqnarray*}

\par

By (1)  $\omega _0 \ast \Phi _\nu \in \mathscr P_{E,s}^0(\rr {2d})
\cap \Gamma ^{(\omega _0)}_{0,1}(\rr {2d})$ is
equivalent to $\omega _0$. Hence Theorem~\ref{locidentification}$'$ shows that 
$\mathrm{Op}^w (\omega _0 \ast \Phi _\lambda)$
is bijective from $M(\omega ,\mathscr B)$ to $M(\omega /\omega _0,\mathscr B)$.
This proves (2).

\par

(3) follows from (2) in the case $\lambda _1\cdot \lambda
_2<1$. If $\lambda _1\cdot \lambda _2 =1$, then as above $\Phi
_\lambda = c W(\phi , \phi )$ for $\phi (x) = e^{-\lambda_1 |x|^2/2}$ and
thus
$$
\op ^w(\omega _{0} \ast \Phi _\lambda) = \mathrm{Tp}_\phi  ^w(\omega _{0})
$$
is bijective from $M(\omega ,\mathscr B)$ to $M(\omega /\omega
  _0,\mathscr B)$, since $\omega _0 \in \mascP _{E,s}^0(\rr{2d})\cap
\Gamma ^{(\omega _0)}_{s}(\rr {2d})$.
\end{proof}

\par

%
%

\appendix

\section{Proof of Lemma \ref{lem:estfaadb}}

\begin{lemma}\label{Lemma:SetEst}
Let $\alpha = (\alpha _1,\dots ,\alpha _d)\in \nn d$. Then the number of
elements in the set
\begin{equation}\label{Eq:SetToEst}
\Omega _{k,\alpha }\equiv
\sets {(\beta _1,\dots ,\beta _k)\in \nn {kd}}{\beta _1+\cdots +\beta _k =\alpha}
\end{equation}
is equal to
$$
\prod _{j=1}^d {{\alpha _j+k} \choose {k}}.
$$
\end{lemma}

\par

For the proof we recall the formula
\begin{equation}\label{Eq:BinomSumFormula}
\sum _{j=0}^k {{n+j} \choose {j}} = {{n+k+1} \choose {k}},
\end{equation}
          which follows by a standard induction argument.

\par

\begin{proof}
Let $N$ be the number of elements in the set \eqref{Eq:SetToEst},
which is the searched number, and let $N_j$ be the number of elements
of the set
$$
\sets {(\beta _1^0,\dots ,\beta _k^0)\in \nn {k}}{\beta _1^0+\cdots +\beta _k^0
=\alpha _j},\qquad j=1,\dots ,d
$$
By straightforward computations it follows that $N=N_1\cdots N_d$, and
it suffices to prove the result in the case $d=1$, and then $\alpha =\alpha _1$.

\par

In order to prove the result for $d=1$, let $\gamma \in \mathbf N$,
$$
S_1 (\gamma ) = \sum _{\beta =0}^\gamma 1 = \gamma +1,
$$
and define inductively
$$
S_{j+1}(\gamma ) = \sum _{\beta =0}^\gamma S_j(\beta ),\qquad
j=1,2,\dots .
$$
By straightforward computations it follows that $N=N_1=S_k(\alpha )$.
We claim
\begin{equation}\label{Eq:SjFormula}
S_j(\gamma ) = {{\gamma +j}\choose {j}},\qquad
j=1,2,\dots .
\end{equation}

\par

In fact, \eqref{Eq:SjFormula} is obviously true for $j=1$. Assume that
\eqref{Eq:SjFormula} holds for $j=n$, and consider $S_{n+1}(\gamma )$. Then
\eqref{Eq:BinomSumFormula} gives
\begin{multline*}
S_{n+1}(\gamma ) = \sum _{\beta =0}^\gamma S_n(\beta)
= \sum _{\beta =0}^\gamma  {{\beta +n}\choose {n}}
\\[1ex]
= \sum _{\beta =0}^\gamma  {{\beta +n}\choose {\beta}}
= {{\gamma +n+1}\choose {\gamma}}
= {{\gamma +n+1}\choose {n+1}},
\end{multline*}
which gives \eqref{Eq:SjFormula} when $j=n+1$. This proves
\eqref{Eq:SjFormula}, and the result follows.
\end{proof}

\par

\begin{lemma}\label{Lemma:SumFactorFractionEst}
Let $\alpha \ge 1$ be an integer, $s_0\in (0,1]$, and let $\Omega _{k,\alpha}$
be the same as in \eqref{Eq:SetToEst}. Then there is a constant $C$ which is
independent of $\alpha$ such that
$$
\sum _{k=1}^\alpha \frac 1k\sum _{\beta \in \Omega _{k,\alpha}}
(\beta !)^{s_0-1} \le 16^\alpha . 
$$
\end{lemma}

\par

%

\par

\begin{proof}
By Lemma \ref{Lemma:SetEst} and the fact that $s_0-1<0$ we get
\begin{equation*}
\sum _{k=1}^\alpha \frac 1k \left (\sum _{\beta \in \Omega _{k,\alpha}}
\beta !^{s_0-1}\right )
\le
\sum _{k=1}^\alpha \left (\sum _{\beta \in \Omega _{k,\alpha}}
1\right )
=
\sum _{k=1}^\alpha  {{\alpha +k}\choose {k}}
\\[1ex]
\le
(1+\alpha ) 2^{2\alpha} \le 16^\alpha . \qedhere
\end{equation*}
\end{proof}

\end{document}